\documentclass[english,11pt,a4paper,reqno]{smfart}

\usepackage{smfthm}
\usepackage{math-book-dsimon}
\usepackage{hyperref}
\usepackage{float}

\usetikzlibrary{patterns}
\usetikzlibrary{matrix}
\usepackage{mathtools}

\DeclareMathOperator{\Guill}{Guill}

\DeclareMathOperator{\PatternShapes}{PS}

\DeclareMathOperator{\Mat}{Mat}

\DeclareMathOperator{\boundaryweights}{bw}

\newcommand{\Schur}{S}

\newcommand{\South}{\mathsf{S}}
\newcommand{\North}{\mathsf{N}}
\newcommand{\West}{\mathsf{W}}
\newcommand{\East}{\mathsf{E}}

\newcommand{\ee}{\mathbf{e}}

\newcommand{\toep}{\mathsf{T}}
\newcommand{\hank}{\mathsf{H}}
\newcommand{\Fourier}{\mathcal{F}}
\newcommand{\fold}{\mathsf{F}}

\newcommand{\CPsiInv}{\mathbf{C}}

\newcommand{\MarkovWeight}[1]{\mathsf{#1}}

\newcommand{\patterntype}[1]{\mathrm{\mathbf{#1}}}

\newtheorem{assumptionalt}{Assumption}
\newenvironment{assumptionp}[1]{
	\renewcommand\theassumptionalt{#1}
	\assumptionalt
}{\endassumptionalt}

\tikzstyle{guillpart} = [scale=0.5,baseline={(current bounding box.center)}]
\tikzstyle{guillfill}=[lightgray,opacity=0.5]
\tikzstyle{guillsep} = [ultra thick]
\tikzstyle{centerline} = [baseline={(current bounding box.center)}]

\title[Operads and boundary conditions for Gaussian Markov fields]{Operadic structure of boundary conditions for two-dimensional Markovian Gaussian random fields}

\author{Emilien \textsc{Bodiot} et Damien \textsc{Simon}}
\address{Laboratoire de Probabilité, Statistique et Modélisation (LPSM, UMR8001), Sorbonne Université, 4 place Jussieu, 75005 Paris}
\email{bodiot@lpsm.paris, damien.simon@lpsm.paris}
\date{2024/07/04}

\begin{abstract}
The theory of Markov processes on the square lattice has been given recently a new algebraic description based on operads. This approach allows for a local description of invariant boundary conditions and provide infinite-volume Gibbs measures out of solutions of local constraints. This theory comes with new algebraic definitions which have not been constructed on any non trivial model yet: the present paper realizes all the operadic computations in the case of Gaussian Markov fields on the square lattice. In particular, we write down and solve all the explicit equations satisfied by the boundary eigen-elements in the operadic sense. As much as possible, we insist on the method and not on the Gaussian context so the approach can be adapted to other models. Then, we relate the boundary eigen-elements to other traditional methods such as the transfer matrix through new tools such as folding and square associativity; we also show that these new elements contain more information than the traditional ones.
\end{abstract}

\begin{document}
	\maketitle
	
	\tableofcontents

\section{Introduction}
\subsection{General presentation}
\paragraph*{Motivations and background}
In the sixth chapter of \cite{friedlivelenik} about the construction infinite-volume Gibbs measures in statistical mechanics, Friedli and Velenik explain page~251 "why not simply use Kolmogorov's extension theorem" for these constructions. Their explanation highlights some difficulties but does not completely discard such an approach: in addition to the Hamiltonian of the system, one must provide additional compatible boundary conditions related to the infinite-volume measures. This task is precisely the purpose of the present paper in the framework of Gaussian Markovian processes on the square lattice, for which computations can be made exactly and can serve as a guideline for the study of other more relevant models of statistical mechanics. This task of providing compatible sets of boundary conditions is made possible by using the abstract framework of \cite{Simon} and the present paper is the first non-trivial example of these abstract constructions and is thus a proof of concept that the operadic approach \cite{Simon} provides concrete sets of equations that can be solved exactly in concrete models. 

In general, without Kolmogorov's extension theorem, infinite volume Gibbs measures are obtained in an \emph{analytical} way in the following manner: firstly, finite size systems are prepared using suitable boundary conditions guessed to approximate the asymptotic ones, secondly analytical bounds are needed to ensure the convergence of observables to some limits, thirdly the use of Riesz representation theorem is used to obtain the existence of the infinite volume measures. Although it works beautifully for many simple models for which physical intuition is already present to prepare the finite systems, the convergence of observables can also be painful and may require technical inequalities. Kolmogorov's extension theorem, on the other hand, do not require any analytical bounds but only the identification of consistent contraints between boundary conditions: guessing directly these consistent constraints is obviously even much harder than guessing approximate ones as in the analytical way but, as seen below, this is essentially due to a lack of an algebraic understanding of the boundary conditions and their consistence.

\paragraph*{Content of the paper}

After an introduction on the Gaussian Markov processes and the associated operadic theory, we show in section~\ref{sec:fixedpoints} that the approach introduced in \cite{Simon} produces, in the Gaussian setting on the discrete lattice through theorems~\ref{theo:westfixedpoint} and \ref{theo:corner:charactblocks}, sets of local equations on so-called \emph{boundary eigen-elements} on half-strips and corners. Solving these equations produces immediately the consistent constraints required by Kolmogorov's extension theorem: this is described by our main theorem~\ref{theo:kolmogorovwithboundaryweights}. These sets of local equations are sufficiently explicit to be solved with finite-dimensional matrices and explicit recursions, by hand or on a computer.

Then, in section~\ref{sec:foldingsetc}, we show that the boundary eigen-elements obtained in the previous two theorems admit  alternative representations through theorems~\ref{theo:hs:fixedpoint:formulae} and \ref{theo:corner:fixedpoint:formulae}. These representations show that these eigen-elements can be \emph{partly} recovered from the classical technique of the transfer matrices through a technique named folding (see figure~\ref{fig:fold:fromhalfplanetocorner}). Moreover, theorem~\ref{theo:gluingtodoublyinf} shows that the eigen-elements also allows for the reconstruction of the two transfer matrices (one in each direction). However, all the parts of the eigen-elements related to both dimensions simultaneously can be obtained only by the operadic method.

As illustrated in figure~\ref{fig:admissiblepatterns}, the operadic approach rebuilds the classical extended objects such as the transfer matrix out of small and local objects, as in section~\ref{sec:fixedpoints} (bottom to top arrow in the figure) where the classical transfer matrix can be used (foldings) to extract part of the local objects, as in section~\ref{sec:foldingsetc} (top to bottom incomplete arrow).

A final section~\ref{sec:complementary} contains reminders on Schur complements and the one-dimensional processes used in the proofs, as well as complementary operadic constructions necessary to fully match with \cite{Simon}.

\paragraph*{Presentation of the results}

Gaussian Markov processes are well-known and their infinite-volume Gibbs measures are well characterized \cite{georgii}. The probabilistic content of the paper is to provide other methods, based on operads, of construction of these measures through theorem~\ref{theo:kolmogorovwithboundaryweights}. We largely insist on the fact that the method presented here is much more general than the pure Gaussian settings: therefore, we present the arguments in a form that depends as little as possible on purely Gaussian properties, in order to be reusable in other contexts. Of course, the reader who focuses on the Gaussian computations will recognize at various stages correspondence with other well-known properties (such as Green functions of random walks or harmonic functions on the lattice).  

In particular, we structure the proofs of the theorems along the precise concepts and tools that they use and that we think to be applicable to other models. In particular, a special role is played by the analytic folding of section~\ref{sec:fold} (related to the guillotine geometry) as well as the square associativity which is part of the definition of the guillotine operad. This square associativity is presented as a systematic tool of guessing of remarkable identities for the concrete computations.

\subsection{Description of the model}

We consider the lattice $\setZ^2\subset \setR^2$ seen as a planar graph with vertices $V=\setZ^2$, horizontal edges  $E_h=\{[k,k+1]\times\{l\}; k,l\in\setZ\}$, vertical edges $E_v= \{\{k\}\times[l,l+1]; k,l\in\setZ\}$ and faces $F=\{ [k,k+1]\times[l,l+1]; k,l\in\setZ\}$. A \emph{domain} (resp. finite domain) of $\setZ^2$ is a sub-planar graph of $\setZ^2$ obtained by considering a subset (resp. a finite subset) of the set $F$ of faces and all the edges and vertices contained in these faces. The boundary $\partial D$ of a domain $D$ is defined as the set of edges that belongs to exactly one face. The set of horizontal (resp. vertical) faces of a domain $D$ is written $E_h(D)$ (resp. $E_v(D)$). For a given face $[k,k+1]\times[l,l+1]$, the edge $[k,k+1]\times \{l\}$ (resp. $[k,k+1]\times \{l\}$, $\{k\}\times [l,l+1]$ and $\{k\}\times [l,l+1]$) is called the \emph{South} (resp. \emph{North}, \emph{West}, \emph{East}) edge, abbreviated with $\South$ (resp. $\North$, $\West$, $\East$) allthrough the paper.

We introduce two real or complex finite-dimensional Hilbert spaces $\ca{H}_1$ and $\ca{H}_2$ and attach to each horizontal edge in $E_h$ a vector in $\ca{H}_1$ and to each vertical edge in $E_v$ a vector in $\ca{H}_2$ through fields $X^{(1)} : E_h\to \ca{H}_1$ and $X^{(2)} : E_v \to \ca{H}_2$. These vectors are chosen at random by using a centred normal distribution with a density w.r.t~the Lebesgue measure on the spaces $\ca{H}_i$ (in particular we consider only Gaussian laws with full support in $\ca{H}_i$ such that the covariance matrix is positive definite).We are interested in random fields $X$ on domains $D$ with the Markov property, i.e.~such that, given any partition of $D$ into two domains $D_1$ and $D_2$, the two restrictions of $X$ to each sub-domains are independent \emph{conditionally} to the values of the field on the boundary of the sub-domains. For a Gaussian process on the edges of $D$, this corresponds (see \cite{Simon}) to the following definition.

\begin{defi}\label{def:GaussMarkovProc:via:density}
	Let $\ca{H}_1$ and $\ca{H}_2$ be two finite-dimensional Hilbert spaces. Let $D$ be a finite domain of $\setZ^2$. A $(
	\ca{H}_1,\ca{H}_2)$-valued homogeneous Gaussian Markov on $D$ is a process $(X_e)_{e\in E(D)}$ on some probability space $(\Omega,\ca{F},\Prob)$ such that:
	\begin{enumerate}[(i)]
		\item for any $e\in E_h(D)$, $X_e$ is a $\ca{H}_1$-valued random variable;
		\item for any $e\in E_v(D)$, $X_e$ is a $\ca{H}_2$-valued random variable;
		\item there exists a Hermitian positive definite operator $Q$ on $\ca{H}_1\oplus\ca{H}_1\oplus\ca{H}_2\oplus\ca{H}_2$ such that the conditional law of $(X_e)_{e\in E(D)}$ knowing $(X_e)_{e\in \partial D}$ admits the following density $f_{D,Q}$ with respect to the product Lebesgue measure on the internal edges $e\in E(D)\setminus\partial D$:
		\begin{equation}\label{eq:def:gaussmarkovinv}
			f_{D,A,x_{\partial D}}( x_{E(D)\setminus \partial D} )
			= \frac{1}{Z_D(Q; x_{\partial D})}
			\prod_{f\in F(D) } \exp\left( - \frac{1}{2}x^*_{\partial f} Q x_{\partial f}^{} \right)
		\end{equation}
	\end{enumerate}
\end{defi}
In order to have notations that can be easily interpreted, any variable associated to a horizontal (resp. vertical) edge will be noted $X^{(1)}_e$ and $X^{(2)}_e$. Whenever a coordinate system is needed, the $\ca{H}_1$-valued (resp. $\ca{H}_2$-valued) variable $X_{[k,k+1]\times\{l\}}$ (resp. $X_{\{k\}\times[l,l+1]}$) will be noted $X^{(1)}_{k,l}$ (resp. $X^{(2)}_{k,l}$), where $(k,l)$ corresponds to the left (resp. bottom) vertex of the edge.

The partition function $Z_D(Q; x_{\partial D})$ is obtained as the integral over all the values attached to interior edges of the domain $D$ of the product of exponentials and is thus given by
\begin{equation}\label{eq:partitionfunctionasquadraticform}
	Z_D(Q; x_{\partial D} )= \alpha_{Q,D} \exp\left( -\frac{1}{2}x^*_{\partial D} \ti{Q}_D x_{\partial D}^{}\right)
\end{equation}
where $\ti{Q}_D$ is a quadratic form on the boundary variables $(x_e)_{e\in\partial D}$.

Obtaining the complete law of the process requires to take the expectation w.r.t.~the boundary variables $x_{\partial D}$ whose law is written $\nu_{\partial D}$. For deterministic boundary conditions, $\nu_{\partial D}$ is a product of Dirac measure on each edge and is not absolutely continuous w.r.t.~the Lebesgue measure on the spaces $\ca{H}_i$. If the boundary law $\nu_{\partial D}$ admits a density  w.r.t.~the product Lebesgue measure, then the complete law of the process admits a density 
$ g_D(x_{\partial D})f_{D,A,x_{\partial D}}(x_{E(D)\setminus\partial D})/Z_D^{\boundaryweights}(Q;g_D)$
whose partition function is given by
\[
Z_D^{\boundaryweights}(Q;g_D)=\int g_D(x_{\partial D}) Z_D(Q;x_{\partial D}) dx_{\partial D}.
\]
and satisfy the Markov property.
\begin{prop}[Markov property]
	For any partition of the domain $D$ into subdomains $(D_1,\ldots,D_n)$, the $n$ processes $(X_e)_{e\in E(D_i)}$, $1\leq i\leq n$, are independent conditionally to the r.v. $(X_e)_{e\in \cup_i \partial D_i}$.
\end{prop}

The face operator $Q$ has a block decomposition
\begin{equation}\label{eq:block_decompo_face_weight}
	Q = \begin{pmatrix}
		Q_{\South\South} & Q_{\South\North}  & Q_{\South\West} & Q_{\South\East} \\
		Q_{\North\South} & Q_{\North\North}  & Q_{\North\West} & Q_{\North\East} \\
		Q_{\West\South} & Q_{\West\North}  & Q_{\West\West} & Q_{\West\East} \\
		Q_{\East\South} & Q_{\East\North}  & Q_{\East\West} & Q_{\East\East} 
	\end{pmatrix}
\end{equation}
on the face space $\ca{H}_1\oplus\ca{H}_1\oplus\ca{H}_2\oplus\ca{H}_2$ where the first (resp. second, third and fourth) space corresponds to the variable on the South (resp. North, West and East) edge of the face (see figure \ref{fig:Qblockstruct}). The Hermitian property then requires $Q_{ab}=Q_{ba}^*$. We will often need to refer to the following sub-blocks of $Q$ and introduce the following block extraction notation, used all through the paper.
\begin{align}\label{eq:block_decompo_partielles}
	Q_{[ab],[ab]}&=\begin{pmatrix}
		Q_{aa} & Q_{ab}\\
		Q_{ba} & Q_{bb}
	\end{pmatrix} &  Q_{[ab],[c]}&=\begin{pmatrix}
		Q_{ac}\\
		Q_{bc}
	\end{pmatrix}
\end{align}
for $a$, $b$, $c\in\lbrace\South, \North, \West, \East\rbrace$.
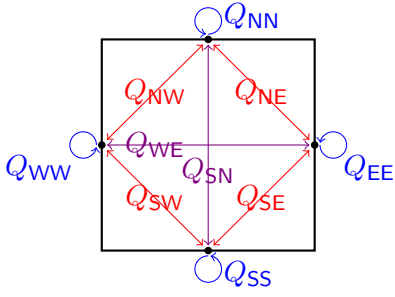
\begin{figure}
	\begin{center}
		\begin{tikzpicture}[scale=0.7]
			\draw[thick] (0,0) rectangle (4,4);
			\node at (2,0) [fill, circle, inner sep=1pt] {};
			\node at (2,4) [fill, circle, inner sep=1pt] {};
			\node at (0,2) [fill, circle, inner sep=1pt] {};
			\node at (4,2) [fill, circle, inner sep=1pt] {};
			
			\draw[->,blue] (-0.1,2) arc (0:350:0.25);
			\node at (-0.35,2) [below left,blue] {$Q_{\West\West}$};
			\draw[->,blue] (4.1,2) arc (-180:170:0.25);
			\node at (4.35,2) [below right,blue] {$Q_{\East\East}$};
			\draw[->,blue] (2,-0.1) arc (90:440:0.25);
			\node at (2.1,-0.45) [right,blue] {$Q_{\South\South}$};
			\draw[->,blue] (2,4.1) arc (-90:260:0.25);
			\node at (2.1,4.45) [right,blue] {$Q_{\North\North}$};
			
			\draw[<->,violet] (2,0.1) --(2,3.9);
			\draw[<->,violet] (0.1,2)--(3.9,2);
			\node at (1,2) [violet] {$Q_{\West\East}$};
			\node at (2,1.5) [violet] {$Q_{\South\North}$};

			\draw[<->,red] (2.1,0.1)-- node [midway] {$Q_{\South\East}$} (3.9,1.9);
			\draw[<->,red] (1.9,0.1)-- node [midway] {$Q_{\South\West}$} (0.1,1.9);
			\draw[<->,red] (2.1,3.9)-- node [midway] {$Q_{\North\East}$} (3.9,2.1);
			\draw[<->,red] (1.9,3.9)-- node [midway] {$Q_{\North\West}$} (0.1,2.1);
		\end{tikzpicture}
	\end{center}
	\caption{\label{fig:Qblockstruct}Block structure of the matrix $Q$.}
\end{figure}

The present paper provides the \emph{explicit computation} from a purely algebraic perspective based on operads of a consistent translation-invariant family of boundary weights $g_{p,q}$ on rectangular domains in the sense of Kolmogorov's extension theorem (see theorem~\ref{theo:kolmogorovwithboundaryweights} below), i.e. such that, for any rectangles $R'$ and $R$ with $R'\subset R$ and sizes $(p',q')$ and $(p,q)$:
\[
\int g^{(\Lambda)}_{p,q} (x_{\partial R}) Z_{R\setminus R'}(Q;x_{\partial R\cup \partial R'}) dx_{R\setminus (R'\cup \partial R)} =\Lambda^{pq-p'q'} g^{(\Lambda)}_{p',q'}\left( x_{\partial R'} \right)
\]
for some $\Lambda\in\setR_+$. This condition is much more general than the case of Gaussian fields but, in the present case, one observes that $g^{(\Lambda)}_{p,q}$ is itself a Gaussian weight 
\begin{equation}\label{eq:boundaryweight:quadraticform}
	g^{(\Lambda)}_{p,q}(x_{\partial R}) = \beta_{p,q} \exp\left(-\frac{1}{2}x^*_{\partial R} Q^{\partial}_{p,q} x_{\partial R}\right)
\end{equation}
for a suitable quadratic form $Q^{\partial}_{p,q}$ to be determined. In this case, we have that for any rectangle $R$ of size $(p,q)$, 
\[
\frac{\log Z_R^{\boundaryweights}(Q;g_R) }{pq} = f = \log \Lambda
\]
where $f$ is the so-called density of free energy. We emphasize that, in the present case, there is no approximation of the boundary condition by a guessed one but rather the obtention followed by the solution of an explicit system of equations that fully determines the weights $g^{(\Lambda)}_{p,q}$.

\subsection{Alternative analytical known results}

Before entering the algebraic setting, we quickly summarize some of the classical analytical approaches to the finite and infinite volume Gibbs measures. From a probabilist point of view, it is well known that the constant $\alpha_{Q,d}$ as well as the quadratic form $\ti{Q}_D$ of \eqref{eq:partitionfunctionasquadraticform} and the quadratic form $Q^{\partial}_{p,q}$ of \eqref{eq:boundaryweight:quadraticform} can be represented as Green functions of random walks and/or restriction of random walks cut when they enter or leave a finite or infinite domain. It is interesting to see that this result will be recovered from a purely operadic way below, which may work also for other types of models.

A first classical analytical way of obtaining the infinite volume Gibbs measure of our Gaussian Markovian model may consider  Dirichlet boundary conditions (or another fixed well-guessed boundary conditions) on an increasing family of finite domains and study the convergence of the local observables \cite{friedlivelenik} to a limit and use afterwards Riesz theorem. This involves limits that are easy for Gaussian models but may be difficult for other processes; moreover fixed deterministic boundary conditions may not be the most useful approximate choice to start a convergence scheme. The analytical part as well as the choice or guessing of suitable approximate boundaries are fully avoided in the present construction.

A second analytical way consists in using the Fourier transform, for example on the finite discrete torus or directly on $\setZ^2$ with white noises, to map the quadratic form on the torus or the plane to a diagonal one and then map the model to independent Gaussian variables. This is almost trivial but the locality of the Markovian field and the Markov property are completely hidden in this framework of Fourier transform, which is a global concept. Our algebraic approach allows to stick to local operations during all the constructions. 

\subsection{Quick sketch of the operadic approach to boundary weights}

The weights $g_{p,q}^{(\Lambda)}$ and their quadratic forms $Q^{\partial}_{p,q}$ have varying dimensions when $p$ and $q$ vary and hence are not easily comparable. This situation is very different from the 1D case in which the boundary of a segment is made of two points regardless of the length of the segments. It is non-trivial in dimension two to write eigenvalue/eigenvector equations directly on $g_{p,q}^{(\Lambda)}$: \cite{Simon} splits this weight into elementary objects which do not depend on the size $(p,q)$ and satisfy such eigenvalue/eigenvectors equations. This ensures the consistency of boundary weights $g_{p,q}^{(\Lambda)}$ and the construction of the Gibbs measure through Kolmogorov's extension theorem in our main theorem~\ref{theo:kolmogorovwithboundaryweights}. The splitting of  \cite{Simon} corresponds to a decomposition
\begin{subequations}\label{eq:boundaryweight:internalstruct}
	\begin{equation}
		g_{p,q}^{(\Lambda)}\left(
		\begin{tikzpicture}[yscale=0.25,xscale=0.4,baseline={(current bounding box.center)}]
			\draw (0,0) rectangle (4,3);
			\node at (0.5,0) [anchor = north] {$x_1$};
			\node at (1.5,0) [anchor = north]{$x_2$};
			\node at (2.5,0) [anchor = north]{$\ldots$};
			\node at (3.5,0) [anchor = north]{$x_{p}$};
			\node at (0.5,3) [anchor = south] {$y_1$};
			\node at (1.5,3) [anchor = south]{$y_2$};
			\node at (2.5,3) [above]{$\ldots$};
			\node at (3.5,3) [above]{$y_{p}$};
			\node at (0,0.5) [anchor = east]{ $w_1$ };
			\node at (0,1.5) [anchor = east]{ $\vdots$ };
			\node at (0,2.5) [anchor = east]{ $w_{q}$ };
			\node at (4,0.5) [anchor = west]{ $z_1$ };
			\node at (4,1.5) [anchor = west]{ $\vdots$ };
			\node at (4,2.5) [anchor = west]{ $z_{q}$ };
		\end{tikzpicture}
		\right)
		= \omega\left( \mathbf{A}_\South(x) U_{\South\East} \mathbf{A}_{\East}(z) U_{\North\East} \mathbf{A}_{\North}(y)U_{\North\West} \mathbf{A}_\West(w) \right)
	\end{equation}
	with the four side elements
	\begin{align}
		\mathbf{A}_\South(x) &= A_\South(x_1)\cdot_\South\ldots \cdot_\South A_\South(x_p) &
		\mathbf{A}_\East(z) &= A_\East(z_1)\cdot_\East\ldots \cdot_\East A_\East(z_q)
	\end{align}
	\begin{align}
		\mathbf{A}_\North(y) &= A_\North(y_p)\cdot_\North\ldots \cdot_\North A_\North(y_1) &
		\mathbf{A}_\West(w) &= A_\West(w_q)\cdot_\West\ldots \cdot_\West A_\West(w_1) 
	\end{align}
\end{subequations}
where the $\mathbf{A}_a(u)$ are operators in some suitable associative structure, the $U_{ab}$ are bimodule-like objects on these structures and $\omega$ is a tracial-like state. Then, the notion of eigen-element up to morphisms can be defined (see \cite{Simon}) for these objects. The following sections present, through theorems~\ref{theo:westfixedpoint}, \ref{theo:corner:charactblocks}, \ref{theo:hs:fixedpoint:formulae} and \ref{theo:corner:fixedpoint:formulae} how these abstract definitions provide computational tools that determine these operators, at least for Gaussian models. Treating all the spaces, products and types of boundary eigen-elements simultaneously require the formalism of operads of \cite{Simon} that we now quickly present.

\section{The operadic approach: lifts, shifts and fixed points}\label{sec:fixedpoints}

\subsection{The guillotine operad for Gaussian Markovian processes on finite rectangles}

\subsubsection{Finite domains}
The theory in \cite{Simon} is based on the guillotine operad $\Guill_2$ which provides an algebraic setting for the computation of partition function and observables when rectangular domains are glued together. The operad is a coloured operad whose colours are given by the shapes $(p,q)$ of the rectangles that are glued along their sides (guillotine cuts). 

We summarize quickly in the following theorem the main construction of \cite{Simon} for finite rectangles and invite the reader to read simultaneously the abstract definitions below and their explicit realizations in the Gaussian case in section~\ref{sec:gaussianlift}.

\begin{theo}[see \cite{Simon}]\label{theo:guillstruct}
	For all $n\in\setN$ and any sequence of colours $(c,c_1,\ldots,c_n)\in \PatternShapes^{n+1}$ with $\PatternShapes=\setN_1^2$ let
	$\Guill_2(c;c_1,\ldots,c_n)$ be the set of equivalence classes of guillotine partitions $(R_1,\ldots,R_n)$ of a rectangle $R$ such that $R$ has size $c$, for all $1\leq i\leq n$, the size of $R_i$ is equal to $c_i$. These sets define a coloured operad $\Guill_2^{(r)}$ over the set of colours $\ov{\PatternShapes}$ with compositions described in \cite{Simon} generated by elementary products in $\Guill_2((p_1+p_2,q);(p_1,q),(p_2,q))$ and in $\Guill_2((p,q_1+q_2);(p,q_1),(p,q_2))$ submitted to three associativity conditions, named horizontal, vertical and square associativities.
\end{theo}

We adopt the graphical notations of \cite{Simon}. For a given $\Guill_2^{(r)}$-algebra $(E_c)_{c\in C}$, the product $m_\rho$ associated to a guillotine partition $\rho =(R_1,R_2,R_3)$ of a rectangle $R$ with shape $c$ evaluated on $(u_1,u_2,u_3)\in E_{c_1}\times E_{c_2} \times E_{c_3}$ produces an element in $E_c$ and is written in one of the three following equivalent ways:
\begin{equation}\label{eq:exampleguillpart}
	m_{\rho}(u_1,u_2,u_3) 
	=
	\begin{tikzpicture}[guillpart,yscale=0.8,xscale=0.8]
		\fill[guillfill] (0,0) rectangle (3,3);
		\draw[guillsep] (0,0) rectangle (3,3);
		\draw[guillsep] (2,0)--(2,3) (0,1)--(2,1);
		\node at (1,0.5) {$1$};
		\node at (2.5,1.5) {$3$};
		\node at (1,2) {$2$};
	\end{tikzpicture}(u_1,u_2,u_3)
	= 
	\begin{tikzpicture}[guillpart,yscale=0.8,xscale=1]
		\fill[guillfill] (0,0) rectangle (3,3);
		\draw[guillsep] (0,0) rectangle (3,3);
		\draw[guillsep] (2,0)--(2,3) (0,1)--(2,1);
		\node at (1,0.5) {$u_1$};
		\node at (2.5,1.5) {$u_3$};
		\node at (1,2) {$u_2$};
	\end{tikzpicture}
\end{equation}
where this drawing corresponds to a rectangle $R$ of size $(3,3)$ with partitions $R_1$ of size $(2,1)$, $R_2$ of size $(2,2)$ and $R_3$ of size $(1,3)$. The elements of the sets $E_c$ are written either as arguments with an order associated to the labelling of the inner rectangles or directly inside the rectangles to shorten the notation.

In the present probabilistic context, we consider the following $\Guill_2^{(r)}$-operad $(L^+_{p,q})_{(p,q)\in\PatternShapes}$ where $(M_1,\ca{M}_1,\mu_1)$ and $(M_2,\ca{M}_2,\mu_2)$ are two $\sigma$-finite measure spaces and, for any $(p,q)\in C$, the set $L^+_{p,q}$ is defined as:
\begin{align*}
	L^+_{p,q} &= \left\{ f: M_1^{2p}\times M_2^{2q} \to \ov{\setR}_+ ; \text{$f$ measurable} \right\}
\end{align*}
Elements in these sets are called quasi-densities on the boundary of a rectangle $(p,q)$ in the present paper. For the two types of elementary guillotine partitions with a vertical and a horizontal cut
\begin{align*}
	\rho_{p_1,p_2|q} &= \left( [0,p_1]\times [0,q], [p_1,p_1+p_2]\times [0,q]\right)
	\\
	\rho_{p|q_1,q_2} &= \left( [0,p]\times [0,q_1], [0,p]\times [q_1,q_1+q_2]\right)
\end{align*}
of the rectangles $[0,p_1+p_2]\times[0,q]$ and $[0,p]\times[0,q_1+q_2]$, we define
\begin{subequations}\label{eq:def:elemguillotprod:dens}
	\begin{align}
		m_{\rho_{p_1,p_2|q}} = \begin{tikzpicture}[guillpart,yscale=1.,xscale=1.]
			\fill[guillfill] (0,0) rectangle (2,1);
			\draw[guillsep] (0,0) rectangle (2,1)  (1,0)--(1,1);
			\node at (0.5,0.5) {$1$};
			\node at (1.5,0.5) {$2$};
		\end{tikzpicture}: L^+_{p_1,q}\times L^+_{p_2,q} & \to  L^+_{p_1+p_2,q} 
		\\
		m_{\rho_{p|q_1,q_2}} = \begin{tikzpicture}[guillpart,yscale=1.,xscale=1.]
			\fill[guillfill] (0,0) rectangle (1,2);
			\draw[guillsep] (0,0) rectangle (1,2) (0,1)--(1,1);
			\node at (0.5,0.5) {$1$};
			\node at (0.5,1.5) {$2$};
		\end{tikzpicture}: L^+_{p,q_1}\times L^+_{p,q_2} & \to  L^+_{p,q_1+p_2} 
	\end{align}
	through the following integrals:
	\begin{align}
		m_{\rho_{p_1,p_2|q}}&(f_1,f_2)(u_\South^{(1)},u_\South^{(2)},u_\North^{(1)},u_\North^{(2)},u_\West,u_\East)
		\\
		&=
		\int_{M_2^q} f_1(u_\South^{(1)},u_\North^{(1)},u_\West,v)f_2(u_\South^{(2)},u_\North^{(2)},v,u_\East)d\mu_2^{\otimes q}(v)	\nonumber
		\\
		m_{\rho_{p|q_1,q_2}}(f_1,f_2)&(u_\South,u_\North,u_\West^{(1)},u_\West^{(2)},u_\East^{(1)},u_\East^{(2)})
		\\
		&= 
		\int_{M_1^p} f_1(u_\South,v,u_\West^{(1)},u_\East^{(1)}) f_2(v,u_\North,u_\West^{(2)},u_\East^{(2)}) d\mu_1^{\otimes p}(v)\nonumber
	\end{align}
\end{subequations}
with $u_a^{(i)}\in M_1^{p_i}$ and $u_a\in M_1^p$ for $a\in \{\South,\North\}$ and $u_b^{(i)}\in M_2^{q_i}$ and $u_b\in M_2^q$ for $b\in \{\West,\East\}$. These two products multiply the quasi-densities and integrate over the variables associated to the common boundary of the two small rectangles and produce a quasi-density on the large rectangle.
\begin{theo}[see \cite{Simon}]\label{theo:L:guillstruct}
	The elementary guillotine products \eqref{eq:def:elemguillotprod:dens} define a $\Guill_2^{(\patterntype{r})}$-structure on the sets $(L^+_{p,q})_{(p,q)\in \PatternShapes}$. 
\end{theo}
This structure theorem is proved in \cite{Simon} and reinterprets various probabilistic computations of marginals and partition functions of Markov processes that we will use massively in the following sections. As an example, we have in the case of \eqref{eq:exampleguillpart}:
\[
Z_{R}(Q;\cdot) = 	\begin{tikzpicture}[guillpart,yscale=1.2,xscale=3]
	\fill[guillfill] (0,0) rectangle (3,3);
	\draw[guillsep] (0,0) rectangle (3,3);
	\draw[guillsep] (2,0)--(2,3) (0,1)--(2,1);
	\node at (1,0.5) {$Z_{R_1}(Q,\cdot)$};
	\node at (2.5,1.5) {$Z_{R_3}(Q;\cdot)$};
	\node at (1,2) {$Z_{R_2}(Q;\cdot)$};
\end{tikzpicture}
\]
Here, the choice of $L^+$ spaces with $\ov{\setR}_+$ as a target space prevents these spaces to have a vector space structure (and thus eigenvalues) but it remains the quickest way to perform the Gaussian integration with arbitrary quadratic forms. These difficulties are irrelevant here for full support Gaussian processes on rectangles using the lifts defined below that are based Schur complements; however they will pop up again for the boundary eigen-elements and will require a careful treatment. 

\subsubsection{Extensions to boundaries}

In order to deal with boundary weights with an internal structure as announced in \eqref{eq:boundaryweight:internalstruct}, \cite{Simon} introduces extended partition functions to deal with arbitrary guillotine partitions of the full discrete lattice $\setZ^2$ and not only finite rectangles. This is done by extending the additive colour palette of the operads to \[
\ov{\PatternShapes}=\left( \setN_1\cup \{\infty_L,\infty_R,\infty_{LR} \} \right)^2
\]
where $\PatternShapes$ stands for "pattern shapes" and with the following conventions for finite $u$:
\begin{align*}
	u-(-\infty) &= \infty_L,   
	&
	(+\infty)-x &= \infty_R, 
	&
	(+\infty)-(-\infty) &= \infty_{LR} 
\end{align*} 
where $L$ and $R$ stands for left and right on the real line. However, to make notations easier to interpret, we will write any element
\begin{align*}
	(a,\infty_L) &= (a,\infty_\South) 
	&
	(a,\infty_R) &= (a,\infty_\North) 
	\\
	(\infty_L, b) &= (\infty_\West,b) 
	&
	(\infty_R, b) &= (\infty_\East,b) 
	\\
	(a,\infty_{LR}) &= (a,\infty_{\South\North})
	&
	(\infty_{LR},b) &= (\infty_{\West\East},b)
\end{align*}
where left and right are replaced by the direction in the plane in the corresponding direction. A synthetic list of all the extended shapes with their colour is presented in figure~\ref{fig:admissiblepatterns}. A rectangular-like shape with a size equal to $\infty_{\West\East}$ or $\infty_{\South\North}$ in one direction contains a line in this direction and thus requires to be pointed, as done in \cite{Simon} to remove redundancy.  It is an easy exercise to check that two translation-equivalent guillotine partitions are made of rectangular shapes with the same sizes. 

\begin{figure}
	\begin{center}
		{\footnotesize
			\begin{tikzpicture}[yscale=0.75, xscale=0.7]
				\draw[->,blue, thick] (-6.5,10)--(-6.5,3);
				\node at (-6.4,3) [anchor=west,blue,rectangle,draw] {section~\ref{sec:foldingsetc}};
				\draw[->,blue, thick] (6.75,0)--(6.75,10);
				\node at (6.6,10) [anchor=east,blue,rectangle,draw] {section~\ref{sec:fixedpoints}};
				\begin{scope}[yshift=9cm]
					\begin{scope}
						\fill[guillfill] (0,0) rectangle (1,1);
						\node at (0.5,0.5) {$\setZ^2$};	
						\node at (0,0.5) [anchor=east] {Full plane};
						\node at (0.5,0.) [anchor = north] {$(\infty_{\West\East},\infty_{\South\North})$};
					\end{scope}
				\end{scope}
				
				\begin{scope}[yshift=7cm]
					\begin{scope}[xshift= 4.cm]
						\fill[guillfill] (0,0) rectangle (1,1);
						\draw[guillsep] (0,1) -- (1,1);
						\node at (0.5,0.) [anchor = north] {$(\infty_{\West\East},\infty_\South)$};
					\end{scope}
					
					\begin{scope}[xshift= 1.5cm]
						\fill[guillfill] (0,0) rectangle (1,1);
						\draw[guillsep] (0,0) -- (1,0);
						\node at (0.5,0.) [anchor = north] {$(\infty_{\West\East},\infty_\North)$};
					\end{scope}
					
					\begin{scope}[xshift= -1.5cm]
						\fill[guillfill] (0,0) rectangle (1,1);
						\draw[guillsep] (1,0) -- (1,1);
						\node at (0.5,0.) [anchor = north] {$(\infty_\West,\infty_{\South\North})$};
					\end{scope}
					
					\begin{scope}[xshift= -4cm]
						\fill[guillfill] (0,0) rectangle (1,1);
						\draw[guillsep] (0,0) -- (0,1);
						\node at (0,0.5) [anchor = east] {Half-planes};
						\node at (0.5,0.) [anchor = north] {$(\infty_\East,\infty_{\South\North})$};
					\end{scope}
				\end{scope}
				
				\begin{scope}[yshift=4.5cm]
					
					\node at (-2.5,1.5) {Corners};
					\draw (-5.5,1.1)-- (-5.5,1.3) -- (1.5,1.3) -- (1.5,1.1);
					
					\node at (5.,1.5) {Strips};
					\draw (3.5,1.1)-- (3.5,1.3) -- (6.5,1.3) -- (6.5,1.1);
					
					\begin{scope}[xshift= 5.5cm]
						\fill[guillfill] (0,0) rectangle (1,1);
						\draw[guillsep] (0,0) -- (0,1);
						\draw[guillsep] (1,1) -- (1,0);
						\node at (0.5,0.) [anchor = north] {$(p,\infty_{\South\North})$};
					\end{scope}

					\begin{scope}[xshift= 3.5cm]
						\fill[guillfill] (0,0) rectangle (1,1);
						\draw[guillsep] (0,0) -- (1,0); 
						\draw[guillsep] (1,1) -- (0,1);
						\node at (0.5,0.) [anchor = north] {$(\infty_{\West\East},q)$};
					\end{scope}
					
					\begin{scope}[xshift= 0.5cm]
						\fill[guillfill] (0,0) rectangle (1,1);
						\draw[guillsep]  (0,1) -- (1,1) -- (1,0);
						\node at (0.5,0.) [anchor = north] {$(\infty_\West,\infty_\South)$};
					\end{scope}
					
					\begin{scope}[xshift= -1.5cm]
						\fill[guillfill] (0,0) rectangle (1,1);
						\draw[guillsep] (1,1) -- (1,0) -- (0,0);
						\node at (0.5,0.) [anchor = north] {$(\infty_\West,\infty_\North)$};
					\end{scope}
					
					\begin{scope}[xshift= -3.5cm]
						\fill[guillfill] (0,0) rectangle (1,1);
						\draw[guillsep] (1,0) -- (0,0) -- (0,1);
						\node at (0.5,0.) [anchor = north] {$(\infty_\East,\infty_\North)$};
					\end{scope}
					
					\begin{scope}[xshift= -5.5cm]
						\fill[guillfill] (0,0) rectangle (1,1);
						\draw[guillsep] (0,0) -- (0,1) -- (1,1);
						\node at (0.5,0.) [anchor = north] {$(\infty_\East,\infty_\South)$};
					\end{scope}
				\end{scope}

				\begin{scope}[yshift=2cm]
					\begin{scope}[xshift= 3cm]
						\fill[guillfill] (0,0) rectangle (1,1);
						\draw[guillsep] (0,0) -- (0,1) -- (1,1) -- (1,0);
						\node at (0.5,0.) [anchor = north] {$(p,\infty_\South)$};
					\end{scope}
					
					\begin{scope}[xshift= -3cm]
						\fill[guillfill] (0,0) rectangle (1,1);
						\draw[guillsep]  (0,1) -- (1,1) -- (1,0)-- (0,0);
						\node at (-0.5,0.5) [anchor = east] {Half-strips};
						\node at (0.5,0.) [anchor = north] {$(\infty_\West,q)$};
					\end{scope}
					
					\begin{scope}[xshift= 1cm]
						\fill[guillfill] (0,0) rectangle (1,1);
						\draw[guillsep] (1,1) -- (1,0) -- (0,0) -- (0,1);
						\node at (0.5,0.) [anchor = north] {$(p,\infty_\North)$};
					\end{scope}
					
					\begin{scope}[xshift= -1cm]
						\fill[guillfill] (0,0) rectangle (1,1);
						\draw[guillsep] (1,0) -- (0,0) -- (0,1)--(1,1);
						\node at (0.5,0.) [anchor = north] {$(\infty_\East,q)$};
					\end{scope}
				\end{scope}
				
				\begin{scope}
					\fill[guillfill] (0,0) rectangle (1,1);
					\draw[guillsep] (0,0) rectangle (1,1);
					\node at (0,0.5) [anchor = east] {Rectangles};
					\node at (0.5,0.) [anchor = north] {$(p,q)$};
				\end{scope}
			\end{tikzpicture}
		}
	\end{center}
	\caption{\label{fig:admissiblepatterns}(extracted from \cite{Simon}) Six types of admissible patterns, with their pattern shape below. Any pattern in a line may appear in the guillotine partitions of some of the patterns above this line. All finite sizes $p$ and $q$ are elements of $\setN_1$. The left top-to-bottom arrow explains how traditional computations based on Fourier transform and transfer matrices, enhanced with foldings in section~\ref{sec:foldingsetc} can be used to split the plane until blocks for corners and halfstrips. The right bottom-to-top arrow explains how the operadic rules described in section~\ref{sec:fixedpoints}, based on elementary local operations and equations, allows one to reach the full plane Gibbs measure in a very different way.}
\end{figure}
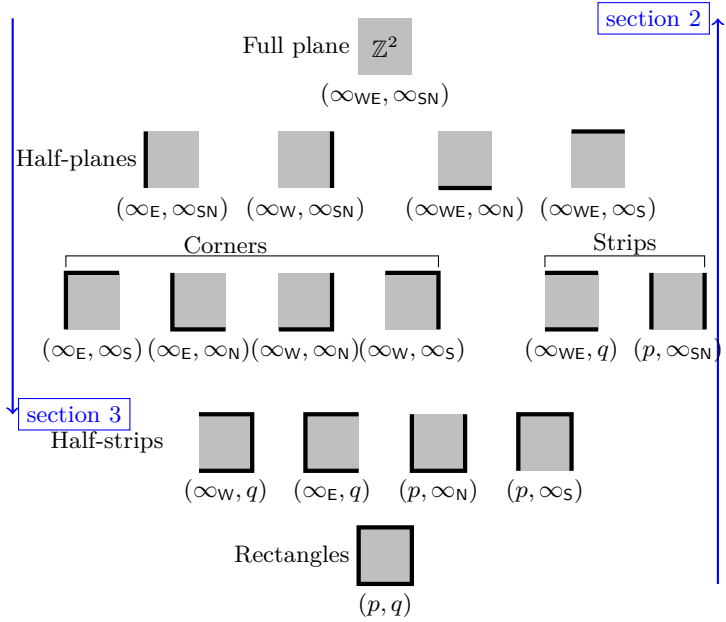

The same theorem as theorem~\ref{theo:guillstruct} remains true when $\PatternShapes$ is replaced by $\ov{\PatternShapes}$ provided guillotine partitions with at least a value $\infty_{LR}$ are pointed. If a rectangular shape is infinite in one direction, the external line is not drawn. For example, the product associated to the elementary guillotine partition 
\[
\rho =(
(-\infty,x-1]\times [y,y+2],
(-\infty,x]\times [y+2,y+3],
[x-1,x]\times [y,y+2]
)
\] applied to $(a,b,c)\in E_{\infty_\West,2}\times E_{\infty_\West,1} \times E_{1,2}$ is represented as
\[
m_\rho(a,b,c) =
\begin{tikzpicture}[guillpart,yscale=0.8,xscale=1]
	\fill[guillfill] (0,0) rectangle (3,3);
	\draw[guillsep] (0,0) --(3,0) --(3,3) --(0,3) (0,2)--(3,2)  (2,0)--(2,2);
	\node at (1,1) {$1$};
	\node at (1.5,2.5) {$2$};
	\node at (2.5,1) {$3$};
\end{tikzpicture}(a,b,c)
= \begin{tikzpicture}[guillpart,yscale=0.8,xscale=1]
	\fill[guillfill] (0,0) rectangle (3,3);
	\draw[guillsep] (0,0) --(3,0) --(3,3) --(0,3) (0,2)--(3,2)  (2,0)--(2,2);
	\node at (1,1) {$a$};
	\node at (1.5,2.5) {$b$};
	\node at (2.5,1) {$c$};
\end{tikzpicture}
\]
For rectangular shapes with a size equal to $\infty_{\West\East}$ or $\infty_{\South\North}$ that require a pointing in this direction, we add a thick point and a dotted line in the drawing. For example, the guillotine partition 
\[
\rho = (\setZ\times [y,y+1], (-\infty,x]\times [y+1,y+3], [x,+\infty)\times [y+1,y+3])
\] of $\setZ\times [y,y+3]$ with a pointing at $x+2$ is represented by
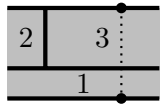
\[
\begin{tikzpicture}[guillpart,yscale=0.8,xscale=1]
	\fill[guillfill] (0,0) rectangle (4,3);
	\draw[guillsep] (0,0)--(4,0) (0,3)--(4,3) (0,1)--(4,1) (1,1)--(1,3); 
	\node at (0.5,2) {$2$};
	\node at (2.5,2) {$3$};
	\node at (2,0.5) {$1$};
	\node (P1) at (3,0) [circle, fill, inner sep=0.5mm] {};
	\node (P2) at (3,3) [circle, fill, inner sep=0.5mm] {};
	\draw[thick, dotted] (P1)--(P2);
\end{tikzpicture}
\]

However, on the probabilistic side, the extension of theorem~\ref{theo:L:guillstruct} for quasi-densities is not trivial on the boundaries. We consider as an example the case of a West half-strip of width $q$ and look for a candidate for the space $L^+_{\infty_\West,q}$. A candidate may be the space of functions $f:M_1^{\setN}\times M_2 \to\ov{\setR}_+$ but is ill-defined if $\mu_1$ is not finite and is subject to analytical subtleties; moreover, as seen in some elementary examples (see below), this may not be the best choice. The advantage of the operadic definitions "up to morphisms" in \cite{Simon} is precisely to allow for a wide variety of possible spaces that are equivalent and the equivalence class of spaces \emph{is part of the unknown quantities} when constructing eigen-elements on the boundary. The purpose of section~\ref{sec:foldingsetc} is indeed to derive algebraically a natural candidate for Gaussian spaces out of the sole operadic definitions.

\subsection{Gaussian weights and lifts}\label{sec:gaussianlift}

The case of Gaussian quasi-densities corresponds to a particular subset of elements $\ee_{Q}$ in $L^+_{p,q}$ with $M_1=\ca{H}_1$ and $M_2=\ca{H}_2$ endowed with the Lebesgue measure. For any $(p,q)\in \PatternShapes=\setN_1\times \setN_1$, we introduce the set of Hermitian positive definite operators on the boundary direct sum of Hilbert spaces:
\begin{equation}\label{eq:def:Qpq:rect}
	\ca{Q}_{p,q} = 
	\left\{ 
	Q \in \End\left(H_1^{2p} \oplus H_2^{2q}\right) ;
	Q^*=Q \text{ and $Q$ positive definite }
	\right\}
\end{equation}
In $H_1^{2p}$ (resp. $H_2^{2q}$), the first $p$ (resp. $q$) vectors correspond to the $p$ (resp. $q$) vectors on a South (resp. West) boundary of a rectangle ordered from left to right (resp. bottom to top) and the last $p$ (resp. $q$) vectors to the $p$ (resp. $q$) vectors on the North (resp. East) boundary with the same ordering. We emphasize on the fact that these sets are not algebras nor vector spaces. For each element $Q\in \ca{Q}_{p,q}$, we define the Gaussian weight $\ee_{Q}\in L^+_{p,q}$ by
\begin{align*}
	\ee_{Q} : H_1^{2p} \oplus H_2^{2q} &\to \setR  
	&
	u &\mapsto  \exp\left( -\frac{1}{2} u^* Q u \right)
\end{align*}
and the map $\ca{Q}_{p,q}\to L^+_{p,q}$ is injective.

We now recall the following fundamental Gaussian integration lemma.
\begin{lemm}[Schur complement]\label{lemm:Schur}
	Let $\ca{W}$ be a finite dimensional Hilbert space and let $\ca{W}_1$ and $\ca{W}_2$ be two subspaces of $\ca{W}$ with dimensions $d_1$ and $d_2$ such that $\ca{W}=\ca{W}_1\oplus\ca{W}_2$. Let $Q$ be a Hermitian positive definite operator on $\ca{W}$ with a decomposition
	\[
	Q = \begin{pmatrix}
		Q_{11} & Q_{12} \\
		Q_{21} & Q_{22} 
	\end{pmatrix}
	\]
	where $Q_{ij}$ is an operator $\ca{W}_j\to\ca{W}_i$. Then,
	\begin{equation}
		\int_{\ca{W}_2} \exp\left( -\frac{1}{2} \begin{pmatrix}
			u_1 \\ u_2 
		\end{pmatrix}^* Q \begin{pmatrix}
			u_1 \\ u_2
		\end{pmatrix}\right)
		du_2 = (2\pi)^{d_2} \det(Q_{22}^{-1}) \exp\left( -\frac{1}{2} u_1^* \ov{Q}_{11} u_1 \right)
	\end{equation}
	where the matrix called the Schur complement $\ov{Q}_{11}=Q_{11}-Q_{12}Q_{22}^{-1}Q_{21}$ is again Hermitian positive definite on $\ca{W}_1$.
\end{lemm}
We now use this lemma to lift the $\Guill_2^{(r)}$-structure on $(L^+_{p,q})$ to $(\ca{Q}_{p,q})$.
\begin{prop}\label{prop:fromweightstoquadraticforms}
	Let $(p_1,q)$ and $(p_2,q)$ be two rectangle sizes in $C$. For any $Q_1\in \ca{Q}_{p_1,q}$ and any $Q_2\in \ca{Q}_{p_2,q}$, there is a unique $\gamma_{\West\East}(Q_1,Q_2)\in\setR^*_+$ and a unique element $\Schur_{\West\East}(Q_1,Q_2)\in\ca{Q}_{p_1+p_2,q}$ such that,
	\begin{equation}\label{eq:horizSchur}
		\begin{tikzpicture}[guillpart,yscale=1.15,xscale=2.5]
			\fill[guillfill] (0,0) rectangle (2,1);
			\draw[guillsep] (0,0) rectangle (2,1) (1,0)--(1,1);
			\node at (0.5,0.5) {$\ee_{Q_1}$};
			\node at (1.5,0.5) {$\ee_{Q_2}$};
		\end{tikzpicture}_{L^+}
		=
		\gamma_{\West\East}(Q_1,Q_2) \ee_{\Schur_{\West\East}(Q_1,Q_2)}
	\end{equation}
	Let $(p,q_1)$ and $(p,q_2)$ be two rectangle sizes in $C$. For any $Q_1\in \ca{Q}_{p,q_1}$ and any $Q_2\in \ca{Q}_{p,q_2}$ there is a unique $\gamma_{\South\North}(Q_1,Q_2) \in \setR^*_+$ and a unique element $\Schur_{\South\North}(Q_1,Q_2)\in\ca{Q}_{p,q_1+q_2}$ such that, 
	\begin{equation}\label{eq:vertSchur}
		\begin{tikzpicture}[guillpart,yscale=1.15,xscale=2.5]
			\fill[guillfill] (0,0) rectangle (1,2);
			\draw[guillsep] (0,0) rectangle (1,2) (0,1)--(1,1);
			\node at (0.5,0.5) {$\ee_{Q_1}$};
			\node at (0.5,1.5) {$\ee_{Q_2}$};
		\end{tikzpicture}_{L^+}
		=\gamma_{\South\North}(Q_1,Q_2)\ee_{\Schur_{\South\North}(Q_1,Q_2)}
	\end{equation}
\end{prop}
\begin{proof}
	The proof is a direct consequence of Gaussian integration through lemma~\ref{lemm:Schur} applied to the integrals in equations  \eqref{eq:def:elemguillotprod:dens}.
\end{proof}
A direct consequence is the lift of $\Guill_2$-structure from quasi-densities to quadratic forms.
\begin{coro}[guillotine structure of Gaussian weights]\label{coro:guillstruct:Q}
	The sets $(\ca{Q}_{p,q})_{(p,q)\in C}$ are also endowed with a $\Guill_2^{(r)}$-structure given by the Schur complements $\Schur_{\West\East}$ and $\Schur_{\South\North}$ as generating products:
	\begin{align}\label{eq:SchurforQ}
		\begin{tikzpicture}[guillpart,yscale=1.15,xscale=1.5]
			\fill[guillfill] (0,0) rectangle (2,1);
			\draw[guillsep] (0,0) rectangle (2,1) (1,0)--(1,1);
			\node at (0.5,0.5) {$Q_1$};
			\node at (1.5,0.5) {$Q_2$};
		\end{tikzpicture}_\ca{Q} 
		&= \Schur_{\West\East}(Q_1,Q_2)
		&	
		\begin{tikzpicture}[guillpart,yscale=1.15,xscale=1.5]
			\fill[guillfill] (0,0) rectangle (1,2);
			\draw[guillsep] (0,0) rectangle (1,2) (0,1)--(1,1);
			\node at (0.5,0.5) {$Q_1$};
			\node at (0.5,1.5) {$Q_2$};
		\end{tikzpicture}_\ca{Q}
		&=\Schur_{\South\North}(Q_1,Q_2)
	\end{align}
\end{coro}
We emphasize however that the sets $\ca{Q}_{p,q}$ are not vector spaces and the products are non-linear operations on the matrices $Q_i$. For more details about Schur complements, the reader can jump to section~\ref{sec:complementary:Schur_Complements}. During the lift from $L^+_\bullet$ to $\ca{Q}_\bullet$, the normalization coefficients $\gamma_{\West\East}$ and $\gamma_{\South\North}$ are lost but they can be recovered afterwards through their cocycle property with respect to the products $\Schur_{\West\East}$ and $\Schur_{\South\North}$: this aspect is described in detail in section~\ref{sec:eigenval}.

In order to have easier notation, we finally introduce the notion of "surface power". 
The partition function $Z_R(Q;x_{\partial R})$ of \eqref{eq:def:gaussmarkovinv} on a rectangle of size $(p,q)$ is a function in $L^+_{p,q}$ (using lemma~\ref{lemm:Schur} above) given by
\[
Z_R(Q;\cdot)  =  \begin{tikzpicture}[guillpart,yscale=1.3,xscale=1.5]
	\fill[guillfill] (0,0) rectangle (4,3);
	\draw[guillsep] (0,0) -- (4,0) (0,1)--(4,1) (0,2)--(4,2) (0,3)--(4,3);
	\draw[guillsep] (0,0) -- (0,3) (1,0)--(1,3) (2,0)--(2,3) (3,0)--(3,3) (4,0)--(4,3);
	\node at (0.5,0.5) { $\ee_{Q}$ };
	\node at (1.5,0.5) { $\ee_{Q}$ };
	\node at (2.5,0.5) { $\dots$ };
	\node at (3.5,0.5) { $\ee_{Q}$ };
	\node at (0.5,2.5) { $\ee_{Q}$ };
	\node at (1.5,2.5) { $\ee_{Q}$ };
	\node at (2.5,2.5) { $\dots$ };
	\node at (3.5,2.5) { $\ee_{Q}$ };
	\node at (0.5,1.5) { $\vdots$ };
	\node at (1.5,1.5) { $\vdots$ };
	\node at (2.5,1.5) { $\vdots$ };
	\node at (3.5,1.5) { $\vdots$ };
\end{tikzpicture}
= (\ee_{Q})^{[p,q]} = \alpha_{p,q} \ee_{Q^{[p,q]}}
\]
with $p$ horizontal edges and $q$ vertical edges in the partitions, where $a^{[p,q]}$ is the surface power notation of \cite{Simon} and corresponds to a guillotine partition along a grid with cells of the same sizes and where $Q^{[p,q]}$ is now the surface power in the lifted guillotine operad $(\ca{Q}_{p,q})$ noted
\begin{equation}
	\label{eq:Q:surfacepowers}
	Q^{[p,q]} = \begin{tikzpicture}[guillpart,yscale=1.1,xscale=1.5]
		\fill[guillfill] (0,0) rectangle (4,3);
		\draw[guillsep] (0,0) -- (4,0) (0,1)--(4,1) (0,2)--(4,2) (0,3)--(4,3);
		\draw[guillsep] (0,0) -- (0,3) (1,0)--(1,3) (2,0)--(2,3) (3,0)--(3,3) (4,0)--(4,3);
		\node at (0.5,0.5) { $Q$ };
		\node at (1.5,0.5) { $Q$ };
		\node at (2.5,0.5) { $\dots$ };
		\node at (3.5,0.5) { $Q$ };
		\node at (0.5,2.5) { $Q$ };
		\node at (1.5,2.5) { $Q$ };
		\node at (2.5,2.5) { $\dots$ };
		\node at (3.5,2.5) { $Q$ };
		\node at (0.5,1.5) { $\vdots$ };
		\node at (1.5,1.5) { $\vdots$ };
		\node at (2.5,1.5) { $\vdots$ };
		\node at (3.5,1.5) { $\vdots$ };
	\end{tikzpicture}_\ca{Q}
\end{equation}
where $Q^{[p,q]}$ can be computed either by a Schur complement on \emph{all} the vertical edges or by successive Schur complements on successive guillotine cuts. This computation can be done on a computer but does not give much insight on the large size properties of $Q^{[p,q]}$ and $\alpha_{p,q}$. 

Up to now, the previous computations on finite rectangles are a way to rewrite Gaussian integration formulae in terms of operadic products with suitable associativity properties and it does not add any algebraic content, excepted that it now shares a common formalism with all the Markov processes on the square lattice. The major operadic step relies in the introduction of suitable boundary spaces and products as explained now.

\subsection{Boundary eigen-elements and their lifts.}
\subsubsection{From general definitions to fixed points on quadratic forms}

The question treated in the present section is the obtention of equations to fix the building blocks $A_a(x)$ and $U_{ab}$ of the boundary weights $g^{(\Lambda)}_{p,q}$ announced in \eqref{eq:boundaryweight:internalstruct}. We must first find spaces with a $\Guill_2$-algebra structure in which to find these objects and then write down equations on these elements.

The following sequence of definitions are directly imported from \cite{Simon} and are adapted to the present case of Gaussian Markov processes, for which we have the following set-up (first notations directly imported from \cite{Simon}):
\begin{itemize}
	\item the $\Guill_2$-algebra $\ca{A}_{\PatternShapes}$ is the one of quasi-densities $L^+_{\PatternShapes}$ with the scalar action $(\lambda,f)\mapsto \alpha f$ and products given by integration over the variables on the cuts as seen above;
	\item the fixed Gaussian semi-group $\MarkovWeight{G}_{\bullet,\bullet}$ in $\ca{A}_{\PatternShapes}$ is the one defined by $\MarkovWeight{G}_{1,1}=\ee_{Q}$ where $Q$ is the elementary fixed face operator that defines the model;
	\item the boundary spaces $(\ca{A}_{p,q})_{(p,q)\in\ov{\PatternShapes}\setminus\PatternShapes}$ \emph{still have to be defined} and will contain the elements $A_a(x)$ and $U_{ab}$ and their products.
\end{itemize} 

\begin{rema}The boundary spaces $(\ca{A}_{p,q})_{(p,q)\in\ov{\PatternShapes}\setminus\PatternShapes}$ do not need to be functional spaces such as $L^+_{p,q}$  and it is not necessarily the case since all the definitions below can be formulated up to morphisms. However, it is natural to suppose a \emph{Gaussian Ansatz} with elements $A_a$ behaving as Gaussian densities $\ee_{Q_a}$ in order for the boundary weights $g^{(\Lambda)}_{p,q}$ to be Gaussian. As it will be seen below, the boundary eigen-elements will correspond to \emph{infinite-dimensional} Gaussian processes without any density with respect to a Lebesgue measure. Defining products on hypothetical spaces $(L^+_{p,q})_{(p,q)\in\ov{\PatternShapes}\setminus\PatternShapes}$ cannot be performed as in \eqref{eq:def:elemguillotprod:dens} due to renormalization effects.
\end{rema}

In order to bypass efficiently these difficulties and have purely algebraic objects in a first step, we first lift all the definitions below from \cite{Simon} to the level of quadratic forms with boundary spaces $(\ca{Q}_{p,q})_{(p,q)\in\ov{\PatternShapes}\setminus\PatternShapes}$ of suitable quadratic forms endowed with non-linear Schur products. We will consider afterwards again the question of densities, products and eigenvalues and see how renormalization enters the construction in the last section~\ref{sec:eigenval}.

The easiest boundary Ansatz is to consider one-dimensional eigen-generators (see \cite{Simon} for the spaces $\ca{V}_\bullet$) with $\ca{V}_{1,\infty_\South}=\setR_+ A_\South$ and $\ca{V}_{p}=\emptyset$ for $p\geq 2$ (and similar choices on the three other boundary sides) and similar one-dimensional spaces $\ca{V}_{\infty_\West,\infty_\South}=\setR_+U_{\South\West}$ on each corner. If the system of equations obtained below using this hypothesis may not find any solution, then more complicated choice should be considered but this is not the case here (for example, we know that, in the present case of Gaussian processes with a positive definite face quadratic form $Q$, the uniqueness of the infinite-volume Gibbs measure is ensured (see \cite{georgii} for example). 

We reproduce here the definitions of \cite{Simon} with notational changes to fit to the present Gaussian case.
It may not be obvious for the reader to recognize at first sight what is fixed from the model and what it is unknown in the following definitions, we then choose to add a top bar for each unknown variable that we will look for.

\begin{defi}[eigen-generator up to morphisms of a 2D-semi-group, from \cite{Simon}]\label{def:eigenalgebrauptomorphims}
	An element $\ov{A}_\South\in \ca{A}_{1,\infty_\South}$ is \emph{a South eigen-$\Guill_1$-generator up to morphisms} of the semi-group $\ee_{Q}^{[\bullet,\bullet]}$ with eigenvalue $\lambda\in\setR_+$ if and only if there exists a collection of linear maps $\Phi_p^{\South,q,r} : \ca{A}_{p,\infty_\South} \to \ca{A}_{p,\infty_\South}$, $p\in\setN^*$, $q\in\setN^*$, $r\in\setN$ such that:
	\begin{enumerate}[(i)]
		\item for any $N\in\setN^*$ and any $p\in\setN^*$, 
		\begin{equation}\label{eq:defeigenSouth:noM}
			\Phi_{p}^{\South,q,0}\left(
			\begin{tikzpicture}[guillpart,yscale=1.25,xscale=2]
				\fill[guillfill] (0,0) rectangle (3,2);
				\draw[guillsep] (0,0)--(0,2)--(3,2)--(3,0) (1,0)--(1,1) (2,0)--(2,1) (0,1)--(3,1);
				\node at (0.5,0.5) {$\ov{A}_\South$};
				\node at (1.5,0.5) {$\dots$};
				\node at (2.5,0.5) {$\ov{A}_\South$};
				\node at (1.5,1.5) {$\ee_{Q}^{[p,q]}$};
			\end{tikzpicture}
			\right)
			= \lambda^{pq} 	\begin{tikzpicture}[guillpart,yscale=1.25,xscale=2]
				\fill[guillfill] (0,0) rectangle (3,1);
				\draw[guillsep] (0,0)--(0,1)--(3,1)--(3,0) (1,0)--(1,1) (2,0)--(2,1) ;
				\node at (0.5,0.5) {$\ov{A}_\South$};
				\node at (1.5,0.5) {$\dots$};
				\node at (2.5,0.5) {$\ov{A}_\South$};
			\end{tikzpicture}
		\end{equation}
		\item for any $N\in\setN^*$, for any $r\in\setN^*$, any $p\in\setN^*$, for any $\MarkovWeight{T}\in \ca{A}_{p,r}$,
		\begin{equation}\label{eq:defeigenSouth:withM}
			\Phi_{p}^{\South,q,r}\left(
			\begin{tikzpicture}[guillpart,yscale=1.25,xscale=2]
				\fill[guillfill] (0,0) rectangle (3,3);
				\draw[guillsep] (0,0)--(0,3)--(3,3)--(3,0) (1,0)--(1,1) (2,0)--(2,1) (0,1)--(3,1) (0,2)--(3,2);
				\node at (0.5,0.5) {$\ov{A}_\South$};
				\node at (1.5,0.5) {$\dots$};
				\node at (2.5,0.5) {$\ov{A}_\South$};
				\node at (1.5,1.5) {$\ee_{Q}^{[p,q]}$};
				\node at (1.5,2.5) {$\MarkovWeight{T}$};
			\end{tikzpicture}
			\right)
			= \lambda^{pq} 	\begin{tikzpicture}[guillpart,yscale=1.25,xscale=2]
				\fill[guillfill] (0,0) rectangle (3,2);
				\draw[guillsep] (0,0)--(0,2)--(3,2)--(3,0) (1,0)--(1,1) (2,0)--(2,1) (0,1)--(3,1);
				\node at (0.5,0.5) {$\ov{A}_\South$};
				\node at (1.5,0.5) {$\dots$};
				\node at (2.5,0.5) {$\ov{A}_\South$};
				\node at (1.5,1.5) {$\MarkovWeight{T}$};
			\end{tikzpicture}
		\end{equation}
	\end{enumerate}
\end{defi}
This definition on South spaces can be replicated \emph{mutatis mutandis} on the four sides by rotating the guillotine partitions or use of the dihedral group. In practice, the lift from $\ov{A}_S "=" \ee_{\ov{Q}_S}$ to a quadratic form corresponds to drop first the scalar eigenvalue part in the previous definitions (as in the passage from $\ee_{Q}\in L^+_{p,q}$ to $Q\in \ca{Q}_{p,q}$ in proposition~\ref{prop:fromweightstoquadraticforms} and corollary~\ref{coro:guillstruct:Q}) and then consider only the notion of \emph{fixed point up to morphisms} for the quadratic forms with Schur products after the lift.

Instead of unknown $L^+$-like boundary spaces, we introduce still-unknown spaces $(\ca{Q}_{p,q})_{(p,q)\in\ov{\PatternShapes}\setminus\PatternShapes}$ in which we will search for the wanted parameters and formulate the following definition, which is central for our computations and is only one compatible with a Gaussian element $\ov{A}_\South := \ee_{\ov{Q}^{[1,\infty_\South]}}$ with $\ov{Q}^{[1,\infty_\South]}\in \ca{Q}_{1,\infty_\South}$. The exponent notation in $\ov{Q}^{[1,\infty_\South]}$ is kept to remind the reader to which space the variable belongs and to distinguish with sub-block extraction as in \eqref{eq:block_decompo_face_weight}.

\begin{defi}[fixed-point up to morphisms]\label{def:fixedpointuptomorphims}
	An element $\ov{Q}^{[1,\infty_\South]}\in \ca{Q}_{1,\infty_\South}$ is \emph{a fixed-point-generator up to morphisms} of the surface product semi-group $Q^{[\bullet,\bullet]}$ if and only if there exists a collection of maps $\phi_p^{\South,q,r} : \ca{Q}_{p,\infty_\South} \to \ca{Q}_{p,\infty_\South}$, $p\in\setN^*$, $q\in\setN^*$, $r\in\setN$ such that:
	\begin{enumerate}[(i)]
		\item for any $N\in\setN^*$ and any $p\in\setN^*$, 
		\begin{equation}\label{eq:deffixedSouth:noM}
			\phi_{p}^{\South,q,0}\left(
			\begin{tikzpicture}[guillpart,yscale=1.2,xscale=2.8]
				\fill[guillfill] (0,0) rectangle (3,2);
				\draw[guillsep] (0,0)--(0,2)--(3,2)--(3,0) (1,0)--(1,1) (2,0)--(2,1) (0,1)--(3,1);
				\node at (0.5,0.5) {$\ov{Q}^{[1,\infty_\South]}$};
				\node at (1.5,0.5) {$\dots$};
				\node at (2.5,0.5) {$\ov{Q}^{[1,\infty_\South]}$};
				\node at (1.5,1.5) {$Q^{[p,q]}$};
			\end{tikzpicture}_{\ca{Q}}
			\right)
			=\begin{tikzpicture}[guillpart,yscale=1.2,xscale=2.8]
				\fill[guillfill] (0,0) rectangle (3,1);
				\draw[guillsep] (0,0)--(0,1)--(3,1)--(3,0) (1,0)--(1,1) (2,0)--(2,1) ;
				\node at (0.5,0.5) {$\ov{Q}^{[1,\infty_\South]}$};
				\node at (1.5,0.5) {$\dots$};
				\node at (2.5,0.5) {$\ov{Q}^{[1,\infty_\South]}$};
			\end{tikzpicture}_{\ca{Q}}
		\end{equation}
		\item for any $N\in\setN^*$, for any $r\in\setN^*$, any $p\in\setN^*$, for any $B\in\ca{Q}_{p,r}$,
		\begin{equation}\label{eq:deffixedSouth:withM}
			\phi_{p}^{\South,q,r}\left(
			\begin{tikzpicture}[guillpart,yscale=1.2,xscale=2.8]
				\fill[guillfill] (0,0) rectangle (3,3);
				\draw[guillsep] (0,0)--(0,3)--(3,3)--(3,0) (1,0)--(1,1) (2,0)--(2,1) (0,1)--(3,1) (0,2)--(3,2);
				\node at (0.5,0.5) {$\ov{Q}^{[1,\infty_\South]}$};
				\node at (1.5,0.5) {$\dots$};
				\node at (2.5,0.5) {$\ov{Q}^{[1,\infty_\South]}$};
				\node at (1.5,1.5) {$Q^{[p,q]}$};
				\node at (1.5,2.5) {$B$};
			\end{tikzpicture}_{\ca{Q}}
			\right)
			= 	\begin{tikzpicture}[guillpart,yscale=1.2,xscale=2.8]
				\fill[guillfill] (0,0) rectangle (3,2);
				\draw[guillsep] (0,0)--(0,2)--(3,2)--(3,0) (1,0)--(1,1) (2,0)--(2,1) (0,1)--(3,1);
				\node at (0.5,0.5) {$\ov{Q}^{[1,\infty_\South]}$};
				\node at (1.5,0.5) {$\dots$};
				\node at (2.5,0.5) {$\ov{Q}^{[1,\infty_\South]}$};
				\node at (1.5,1.5) {$B$};
			\end{tikzpicture}_{\ca{Q}}
		\end{equation}
	\end{enumerate}
\end{defi}

Previous definition~\ref{def:eigenalgebrauptomorphims} deals with only one side and one of the objects $\ov{A}_a$ and is similar to the one-dimensional definition of eigenvectors with the addition of the notion of morphisms. A quick correspondence for the reader lost in the operadic notations is as follows: we consider a Gaussian transition kernel on the 1D lattice $\setZ$ with quadratic form $Q$. Definition~\ref{def:eigenalgebrauptomorphims} corresponds to the eigenvalue equation for the kernel
\[
\int_\ca{H} f(u) e^{-Q((u,v),(u,v))/2}  du = \Lambda f(v) 
\]
The lift corresponds to the Gaussian Ansatz $f(u)= e^{-q(u)/2}$ and then induces the eigenvalue-free fixed-point equation 
\[
Q_{22}-Q_{21} (Q_{11}+q)^{-1} Q_{12} = q
\]
where we recognize a Schur complement on the left. Such equations are studied in great details in \cite{Bodiot} and are reminded in section \ref{sec:dimone:reminder}.

The two-dimensional geometry introduces additional morphisms since the transverse directions is endowed with associative products. It also adds the notion of corners, to which we associate the elements $U_{ab}$. The eigen-element property on the corner adds two types of algebraic constraints: the half-strip elements $A_a$ and $A_b$ satisfying both definition~\ref{def:eigenalgebrauptomorphims} on the two sides adjacent to the corner have to be consistent in the sense of "left (or right)-extended system of eigen-generators" and "corner system of eigen-semi-groups" as introduced in \cite{Simon}. We do not reproduce here the full definitions with the eigenvalues of \cite{Simon} but only their lifts as fixed points at the level of quadratic forms. 

\begin{defi}[left-extended fixed point up to morphisms]\label{def:fixedpointuptomorph:leftextended}
	In addition to the semi-group $Q^{[\bullet,\bullet]}$, we also consider a fixed element $R_\West \in \ca{Q}_{\infty_\West,1}$ with Schur powers $R_\West^{[q]}$, $q\in\setN^*$. A couple $(\ov{Q}^{[1,\infty_\South]},\ov{Q}^{[\infty_\West,\infty_\South]})$ in $\ca{Q}_{1,\infty_\South}\times\ca{Q}_{\infty_\West,\infty_\South}$ is a left-extended fixed point up to morphisms of $( Q^{[\bullet,\bullet]}  ,  R_\West^{[\bullet]} )$ if and only if there exists a collection of morphisms $\phi_p^{\South,q,r}: \ca{Q}_{p,\infty_\South}\to\ca{Q}_{p,\infty_\South}$, $p\in\setN^*\cup\{\infty_\West\}$, $q\in\setN^*$, $r\in\setN$ such that:
	\begin{enumerate}[(i)]
		\item $\ov{Q}^{[1,\infty_\South]}$ is an fixed-point-generator up to morphisms of $Q^{[\bullet,\bullet]}$ up to the morphisms $\phi_p^{\South,q,r}$ with finite $p\in\setN^*$.
		
		\item for any $N\in\setN$, for any $r\in\setN^*$, any $p\in\setN$, for any $B_\West\in \ca{Q}_{\infty_\West,r}$,
		\begin{equation}\label{eq:cornermorphism:fixed:def}
			\begin{split}&\phi_{\infty_\West}^{
					\South,q,r}\left(
				\begin{tikzpicture}[guillpart,yscale=1.2,xscale=3.5]
					\fill[guillfill] (0,0) rectangle (4,3);
					\draw[guillsep] (4,0)--(4,3)--(0,3) (0,1)--(4,1) (0,2)--(4,2) (1,0)--(1,2) (2,0)--(2,1) (3,0)--(3,1);
					\node at (0.5,0.5) {$\ov{Q}^{[\infty_\West,\infty_\South]}$};
					\node at (1.5,0.5) {$\ov{Q}^{[1,\infty_\South]}$};
					\node at (2.5,0.5) {$\dots$};
					\node at (3.5,0.5) {$\ov{Q}^{[1,\infty_\South]}$};
					\node at (0.5,1.5) {$R_\West^{[q]}$};
					\node at (2.5,1.5) {$Q^{[p,q]}$};
					\node at (2,2.5) {$B_\West$};
				\end{tikzpicture}_\ca{Q}
				\right)
				\\
				&=
				\begin{tikzpicture}[guillpart,yscale=1.2,xscale=3.5]
					\fill[guillfill] (0,0) rectangle (4,2);
					\draw[guillsep] (4,0)--(4,2)--(0,2) (0,1)--(4,1) (1,0)--(1,1) (2,0)--(2,1) (3,0)--(3,1);
					\node at (0.5,0.5) {$\ov{Q}^{[\infty_\West,\infty_\South]}$};
					\node at (1.5,0.5) {$\ov{Q}^{[1,\infty_\South]}$};
					\node at (2.5,0.5) {$\dots$};
					\node at (3.5,0.5) {$\ov{Q}^{[1,\infty_\South]}$};
					\node at (2,1.5) {$B_\West$};
				\end{tikzpicture}_\ca{Q}
			\end{split}
		\end{equation}
	\end{enumerate}
\end{defi}
Given an element $\ov{Q}^{[1,\infty_\South]}$ (lift of $\ov{A}_S$), we note its $p$-th Schur power $\ov{Q}^{[p,\infty_\South]}$ to make notations easier: the reader must however remember that all these objects are generated by a single element given for $p=1$. In the definition below, $\ov{Q}^{[\infty_\West,1]}$ is the lift of $\ov{A}_\West$ and $\ov{Q}^{[\infty_\West,\infty_\South]}$ is the lift of $U_{\South\West}$.
\begin{defi}[corner system of fixed points]\label{def:fixedpointcorner}
	A South-West corner system of fixed points of a 2D-semi-group $Q^{[\bullet,\bullet]}$ up to morphisms  is a triplet
	$(\ov{Q}^{[\infty_\West,1]},\ov{Q}^{[1,\infty_\South]},\ov{Q}^{[\infty_\West,\infty_\South]})$
	such that:
	\begin{enumerate}[(i)]
		\item $(\ov{Q}^{[1,\infty_\South]},\ov{Q}^{[\infty_\West,\infty_\South]})$ defines a left-extended system of fixed points of $( Q^{[\bullet,\bullet]}  , \ov{Q}^{[\infty_\West,\bullet]} )$ up to morphisms $\phi^{\South,q,r}_p$.
		\item $(\ov{Q}^{[\infty_\West,1]},\ov{Q}^{[\infty_\West,\infty_\South]})$ defines a left-extended system of fixed points of $( Q^{[\bullet,\bullet]}  ,  \ov{Q}^{[\bullet,\infty_\South]} )$  up to morphisms $\phi^{\West,p,r}_q$.
	\end{enumerate}
\end{defi}
The three other corners provide three similar definitions \emph{mutatis mutandis}. The following steps are the construction (up to morphisms) of the spaces $(\ca{Q}_{p,q})_{(p,q)\in\ov{\PatternShapes}\setminus\PatternShapes}$ with the suitable eigen-elements morphisms and the transformations of the last definitions into concrete equations.

\subsubsection{Natural boundary spaces of quadratic forms and their structures}

\paragraph{Definitions}
In order to satisfy definitions~\ref{def:fixedpointuptomorphims}, \ref{def:fixedpointuptomorph:leftextended} and \ref{def:fixedpointcorner}, we must first specify, up to morphisms, the spaces $\ca{Q}_{p,q}$ on the boundaries. In the Gaussian Ansatz after lifts to quadratic forms, the most natural choice is to assume the existence of Hilbert spaces $\ca{W}_{a}$ with $a\in\{\South,\North,\West,\East\}$ associated to infinite half-lines in the corresponding direction $a$ such that $\ca{Q}_{p,q}$ is a quadratic form on the orthogonal sum of the Hilbert spaces associated to its boundary sides in the same way as in \eqref{eq:def:Qpq:rect} when $p$ and $q$ are both finite. For example, an element $Q\in \ca{Q}_{p,\infty_\South}$ is expected to be an operator on $\ca{H}_1^p\oplus\ca{W}_\South^2$ since a North half-strip has one South side of length $p$ and to half-lines in the South direction.

In order to ensure that the operadic guillotine products in definition \ref{def:fixedpointuptomorphims} are defined and expressed by Schur products (see sections~\ref{sec:complementary:Schur_Complements} and \ref{sec:dimone:reminder}), the operators in $\ca{Q}_{p,\infty_\South}$ are required to be self-adjoint and positive definite at least. However, this hypothesis in \eqref{eq:def:Qpq:rect} for finite-dimensional spaces is not sufficient to encompass the case of infinite-dimensional spaces $\ca{W}_{a}$. We therefore directly introduce boundary spaces with sufficient hypothesis to ensure a well-defined operadic structure, for all $a\in\{\West,\East\}$ and $b\in\{\South,\North\}$:
\begin{subequations}
	\label{eq:hs:Qspaces}
	\begin{align}
		\ca{Q}_{\infty_a,q} &= \left\{ Q\in \ca{B}(\ca{W}_a^2\oplus\ca{H}_2^q); Q=Q^*\text{ and }\inf_{x\neq 0} \frac{\scal{x}{Qx}}{\scal{x}{x}} > 0 \right\}
		\\
		\ca{Q}_{p,\infty_b} &= \left\{ Q\in \ca{B}(\ca{H}_1^{p}\oplus\ca{W}_b^2); Q=Q^*\text{ and }\inf_{x\neq 0} \frac{\scal{x}{Qx}}{\scal{x}{x}} > 0 \right\}
	\end{align}	
	and for corners the following spaces:
	\begin{align}	\label{eq:corner:Qspaces}
		\ca{Q}_{\infty_a,\infty_b} &= \left\{ Q\in \ca{B}(\ca{W}_a\oplus\ca{W}_b); Q=Q^*\text{ and }\inf_{x\neq 0} \frac{\scal{x}{Qx}}{\scal{x}{x}} > 0 \right\}
	\end{align}	
\end{subequations}
The existence of products
\[
\begin{tikzpicture}[guillpart,yscale=1,xscale=1]
	\fill[guillfill] (0,0) rectangle (2,2);
	\draw[guillsep] (0,0)--(2,0)--(2,2)--(0,2) (0,1)--(2,1);
	\node at (1,0.5) {$Q_1$};
	\node at (1,1.5) {$Q_2$};
\end{tikzpicture}
\text{\qquad and\qquad}
\begin{tikzpicture}[guillpart,yscale=1,xscale=1]
	\fill[guillfill] (0,0) rectangle (4,1);
	\draw[guillsep] (0,0)--(4,0)--(4,1)--(0,1) (2,0)--(2,1);
	\node at (1,0.5) {$Q_1$};
	\node at (3,0.5) {$Q$};
\end{tikzpicture}
\]
for $Q_i\in \ca{Q}_{\infty_\West,q_i}$ and $Q\in \ca{Q}_{p,q_1}$ with suitable associativity properties on the $\Guill_2$-operad requires these constraints on the the quadratic forms and a second shift requirements on the spaces $\ca{W}_a$ described below.

The fixed points of definitions~\ref{def:fixedpointuptomorphims} and \ref{def:fixedpointcorner} thus have the following structures in the block notations~\eqref{eq:block_decompo_partielles} (we have drawn them in the corresponding shape in order to prepare the products associated to the guillotine partitions):
\begin{subequations}
	\label{eq:fixedpoints:blocknotations}
	\begin{align}
		\begin{tikzpicture}[guillpart,yscale=1.4,xscale=1.4]
			\fill[guillfill] (0,0) rectangle (2,1);
			\draw[guillsep] (0,0)--(0,1)--(2,1)--(2,0);
			\node at (1,0.5) {$\ov{Q}^{[1,\infty_\South]}$};
		\end{tikzpicture}
		&= \begin{pmatrix} 
			\ov{Q}^{[1,\infty_\South]}_{\North,\North} & \ov{Q}^{[1,\infty_\South]}_{\North,[\West\East]}
			\\
			\ov{Q}^{[1,\infty_\South]}_{[\West\East],\North} & \ov{Q}^{[1,\infty_\South]}_{[\West\East],[\West\East]}
		\end{pmatrix}
		&
		\begin{tikzpicture}[guillpart,yscale=1.4,xscale=1.4]
			\fill[guillfill] (0,0) rectangle (2,1);
			\draw[guillsep] (0,0)--(2,0)--(2,1)--(0,1);
			\node at (1,0.5) {$\ov{Q}^{[\infty_\West,1]}$};
		\end{tikzpicture}
		&= \begin{pmatrix}
			\ov{Q}^{[\infty_\West,1]}_{[\South\North],[\South\North]}
			& 
			\ov{Q}^{[\infty_\West,1]}_{[\South\North],\East}
			\\
			\ov{Q}^{[\infty_\West,1]}_{[\East,[\South\North]}
			&
			\ov{Q}^{[\infty_\West,1]}_{\East,\East}
		\end{pmatrix}
		\\
		\begin{tikzpicture}[guillpart,yscale=1.5,xscale=1.5]
			\fill[guillfill] (0,0) rectangle (2,1);
			\draw[guillsep] (0,1)--(2,1)--(2,0);
			\node at (1,0.5) {$\ov{Q}^{[\infty_\West,\infty_\South]}$};
		\end{tikzpicture} &=
		\begin{pmatrix}
			\ov{Q}^{[\infty_\West,\infty_\South]}_{\North,\North}
			&
			\ov{Q}^{[\infty_\West,\infty_\South]}_{\North,\East}
			\\
			\ov{Q}^{[\infty_\West,\infty_\South]}_{\East,\North}
			&
			\ov{Q}^{[\infty_\West,\infty_\South]}_{\East,\East}
		\end{pmatrix}
	\end{align}
\end{subequations}

\paragraph{Operadic structure of boundary spaces}

We detail here the construction of the products and then prove the $\Guill_2$-structure of the boundary spaces. Given a guillotine partition $\rho$ of a pattern type $D$ of $\setZ^2$ with shapes $D_i$, $1\leq i\leq n$. Each shape $D_i$ has $b_i\leq 4$ boundaries $B_{i,k}$, that are segments of half-lines. We now consider all the possible intersections of $B_{i,k}$ (there are either empty or lines or half-lines or segments) and label them $e_j$, $1\leq j\leq m$. We assume that each elementary boundary shape $e\in E_j$ is decorated by a Hilbert space $\ca{V}_e$ and each face is decorated by a linear map $Q_i$ acting on $\oplus_{e\subset \partial D_i} \ca{V}_e$. As an example, we consider
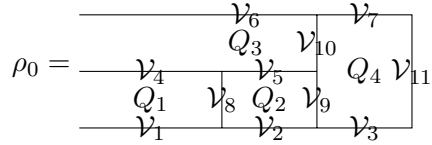
\[
\rho_0 = \begin{tikzpicture}[guillpart,yscale=1.5,xscale=2.5]
	\draw (-0.5,0)--(3,0)--(3,2)--(-.5,2) (2,0)--(2,2) (-0.5,1)--(2,1) (1,0)--(1,1);
	\node at (0.25,0) {$\ca{V}_1$};
	\node at (1.5,0) {$\ca{V}_2$};
	\node at (2.5,0) {$\ca{V}_3$};
	\node at (0.25,1) {$\ca{V}_4$};
	\node at (1.5,1) {$\ca{V}_5$};
	\node at (1.25,2) {$\ca{V}_6$};
	\node at (2.5,2) {$\ca{V}_7$};
	\node at (1,0.5) {$\ca{V}_8$};
	\node at (2,1.5) {$\ca{V}_{10}$};
	\node at (2,0.5) {$\ca{V}_9$};
	\node at (3,1) {$\ca{V}_{11}$};
	\node at (0.25,0.5) {$Q_1$};
	\node at (1.5,0.5) {$Q_2$};
	\node at (1.25,1.5) {$Q_3$};
	\node at (2.5,1) {$Q_4$};
\end{tikzpicture}
\]
We introduce the linear map 
\begin{align*}
	j_\rho : \oplus_{i} \End\left(
	\oplus_{e\subset \partial F_i} \ca{V}_e
	\right) &\to \End\left(
	\oplus_{1\leq j\leq m} \ca{V}_{e_j}
	\right)
	&
	\begin{pmatrix}
		Q_1 \\ \vdots \\ Q_{|F_\rho|}
	\end{pmatrix} & \mapsto J_\rho\begin{pmatrix}
		Q_1 \\ \vdots \\ Q_{|F_\rho|} \end{pmatrix}
\end{align*}
where $J_\rho$ is the rectangular matrix of size $m \times (\sum_{j} b_j)$ with identities $\id_{\ca{H}_e}$ whenever an edge $e$ belongs to $\partial F$ and zeroes elsewhere. We now consider the subspace associated to internal edges $\ca{V}^\text{in} =\oplus_{e\text{ inner }} \ca{V}_e$, i.e. edges that belong to exactly two faces. In the previous example, we have
\[
J_{\rho_0} = \begin{pmatrix}
	1 & 0 & 0 & 1 & 0 & 0 & 0 &1 & 0 & 0 &0 \\
	0 & 1 & 0 & 0 & 1 & 0 & 0 &1 & 1 & 0 &0 \\ 
	0 & 0 & 0 & 1&1&1&0&0&0&1&0\\
	0&0&1&0&0&0&1&0&1&1&1
\end{pmatrix}
\]
Whenever the Schur product is well-defined, we finally define the pre-product associated to the guillotine partition:
\begin{equation}\label{eq:preproduct}
	m'_\rho( Q_1,\ldots, Q_{|F_\rho|} ) = \mathrm{Schur}_{\ca{V}^\text{in}}\circ j_\rho(Q_1,\ldots, Q_{|F_\rho|})
\end{equation} 
which is now on operator on $\ca{V}^{\text{out}}=\oplus_{e\text{ outer }} \ca{V}_e$. In the previous example $\rho_0$ the inner edges are $e_j$ with $j\in\{4,5,8,9,10\}$ and the outer edges are the $e_j$ with $j\in\{1,2,3,6,7,11\}$. 

The pre-product $m'_\rho$ are \emph{not yet} the operadic products for an interesting reason which is related to the colours of the operads and plays a role for boundary spaces. 

The spaces $\ca{Q}_{p,q}$ act on spaces associated to each side of the boundary of a rectangular shape and does not see the full guillotine partition and hence to not distinguish all the edge spaces $\ca{V}_e$. In the example $\rho_0$ above, $Q_4 \in \ca{Q}_{p_4,q_4}$ acts on $\ca{V}_3=\ca{V}_7=\ca{H}_1^{p_4}$ and $\ca{V}_{11}=\ca{H}_2^{q_4}$ and a space $\ca{V}'=\ca{H}_2^{q_4}$ on the West, which replaces the two spaces $\ca{V}_{9}$ and $\ca{V}_{10}$ inherited from the global structure of $\rho_0$ (and not only $D_4$). Whenever the lengths are finite, there is a canonical isomorphism $ \ca{H}^{P}\simeq \ca{H}^{p_1}\oplus\ldots\oplus\ca{H}^{p_k}$ with $P=p_1+\ldots+p_k$ so that any $Q\in\ca{Q}_{p,q}$ is extended to the suitable spaces associated to the edges $E_j$ on each side. This corresponds to the trivial identification already made implicitly in the Gaussian computations of property~\ref{prop:fromweightstoquadraticforms} and corollary~\ref{coro:guillstruct:Q}.

However, whenever a size is infinite, such a simple isomorphism is absent \emph{a priori} and has to be added. In the example $\rho_0$, an element $Q_3\in\ca{Q}_{\infty_\West,q}$ acts on $\ca{W}_\West^2\oplus\ca{H}_2^q$ but the first space $\ca{W}_\West$ has to be identified with $\ca{V}_4\oplus\ca{V}_5$ with $\ca{V}_4=\ca{W}_\West$ and $\ca{V}_{5}=\ca{H}_1^{p_2}$. Considering any guillotine  partitions and associativity requirements (see below) and the hierarchical structure of guillotine partitions shows that the half-line spaces $\ca{W}_a$ need to satisfy the following minimal definition.

\begin{defi}[shift property]\label{def:shiftprop}
	Let $\ca{H}$ be a Hilbert space. A Hilbert space $\ca{W}$ has a left (resp. right) $\ca{H}$-shift property if there exists a collection of bijective isometries $(D_p)_{p\in\setN}$ with $D_p : \ca{W}\oplus\ca{H}^p \to \ca{W}$ (resp. $D_p : \ca{H}^p\oplus \ca{W} \to \ca{W}$) such that, for all $p_1,p_2\geq 0$,
	\[
	D_{p_2}\circ (D_{p_1}\oplus\id_{\ca{H}^{p_2}}) = D_{p_1+p_2}
	\]
	(resp. $D_{p_2}\circ (\id_{\ca{H}^{p_2}}\oplus D_{p_1}) = D_{p_1+p_2}$).
\end{defi}
These operators $D_p$ are introduced so that identifications
\[
\begin{tikzpicture}[scale=1,baseline=0.cm]
	\draw[-o] (0,0)--(1,0) ;
	\draw[-o] (1,0)--(2,0) ;
	\draw[-o] (2,0)--(3,0) ;
	\node at (0.5,0) {$\ca{W}_{\West}$} ;
	\node at (1.5,0) {$\ca{H}_{1}^p$} ;
	\node at (2.5,0) {$\ca{H}_{1}^{p'}$} ;
\end{tikzpicture}
\qquad \to\qquad
\begin{tikzpicture}[scale=1,baseline=0.cm]
	\draw[-o] (0,0)--(1,0) ;
	\draw[-o] (1,0)--(2,0) ;
	\node at (0.5,0) {$\ca{W}_{\West}$} ;
	\node at (1.5,0) {$\ca{H}_{1}^p$} ;
\end{tikzpicture}
\qquad \to\qquad 
\begin{tikzpicture}[scale=1,baseline=0.cm]
	\draw[-o] (0,0)--(1,0) ;
	\node at (0.5,0) {$\ca{W}_{\West}$} ;
\end{tikzpicture}
\] 
can be performed sequentially in guillotine partitions such as $\rho_0$ above and will be useful for fixed point equations.

In order to be fully rigorous and compatible with \cite{Simon}, we also need to introduce line spaces $\ca{W}_{\West\East}$ and $\ca{W}_{\South\North}$ in order to deal with guillotine partitions with lengths $\infty_{LR}$ and pointings. More details are given in section~\ref{sec:moreonguill}: the reader interested only in the fixed point equations can skip them in a first time.

We also introduce the associated spaces of quadratic forms, for any $p,q\in\setN_1$, $a\in\{\infty_\West,\infty_\East\}$ and $b\in\{\infty_\South,\infty_\North\}$.
\begin{subequations}
	\label{eq:strip:Qspaces}
	\begin{align}
		\ca{Q}_{\infty_{\West\East},q} &= \left\{ Q\in \ca{B}(\ca{W}_{\West\East}^2); Q=Q^*\text{ and }\inf_{x\neq 0} \frac{\scal{x}{Qx}}{\scal{x}{x}} > 0 \right\}
		\\
		\ca{Q}_{\infty_{\West\East},\infty_b} &= \left\{ Q\in \ca{B}(\ca{W}_{\West\East}); Q=Q^*\text{ and }\inf_{x\neq 0} \frac{\scal{x}{Qx}}{\scal{x}{x}} > 0 \right\}
	\end{align}
	\begin{align}
		\ca{Q}_{p,\infty_{\South\North}} &= \left\{ Q\in \ca{B}(\ca{W}_{\South\North}^2); Q=Q^*\text{ and }\inf_{x\neq 0} \frac{\scal{x}{Qx}}{\scal{x}{x}} > 0 \right\}
		\\	
		\ca{Q}_{\infty_a,\infty_{\South\North}} &= \left\{ Q\in \ca{B}(\ca{W}_{\South\North}); Q=Q^*\text{ and }\inf_{x\neq 0} \frac{\scal{x}{Qx}}{\scal{x}{x}} > 0 \right\}
	\end{align}
\end{subequations} 

\paragraph{Extended guillotine structure}

\begin{theo}
	Let $\ca{W}_{\South}$, $\ca{W}_\North$, $\ca{W}_{\West}$ and $\ca{W}_\East$ be spaces endowed respectively with left $\ca{H}_2$-shifts, right $\ca{H}_2$-shifts, left $\ca{H}_1$-shifts and right $\ca{H}_1$-shifts and let $\ca{W}_{\South\North}$ and $\ca{W}_{\West\East}$ be endowed with the corresponding shift-with-pairing property. 
	
	For any given guillotine partition $\rho$ with external shape and internal shapes $(p_i,q_i)_{1\leq i\leq n}$ and any sequence of elements $Q_i\in\ca{Q}_{p_i,q_i}$, we note $Q'_i$ the element $Q_i$ applied on the elementary intervals on the boundaries identified to the whole boundary side using the shifts $D^a_\bullet$, $a\in\{\South,\North,\West,\East\}$ and and the pairings. The product $m_\rho(Q_1,\ldots,Q_n)$ is defined by applying $m'_\rho$ defined in \eqref{eq:preproduct} to the elements $Q_i'$.
	
	The spaces $(\ca{Q}_{p,q})_{(p,q)\in\ov{\PatternShapes}}$ endowed with the previous Schur products form a $\Guill_2$-algebra.
\end{theo}

The most interesting part in this theorem is that the operadic structure on the boundary \emph{requires} the shift structure $D_\bullet$ on the boundaries in order to be well-defined. We now exploit this structure to write down a fixed point equations for the elements 	\eqref{eq:fixedpoints:blocknotations} and their homologues on the other sides and corners. 

The proof is postponed to section~\ref{sec:proofguillot} for the reader interested in the operadic details.

In order to make things clearer and keep track of the structures that are required, all the products that require the use of the shift operators will be written with an additional exponent $D$, as in 
\begin{equation}	
	\begin{tikzpicture}[guillpart,yscale=1.5,xscale=1.5]
		\fill[guillfill] (0,0) rectangle (2,1);
		\draw[guillsep] (0,0)--(2,0)--(2,1)--(0,1) (1,0)--(1,1);
		\node at (0.5,0.5) {$Q_1$};
		\node at (1.5,0.5) {$Q_2$};
	\end{tikzpicture}_{\ca{Q}}^{D} 
	=
	(D^\West_{p}\oplus D^\West_{p}\oplus\id_{\ca{H}_2^q})
	\mathrm{Schur}_{\ca{H}_2^q}\left(
	j_{\begin{tikzpicture}[guillpart,yscale=0.45,xscale=0.5]
			\fill[guillfill] (0,0) rectangle (2,1);
			\draw[guillsep] (0,0)--(2,0)--(2,1)--(0,1) (1,0)--(1,1);
	\end{tikzpicture}}(Q_1,Q_2)
	\right)
	(D^L_{p}\oplus D^L_{p}\oplus\id_{\ca{H}_2^q})^{-1}
\end{equation}

\paragraph{Canonical realization of the half-line spaces}

Among all the spaces with the shift property from definition~\ref{def:shiftprop} and the pairing of definition~\ref{def:shiftpairingprop}, there are canonical examples associated to any Hilbert space $\ca{H}$ given by 
\begin{subequations}
	\label{eq:canonicalhalflinespaces}
	\begin{align}
		\ca{W}_{L}(\ca{H})&=l^2(\setZ_{<0};\ca{H})
		&
		\ca{W}_R(\ca{H})&=l^2(\setZ_{\geq 0},\ca{H})
		\\
		\ca{W}_{LR}(\ca{H})&=l^2(\setZ,\ca{H})
	\end{align}
\end{subequations}
with the following definitions of shifts and pairings, for any $w_L\in\ca{W}_L(\ca{H})$, $w_R\in\ca{W}_R(\ca{H})$ and $h\in\ca{H}^p$,
\begin{align*}
	\left(D_p^L(w_L,h)\right)(k) &= \begin{cases}
		h_{k+1+p} & \text{for $-p\leq k<0$} \\
		w_L(k+p) & \text{for $k<-p$}
	\end{cases}
	\\
	\left(D^R_p(h,w_R)\right)(k) &= \begin{cases}
		h_{k+1} & \text{for $0\leq k < p$} \\
		w_R(k-p) & \text{for $k\geq p$}
	\end{cases}
	\\
	\left(h(w_L,w_R)\right)(k)&= \begin{cases}
		w_L(k) & \text{for $k<0$} \\
		w_R(k) & \text{for $k\geq 0$}
	\end{cases}
\end{align*}
\begin{lemm}
	The spaces $\ca{W}_L(\ca{H})$, $\ca{W}_R(\ca{H})$ and $\ca{W}_{LR}(\ca{H})$ endowed with $D^L_\bullet$, $D^R_\bullet$ and $h$ satisfy definitions~\ref{def:shiftprop} and \ref{def:shiftpairingprop}.
\end{lemm}
Again this is only a possibility among many others, which provides suitable solutions for an infinite volume Gibbs measures. All the definitions are up to isomorphisms so there are many alternative candidates equivalent to this one, but for which the definition of the shifts may more obscure. An interesting question would be to classify all of them, at least all the ones that provides well-defined fixed point equations as given below with a non-trivial solution. The spaces above do not keep track of any information at infinity; however, for models with phase transitions or long-range order, such candidates are probably not valid.

\subsubsection{Fixed point equations on half-strips}\label{sec:fixedpointeqs}

We now consider the four spaces
\begin{subequations}
	\label{eq:halflinechoices}
	\begin{align}
		\ca{W}_\South &= \ca{W}_L(\ca{H}_2) 
		&
		\ca{W}_\North &= \ca{W}_R(\ca{H}_2) 	
		\\
		\ca{W}_\West &= \ca{W}_L(\ca{H}_1) 
		&
		\ca{W}_\East &= \ca{W}_R(\ca{H}_1) 	
		\\
		\ca{W}_{\South\North} &= \ca{W}_{LR}(\ca{H}_2)
		&
		\ca{W}_{\West\East} &= \ca{W}_{LR}(\ca{H}_1)
	\end{align}
\end{subequations}
with the space of quadratic forms \eqref{eq:hs:Qspaces} and \eqref{eq:corner:Qspaces} and see how these explicitly constructed spaces from the Gaussian Ansatz turn definitions~\ref{def:fixedpointuptomorphims} and \ref{def:fixedpointcorner} into concrete equations. 

\paragraph{Equations with trivial morphisms on the West side.}
\begin{theo}\label{theo:westfixedpoint}
	Let $\ov{Q}^{[\infty_\West,1]}$ be an element in $\ca{Q}_{\infty_\West,1}$ with short block notations
	\begin{align}
		\label{eq:hsfixed:blocknotations}
		\ov{Q}^{[\infty_\West,1]} &= \begin{pmatrix}
			B^\West_{[\South\North],[\South\North]} & B^\West_{[\South\North],\East} 
			\\
			B^\West_{\East,[\South\North]} &
			B^\West_{\East,\East}
		\end{pmatrix}
	\end{align}
	The quadratic form $\ov{Q}^{[\infty_\West,1]}$ is the West fixed point of $Q$ (with \emph{identity morphisms} $\phi_p^{\West,q,r}=\id$)
	if and only if, using $K=B^\West_{\East,\East}+Q_{\West,\West}$, we have for any $a,b\in\{\South,\North\}$
	\begin{subequations}
		\label{eq:Westfixedpoint:blocks}
		\begin{align}
			B^\West_{\East,\East} &= Q_{\East\East} - Q_{\East\West} K^{-1} Q_{\West\East}
			\label{eq:Westfixedpoint:1DEast}
			\\
			B^\West_{a,\East} &= D^{L}_{1} \begin{pmatrix}
				-B^\West_{a,\East} K^{-1} Q_{\West\East}
				\\
				Q_{a\East} - Q_{a\West}K^{-1} Q_{\West\East}
			\end{pmatrix}
			\label{eq:Westfixedpoint:transversal}
			\\
			B^\West_{a,b} &=
			D^L_{1}\left[
			\begin{pmatrix}
				B^\West_{a,b} & 0
				\\
				0 & Q_{ab}
			\end{pmatrix}
			-\begin{pmatrix}
				B^\West_{a,\East} \\ Q_{a\West}
			\end{pmatrix}
			K^{-1}
			\begin{pmatrix}
				B^\West_{\East b}& Q_{\West b}
			\end{pmatrix}
			\right](D^L_{1})^* 
			\label{eq:Westfixedpoint:SNtofold}
		\end{align}
	\end{subequations}
\end{theo}
The proof is a direct consequence of the shift structure. 

An important point is the triviality of the morphisms $\phi$ here in the definition. Firstly, considering identity morphisms is allowed by the definition but it may not necessarily lead to the existence of solutions to equations~\eqref{eq:Westfixedpoint:blocks}. Secondly, there are somewhat hidden behind the choice~\eqref{eq:halflinechoices} of the space $\ca{W}_\West$ with its shift maps $D^L_p$. As explained in \cite{Simon}, constructing suitable boundary spaces is not trivial excepted through disjoint unions, which do not lead to any closed system of equations: there is always a need for morphisms and/or renormalization choices in order to close the system of equations: this is done here by the simple $l^2$ Ansatz \eqref{eq:canonicalhalflinespaces} and, despite its apparent simplicity due to the triviality of the phase diagram of the Gaussian model, most of the work is done here. We however expect such simple Ans\"atze to be insufficient for models with non-trivial phases with long range order. Thirdly, the spaces $\ca{W}_a$ are defined only up to morphisms and, in other representations than \eqref{eq:halflinechoices}, the morphisms are less trivial (even if the Fourier space described below). It would very interesting to have a general algebraic theory to classify all the possible spaces with their morphisms.

The same type of proposition holds trivially on the East side in the second South-North direction with the suitable indices and shifts.

\paragraph*{Solving the equations \eqref{eq:Westfixedpoint:blocks}.}

We now see how the algebraic equations can be solved explicitly in the correct order.

First, the non-linear equation \eqref{eq:Westfixedpoint:1DEast} involves only the block $B^\West_{\East,\East}$ (on the l.h.s. and through $K$ in the r.h.s.), which is a finite-dimensional matrix. This equation is a Schur fixed point for a Gaussian Markov field on the one-dimensional lattice $\setZ$ (see below) and has been studied in details in \cite{Bodiot}.

The second step corresponds to the solution of \eqref{eq:Westfixedpoint:transversal}, which corresponds to a recursion. The block operator $B^\West_{a,\East}$ is a map $\ca{H}_2\to \ca{W}_\West$ and we define $(\gamma^{(a)}_k)_{k<0}$ as the unique sequence of operators $\ca{H}_2\to\ca{H}_1$ such that, for any $k<0$, $u_1\in\ca{H}_1$ and $u_2\in\ca{H}_2$, 
\[
\scal{u_1}{\gamma^{(a)}_k u_2}_{\ca{H}_1} = \scal{ u_1\indic{k} }{ B^\West_{a,\East} u_2 }_{\ca{W}_\West}
\]
\begin{lemm}\label{lemm:hs:recursion1}
	Given $B^\West_{\East,\East}$, the block $B^\West_{a,\East}$ for $a\in\{\South,\North\}$ is a solution of \eqref{eq:Westfixedpoint:transversal} if and only if the sequence $(\gamma^{(a)}_k)_{k<0}$ satisfies, with $K=B_{\East,\East}^{\West} + Q_{\West,\West}$,
	\begin{align*}
		\gamma^{(a)}_{-1} &=  Q_{a\East} - Q_{a\West} K^{-1} Q_{\West\East}
		&
		\gamma^{(a)}_{k-1} &= - \gamma^{(a)}_k K^{-1} Q_{\West\East}
	\end{align*}
	for $k<-1$, whose solution is a geometric sequence.
\end{lemm}

The other blocks of $\ov{Q}^{[\infty_\West,1]}$ are infinite-dimensional but can be indexed by $\setZ_{<0}$. We have a trival isomorphism $\ca{W}_\West^2 \simeq l^2(\setZ_{<0})\otimes \ca{H}_1^2$ through $(u \indic{k},0)\simeq \indic{k} \otimes (u,0)$ and $(0,u \indic{k})\simeq \indic{k} \otimes (0,u)$. The operator  $B^\West_{[\South\North],[\South\North]}\in \End(\ca{W}_\West^2)$ is then characterized by the unique sequence of operators $(\beta_{k,l})_{k,l<0}$ in $\End(\ca{H}_1^2)$ such that, for all $u,v\in \ca{H}_1^2$ and $k,l<0$,
\begin{equation}\label{eq:def:blockindicnotations}
	\scal{\begin{pmatrix}u_1\\u_2\end{pmatrix}}{\beta_{k,l}\begin{pmatrix}v_1\\v_2\end{pmatrix}}_{\ca{H}_1^2}
	=
	\scal{ (u_1 \indic{k},u_2\indic{k}) }{ B^\West_{[\South\North],[\South\North]} (v_1 \indic{l},v_2\indic{l} ) }_{\ca{W}_\West^2}
\end{equation}
\begin{lemm}\label{lemm:hs:recursion2}
	Given $B^\West_{\East,\East}$ and $B^\West_{a,\East}$ for $a\in\{\South,\North\}$, the block $B^\West_{[\South\North],[\South\North]}$ is a solution of \eqref{eq:Westfixedpoint:SNtofold} if and only if the sequence $(\beta_{k,l})_{k,l<0}$ satisfies, with $K=B_{\East,\East}^{\West} + Q_{\West,\West}$,
	\begin{equation}
		\beta_{k-1,l-1} = \beta_{k,l} - \begin{pmatrix}
			\gamma^{(\South)}_{k} \\ \gamma^{(\North)}_{k}
		\end{pmatrix} K^{-1} \begin{pmatrix}
			\gamma^{(\South)}_{l} & \gamma^{(\North)}_{l}	\end{pmatrix}  
	\end{equation}
	with boundary conditions
	\begin{align*}
		\beta_{-1,-1} &= Q_{[\South\North],[\South\North]}-Q_{[SN],\West} K^{-1} Q_{\West,[\South\North]}  &  
	\end{align*}
	\begin{align*}
		\beta_{-1,l-1} &= -Q_{[\South\North],\West} K^{-1} \begin{pmatrix}
			\gamma^{(\South)}_{l} & \gamma^{(\North)}_{l}
		\end{pmatrix}  
		&
		\beta_{k-1,-1} &= - \begin{pmatrix}
			\gamma^{(\South)}_{k} \\\gamma^{(\North)}_{k}
		\end{pmatrix} K^{-1} Q_{\West,[\South\North]}  
	\end{align*}
\end{lemm}
In this case the boundary terms $\beta_{k,-1}$ and $\beta_{-1,l}$ are also geometric and the generic terms $\beta_{k,l}$ are sums of (matricial) geometric terms and all of them can be written using $Q$ and $B^{\West}_{\East,\East}$, which is a finite-dimensional matrix satisfying a non-linear equation studied in \cite{Bodiot}.

These lemma are important since they show how the Gaussian Ansatz of quadratic forms in the boundary structure with the additional hypothesis \eqref{eq:halflinechoices}, which realizes the shift operators without any information coming from infinity, translate the definition up to fixed points into concrete and solvable equations.

\paragraph{Other Ans\"atze as exercises}
We left to the reader the following interesting exercise. Instead of \eqref{eq:halflinechoices}, one may have thought of finite dimensional spaces $\ca{W}_a$. For example, one may think of $\ca{W}_\West = \setC^d$ for some fixed dimension $d$. The maps $D^L_\bullet$ are all generated by $D^L_1 (w,h) = Rw+Sh$ with $R\in\Mat_{d,d}(\setC)$ and $S\in\Mat_{d,d_1}$ are finite matrices to be found. Such a choice also translates equations~\eqref{eq:Westfixedpoint:blocks} to be written and solved and one then observes that generically no solution exists (but it may be the case for some special points).

\subsubsection{The corner fixed points}
\paragraph{From operadic fixed points to matricial equations}
We now focus on the corner elements, which are new elements without any one-dimensional analogue. The following proposition presents the fixed point equations for the South-West corner element $Q^{[\infty_\West,\infty_\South]}$ once the West and South half-strips fixed points are known. At first sight, the following equations are of the same type as the one in theorem~\ref{theo:westfixedpoint} but the interesting points is that the system of equations on the corner blocks is \emph{overdetermined}: the consequence is that it will add further constraints between the West and the South half-strip fixed points, hence validating or invalidating the choice of spaces \ref{eq:halflinechoices} and the Gaussian Ansatz.

\begin{theo}\label{theo:corner:charactblocks}
	A triplet of elements $(\ov{Q}^{[\infty_\West,1]}, \ov{Q}^{[1,\infty_\South]},\ov{Q}^{[\infty_\West,\infty_\South]})$ in $\ca{Q}_{\infty_\West,1}\times \ca{Q}_{1,\infty_\South}\times \ca{Q}_{\infty_\West,\infty_\South}$ forms a South-West corner of fixed points of the semi-group generated by $Q$ (with identity morphisms $\phi_{\infty_a}^{b,\bullet,\bullet}=\id$) if and only if $\ov{Q}^{[\infty_\West,1]}$ satisfies theorem~\ref{theo:westfixedpoint}, $\ov{Q}^{[1,\infty_\South]}$ satisfies theorem~\ref{theo:westfixedpoint} up to a index change to match the South direction and the corner element $\ov{Q}^{[\infty_\West,\infty_\South]}$ acting on $\ca{W}_\West\times \ca{W}_{\South}$ with block notation ($C$ as "corner") acting on $\ca{W}_\West\times \ca{W}_{\South}$
	\begin{equation}
		\label{eq:corner:blocknames}
		\ov{Q}^{[\infty_\West,\infty_\South]}
		=\begin{pmatrix}
			C^{\South\West}_{\East,\East} & C^{\South\West}_{\East,\North} \\
			C^{\South\West}_{\North,\East} & C^{\South\West}_{\North,\North} 
		\end{pmatrix}
	\end{equation}
	satisfies the two fixed point equations
	\begin{align*}
		\ov{Q}^{[\infty_\West,\infty_\South]}&=
		\begin{tikzpicture}[guillpart,yscale=1.15,xscale=2.5]
			\fill[guillfill] (0.5,0) rectangle (3,1);
			\draw[guillsep] (0.5,1)--(3,1)--(3,0) (2,0)--(2,1);
			\node at (1.25,0.5) {$\ov{Q}^{[\infty_\West,\infty_\South]}$};
			\node at (2.5,0.5) {$\ov{Q}^{[1,\infty_\South]}$};
		\end{tikzpicture}_{\ca{Q}}^{D^L} 
		&
		\ov{Q}^{[\infty_\West,\infty_\South]}&=
		\begin{tikzpicture}[guillpart,yscale=1.15,xscale=3.25]
			\fill[guillfill] (0,1) rectangle (1,3);
			\draw[guillsep] (1,1)--(1,3)--(0,3) (0,2)--(1,2);
			\node at (0.5,1.5) {$\ov{Q}^{[\infty_\West,\infty_\South]}$};
			\node at (0.5,2.5) {$\ov{Q}^{[\infty_\West,1]}$};
		\end{tikzpicture}_{\ca{Q}}^{D^L} 
	\end{align*}
	which can be rewritten, using the invertible elements $L_2 = C^{\South\West}_{\East,\East}+B^{\South}_{\West,\West} $ and $L_1 = C^{\South\West}_{\North,\North}+B^{\West}_{\South,\South}$, as
	\begin{subequations}
		\label{eq:cornerfixedpoint:blocks}
		\begin{align}
			\label{eq:corner:fixed:EE:1D}
			C^{\South\West}_{\East,\East} &= B^{\South}_{\East,\East} - B^{\South}_{\East,\West} L_2 ^{-1}B^{\South}_{\West,\East}
			\\
			\label{eq:corner:fixed:nondiagonal:NE}
			C^{\South\West}_{\North,\East} &= (C^{\South\West}_{\East,\North})^*  = D_1^L 
			\begin{pmatrix}
				-C^{\South\West}_{\North,\East}L_2^{-1} B^{\South}_{\West,\East} \\
				B^{\South}_{\North,\East}- B^{\South}_{\North,\West}L_2^{-1} B^{\South}_{\West,\East}
			\end{pmatrix}
			\\
			\label{eq:corner:fixed:NN:shiftinv}
			C^{\South\West}_{\North,\North} &= D_1^L\left[ 
			\begin{pmatrix}
				C^{\South\West}_{\North,\North} & 0 \\
				0 & B^{\South}_{\North,\North}
			\end{pmatrix}-
			\begin{pmatrix}
				C^{\South\West}_{\North,\East} \\
				B^{\South}_{\North,\West}
			\end{pmatrix}
			L_2^{-1}
			\begin{pmatrix}
				C^{\South\West}_{\East,\North}, 
				B^{\South}_{\West,\North}
			\end{pmatrix}
			\right](D_1^L)^*
		\end{align}
		for the first one and as
		\begin{align}
			\label{eq:corner:fixed:NN:1D}
			C^{\South\West}_{\North,\North} &= B^{\West}_{\North,\North}- B^{\West}_{\North,\South} L_1^{-1} B^{\West}_{\South,\North}
			\\
			\label{eq:corner:fixed:nondiagonal:EN}
			C^{\South\West}_{\East,\North} &= (C^{\South\West}_{\North,\East})^*
			=D_1^L \begin{pmatrix}
				-C^{\South\West}_{\East,\North}L_1^{-1} B^{\West}_{\South,\North}
				\\
				B^{\West}_{\East,\North}-B^{\West}_{\East,\South}L_1^{-1}B^{\West}_{\South,\North}
			\end{pmatrix}
			\\
			\label{eq:corner:fixed:EE:shiftinv}
			C^{\South\West}_{\East,\East} &= D_1^L\left[
			\begin{pmatrix}
				C^{\South\West}_{\East,\East} & 0 \\
				0 & B^{\West}_{\East,\East}
			\end{pmatrix}
			-
			\begin{pmatrix}
				C^{\South\West}_{\East,\North}
				\\
				B^{\West}_{\East,\South}
			\end{pmatrix}
			L_1^{-1}
			\begin{pmatrix}
				C^{\South\West}_{\North,\East}
				,
				B^{\West}_{\South,\East}
			\end{pmatrix}
			\right] (D_1^L)^*
		\end{align}
		for the second one.
	\end{subequations}
\end{theo}
\begin{proof}
	This is a direct rewritting of definitions~\ref{def:fixedpointuptomorph:leftextended} and \ref{def:fixedpointcorner} adapted to the present case of quadratic forms with Schur complements.
\end{proof}
Each block appears twic $C^{\South\West}_{aa}$ appears twice: once as a one-dimensional fixed point equation with \eqref{eq:corner:fixed:EE:1D} or \eqref{eq:corner:fixed:NN:1D} (see \cite{Bodiot}) under the 1D Gaussian dynamics induced by the half-strip diagonal blocks, once as shift-invariant elements with \eqref{eq:corner:fixed:EE:1D} or \eqref{eq:corner:fixed:NN:1D} which uses the non-diagonal blocks $C^{\South\West}_{ab}$, $a\neq b$, which are intimately related to the two-dimensional structure and the square associativity as explained in the next section.

\paragraph{Solving equations~\eqref{eq:cornerfixedpoint:blocks}}

Equations \eqref{eq:corner:fixed:EE:1D} and \eqref{eq:corner:fixed:NN:1D} are fixed points under Schur complements and be studied directly through \cite{Bodiot}. We will also see in section~\ref{sec:foldingsetc} that, in the generic case, there exists a unique solution with alternative explicit representations inherited from the Fourier transform.

The four other equations in \eqref{eq:cornerfixedpoint:blocks} can be described by recursions with lemmata similar to lemmata~\ref{lemm:hs:recursion1} and~\ref{lemm:hs:recursion2} using the shift structure $D^L_\bullet$ of the spaces $\ca{W}_a$. We do not reproduce them here in order to gain some place but there is no difficulty in it.

The most important and difficult point to check is that these explicit solutions obtained by recursion are consistent with each other since there are six equations for three unknown blocks. There are two points of view.
\begin{itemize}
	\item from a concrete numerical perspective: for a given model with a specified matrix $Q$, it is a quick task to solve, at least on a computer, the previous corner equations by the various recursions and check that the results coincide for a large subset of indices. It allows at least to quickly invalidate the Ansatz \eqref{eq:halflinechoices} or push forward the computations.
	\item from a abstract rigorous perspective, we can solve the various recursions explicitly (this is feasible formally) and use remarkable identities to switch from one of the representations to another. This can be done but is not enlightening at all. However, all the remarkable identities can be proved using two new tools introduced in the following section in the generic case (the missing cases are just much more technical but presenting them does not add any new concept nor method): folding and square associativity. Therefore, we prefer stop here the study of the recursions for themselves and come back to a wider operadic picture.
\end{itemize}

\subsubsection{Infinite-volume Gibbs measures out of fixed points}

Given solutions of the previous fixed point equations on the four halpf-strips and the four corners, it is then easy to obtain formally the boundary weights $g^{(\Lambda)}_{p,q}$ and hence the infinite volume Gibbs measure. We detail quickly here the computations.

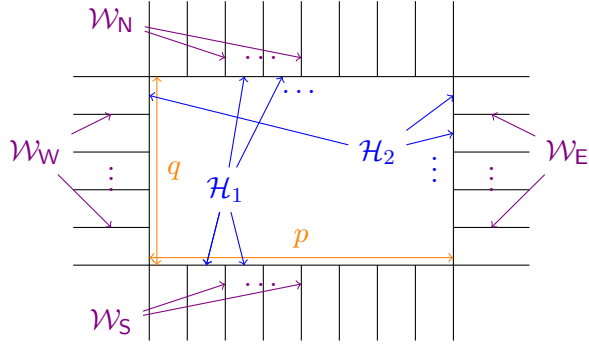
\begin{figure}
	\begin{center}
		\begin{tikzpicture}[scale=0.5]
			\draw (0,0) rectangle (8,5);
			\foreach \x in {0,1,...,8} {
				\draw (\x,0)--(\x,-2);
				\draw (\x,5)--(\x,7);
			}
			\foreach \y in {0,1,...,5} {
				\draw (-2,\y)--(0,\y);
				\draw (8,\y)--(10,\y);	
			}
			\draw[<->, orange] (0,0.2)-- node [above] {$p$} (8,0.2);
			\draw[<->,orange ] (0.2,0)-- node [right] {$q$} (0.2,5);
			\node (H2) at (6,3) [blue] {$\ca{H}_2$};
			\draw[blue,->] (H2) -- (8,4.5);
			\draw[blue,->] (H2) -- (8,3.5);
			\draw[blue,->] (H2) -- (0,4.5);
			\node at (7.5,2.75) [blue] {$\vdots$};
			\node (H1) at (2,2) [blue] {$\ca{H}_1$};
			\draw[blue,->] (H1)--(1.5,0);
			\draw[blue,->] (H1)--(2.5,0);
			\draw[blue,->] (H1)--(2.5,5);
			\draw[blue,->] (H1)--(3.5,5);
			\node at (4,4.6) [blue] {$\dots$};
			\draw[blue,->] (H1)--(1.5,0);
			\node (WE) at (11,3) [violet] {$\ca{W}_\East$};
			\draw[violet,->] (WE) -- (9,4);
			\draw[violet,->] (WE) -- (9,1);
			\node at (9,2.5) [violet] {$\vdots$};
			\node (WW) at (-3,3) [violet] {$\ca{W}_\West$};
			\draw[violet,->] (WW) -- (-1,4);
			\draw[violet,->] (WW) -- (-1,1);
			\node at (-1,2.5) [violet] {$\vdots$};
			\node (WS) at (-1,-1.5) [violet] {$\ca{W}_\South$};
			\draw[violet,->] (WS) -- (2,-0.5);
			\draw[violet,->] (WS) -- (4,-0.5);
			\node at (3,-0.5) [violet] {$\dots$};
			\node (WN) at (-1,6.5) [violet] {$\ca{W}_\North$};
			\draw[violet,->] (WN) -- (2,5.5);
			\draw[violet,->] (WN) -- (4,5.5);
			\node at (3,5.5) [violet] {$\dots$};
		\end{tikzpicture}
	\end{center}
	\caption{\label{fig:peigne}Geometry for the reconstruction of the boundary weights from the generalized eigen-elements on half-strips and corners. Suitable Hilbert spaces are attached to each segment and each half-lines, depending on their orientation.}
\end{figure}

Given a rectangle with size $(p,q)$, we consider the geometry of figure~\ref{fig:peigne} with $p+1$ half-line to the North and to the South and $q+1$ half-lines to the West and to the East. For each $(p,q)\in \setN_1^2$, we introduce a Hilbert space as the orthogonal sum:
\[
\ca{V}_{p,q} = \left(\ca{W}_\South^{p+1}\oplus\ca{W}_{\North}^{p+1}\oplus \ca{W}_\West^{q+1}\oplus\ca{W}_\East^{q+1}\right) \oplus \left(\ca{H}_1^{2p} \oplus\ca{H}_2^{2q}\right)
=\ca{W}^{(out)}_{p,q} \oplus \ca{H}^{(b)}_{p,q}
\]
where each $\ca{W}_a$ corresponds to a suitable half-line and each $\ca{H}_i$ to a segment with length one. We then define a new quadratic form on this space by using the fixed point of theorems~\ref{theo:hs:fixedpoint:formulae} and \ref{theo:corner:fixedpoint:formulae}:
\begin{equation}
	\begin{split}
		Q^{out}_{p,q} =& \sum_{i=1} \left(\iota^{\South}_i( B^\South) + \iota^{\North}_i( B^\North) \right) + \sum_{j=q} \left(\iota^{\West}_j( B^\West) + \iota^{\East}_j( B^\East) \right)
		\\
		&+ \iota^{\South\West}(C^{\South\West}) 
		+\iota^{\South\East}(C^{\South\East})
		+\iota^{\North\West}(C^{\North\West})
		+\iota^{\North\East}(C^{\North\East})
	\end{split} 
\end{equation}
where $\iota^a_k(B)$ is a quadratic form on $\ca{V}_{p,q}$ defined as the action of $B$ on the three subspaces that corresponds to the three boundaries of the $i$-half-strip in direction $a$ (two half-lines and a segment) and $\iota^{ab}(C)$ acts only the two subspaces corresponding to the boundary of the corresponding corner. The quadratic form $Q^{out}_{p,q}$ on $\ca{V}_{p,q}$ is positive definite (and bounded away from zero) and define a centered Gaussian process on $\ca{V}_{p,q}$. Taking the Schur complement $Q^{out}_{p,q}$ w.r.t. 
$\ca{W}^{(out)}_{p,q}$ can be done sequentially over the half-lines and corresponds to the product of the boundary elements within $\ca{Q}_{\bullet}$: it then produces a quadratic form
\begin{equation}
	\label{eq:quadraticformfromROPE}
	Q^{\partial}_{p,q} = \mathrm{Schur}_{\ca{W}^{(out)}_{p,q}}( Q^{out}_{p,q} )
\end{equation}
that can be used to define directly the boundary weights $g^{(\Lambda)}_{p,q}$ through a Gaussian density. We thus have one of the main result of the paper, which solves the first question raised in the introduction and which is the first case of construction of a Gibbs measure directly issued from the Markov property and from concrete exact solvable fixed point equations in dimension strictly larger than $1$.

\begin{theo}\label{theo:kolmogorovwithboundaryweights}
	The collection of boundary weights $g^{(\Lambda)}_{p,q}$ for any rectangle with non-degenerate sizes $(p,q)$ defined in \eqref{eq:boundaryweight:quadraticform} from the quadratic forms $Q^{\partial}_{p,q}$  obtained in \eqref{eq:quadraticformfromROPE} from the fixed-points described in theorems~\ref{theo:hs:fixedpoint:formulae} and \ref{theo:corner:fixedpoint:formulae}, when a solution exists (see below) is well-defined generically and forms a projective limit (under restriction of probability laws from a rectangle to an included smaller one) and defines, by Kolmogorov's extension theorem, a infinite-volume Gibbs measure.
\end{theo}
\begin{proof}
	We first assume that the solutions given by theorems~\ref{theo:westfixedpoint} and \ref{theo:corner:charactblocks} are well-defined and belong to the spaces $\ca{Q}_\bullet$: these are consequences of theorems~\ref{theo:hs:fixedpoint:formulae} and \ref{theo:corner:fixedpoint:formulae} below. Then, the consistency of the boundary weights for growing rectangles are a direct consequence of the fixed point property under the Schur complement (the scalar coefficient in front of the density are indeed irrelevant and can absorbed into normalizations; they can however be given a meaning as seen in section~\ref{sec:eigenval} below). The infinite-volume Gibbs measure is then a direct consequence of \cite{Simon}.
	
	Showing that the solutions of the fixed point equations are well-defined and belong to the spaces $\ca{Q}_\bullet$ requires to prove suitable analytic bounds (bounded operator and the additional condition in $\ca{Q}_{\infty}$, which induces that the inverse exists  and is also bounded) and are also consequences of theorems~\ref{theo:hs:fixedpoint:formulae} and \ref{theo:corner:fixedpoint:formulae} below. This can be done under suitable generic hypotheses on the coupling matrix $Q$ directly using the recursive formulae but it is lengthy and quite obscure: we prefer present in the next section much more intuitive and direct results of existence of solutions by relating the fixed points inherited from the operadic structure to other intuitive objects that are well-known in other approaches in the literature.
\end{proof}

\section{From fixed points up to morphisms to analytic solutions and back}\label{sec:foldingsetc}

We have seen in theorems~\ref{theo:westfixedpoint} and \ref{theo:corner:charactblocks} and theorem~\ref{theo:kolmogorovwithboundaryweights} how solving simple recursions provide consistent boundary weights and an infinite-volume Gibbs measure. 

In the present case under generic hypothesis, we know for a long time from \cite{georgii} that, there is a unique translation-invariant infinite-volume Gibbs measure and thus the boundary weights obtained in the previous section shall coincide with the one inherited from \cite{georgii} (which is obtained through Fourier transform in a non-local way). Nonetheless these boundary weights are not written explicitly in \cite{georgii} and the correpondence is not straightforward. The present section makes these relations explicit.

The first  interesting points is that the various blocks $B^{c}_{ab}$ and $C^{c,d}_{a,b}$ that emerge from the operadic fixed points acquire an interesting structure related to various other properties of the model.

The second interesting point is the major roles played by two fundamental operations: folding and square associativity. Both of them have a geometrical nature deeply related to the guillotine operad and provide additional algebraic tools. The main question raised by the computations presented below with these two concepts is whether they can be adapted to other non-trivial models.

\subsection{A simplifying hypothesis of symmetry}

The purpose of the present paper is to introduce a general method and not to focus on specific features of singular models. Thus, we introduce additional generic hypothesis on the face coupling matrix $Q$ in order to have simplified notations. We however insist that these hypotheses can be avoided by using more involved variants of the present method.

The first assumption is the one of dihedral invariance. The square lattice $\setZ^2$ is invariant under the dihedral group, i.e. the group generated by the rotation of angle $\pi/2$ and the orthogonal symmetry with respect to the first diagonal. The definitions above of half-strip fixed points and corner fixed points are presented only for the South and West directions but similar definitions and properties can be formulated in the other directions and corners: the assumption of dihedral invariance for $Q$ avoids the need to write specific computations for each direction and each corner with heavy index notations. Moreover, it will also simplify the definition of foldings below.

\begin{lemm}
	A face operator $Q$ is invariant under the dihedral group if and only if both Hilbert spaces $\ca{H}_1$ and $\ca{H}_2$ are equal to some Hilbert space $\ca{H}$ and there exist three self-adjoint operators $T$, $A$ and $U$ acting on $\ca{H}$ such that (see figure~\ref{fig:Qblockstruct})
	\begin{equation}\label{eq:faceoperator:dihedralinvariant}
		Q= \begin{pmatrix}
			T & A & U & U \\
			A & T & U & U \\
			U & U & T & A \\
			U & U & A & T 
		\end{pmatrix}.
	\end{equation}
\end{lemm}
Every time a result of the present paper requires dihedral invariance, we will refer to it through a reference to the following assumption

\begin{assumptionp}{(Dihedral invariance)}\label{assump:dihedral}
	The face operator $Q$ is invariant under the dihedral group and thus admits the representation \eqref{eq:faceoperator:dihedralinvariant} with three self-adjoint operators $T$, $A$ and $U$ acting on $\ca{H}=\ca{H}_1=\ca{H}_2$.
\end{assumptionp}

\subsection{The standard approach by Fourier transform}\label{sec:fourier}
\paragraph*{Fourier transforms}
We quickly summarize the approach to the infinite-volume Gibbs measure by the Fourier transform and introduce all useful notations. For any function $f: \setZ \to \setC$, we define the Fourier transform $\ha{f}: S^1 \to \setC$ by 
\[
\ha{f}(e^{i\theta}) = \sum_{k\in\setZ} f(k) e^{i \theta k}
\]
with inverse Fourier transform given by $f(k)=(2\pi)^{-1} \int_0^{2\pi} f(e^{i\theta})e^{-i\theta k}d\theta$. Using the linear structure of Gaussian processes and the translation invariance of the face weight on the square lattice, the random field $(X_e)$ can be written in the Fourier space and it is easy to show that Fourier modes are independent. In order to describe the covariance structure, we introduce the function:
\begin{equation}\label{eq:psifunction:hv}
	\begin{split}
		\Psi_Q  :  \setC^*\times\setC^* &\to \End(\ca{H}_1\oplus \ca{H}_2 )\\
		(z,w) &\mapsto 
		\begin{pmatrix}
			(\Psi_Q(z,w))_{1,1}	 & (\Psi_Q(z,w))_{1,2}\\
			(\Psi_Q(z,w))_{2,1} &  (\Psi_Q(z,w))_{2,2}
		\end{pmatrix}
	\end{split}
\end{equation}
where the four blocks are given by:
\begin{subequations}
	\label{eq:def:1Dphi}
	\begin{align}
		(\Psi_Q(z,w))_{1,1} &=	\phi^{\South\North}_Q(w) = Q_{\South\South} +  Q_{\North\North} + w Q_{\South\North}  + w^{-1} Q_{\North\South}   
		\\
		(\Psi_Q(z,w))_{2,2} &= \phi^{\West\East}_{Q}(z) = Q_{\West\West}+Q_{\East\East} + z Q_{\West\East} +z^{-1}Q_{\East\West} 
		\\ 
		(\Psi_Q(z,w))_{1,2} &= Q_{\South\West}+z Q_{\South\East} + w^{-1} Q_{\North\West}+ w^{-1} z Q_{\North\East}
		\\
		(\Psi_Q(z,w))_{2,1}&= Q_{\West\South}+z^{-1} Q_{\East\South} + w Q_{\West\North} + w z^{-1} Q_{\East\North}
	\end{align}
\end{subequations}
The two diagonal blocks are related to the meromorphic function $\Phi_A$ of \cite{Bodiot} and related to the one-dimensional processes in $\ca{H}_1$ and $\ca{H}_2$ living on columns and rows of the two-dimensional lattice $\setZ^2$ with coupling matrices given by $Q_{[\South\North],[\South\North]}$ and $Q_{[\West\East],[\West\East]}$ respectively. 

We introduce for the strip and half-plane geometry below the partial Fourier coefficients of order $k\in\setZ$ as
\begin{align*}
	\Fourier^{h}_{k,\bullet}(\ha{a})& = \frac{1}{2\pi}\int_0^{2\pi} \ha{a}(e^{i\theta},\cdot)e^{-ik\theta}d\theta & 	\Fourier^{v}_{\bullet, k}(\ha{a})& = \frac{1}{2\pi}\int_0^{2\pi} \ha{a}(\cdot,e^{i\theta})e^{-ik\theta}d\theta.
\end{align*}
which are functions on the circle $S^1$ and we will often write
$\Fourier^h_{k,\bullet}(\ha{a})(e^{i\varphi}) = \Fourier^h_{k,\varphi}(\ha{a})$. We define also the full Fourier transform
\[
\Fourier_{k,l}(\ha{a}) = \frac{1}{(2\pi)^2}\int_{S^1\times S^1} \ha{a}(e^{i\theta},e^{i\phi}) e^{-ik\theta-il\phi} d\theta\,d\phi
\]
We finally introduce, for any $k,l\in\setZ$ and any $i,j\in\{1,2\}$, the coefficients
\begin{align}\label{eq:Fourier_coeff_Psi}
	\CPsiInv_{k,l}&= \Fourier_{k,l}(\Psi_Q^{-1})
	&
	\CPsiInv_{k,l}^{i,j}&= \Fourier_{k,l}((\Psi_Q^{-1})_{i,j})
\end{align}
and the following short notations for the partial Fourier coefficients of this function, for all $1\leq i,j\leq 2$ and $k$, $l\in\setZ$.
\begin{subequations}
	\label{eq:Partial_Fourier_coeff_Psi}
	\begin{align}
		\CPsiInv_{k,\bullet} &= \Fourier^h_{k,\bullet}(\Psi_Q^{-1}) 
		&	
		\CPsiInv_{\bullet,k} &= \Fourier^v_{\bullet, k}(\Psi_Q^{-1}) 
		\\	
		\CPsiInv_{k,\bullet}^{i,j} &= \Fourier^h_{k,\bullet}((\Psi_Q^{-1})_{i,j}) 
		& 
		\CPsiInv_{\bullet,k}^{i,j}&= \Fourier^v_{\bullet,k}((\Psi_Q^{-1})_{i,j}).
	\end{align}
\end{subequations}

\paragraph*{Known results about the infinite volume Gibbs measure}
We recall here the main construction from \cite{georgii}.
\begin{theo}[(existence, unicity and covariance, direct consequence of chapter~13 \cite{georgii})]\label{theo:georgii}
	Let $Q$ be a positive definite Hermitian matrix such that the function $\Psi_Q : S^1\times S^1 \to \End(\ca{H}_1\oplus\ca{H}_2)$ defined in \eqref{eq:psifunction:hv} satisfies $\Psi_Q(e^{i\theta},e^{i\phi})$ is positive definite for all $\theta,\phi\in[0,2\pi]$.
	There exists a unique translation-invariant infinite Gibbs measure for the  Gaussian Markov random field defined by the face coupling matrix $Q$ in definition~\ref{def:GaussMarkovProc:via:density}. 
	
	In the Fourier space, it is given by the isometry $\ha{\ca{V}} \to L^2(\Omega,\ca{G},\Prob)$ where $\ca{V}$ is the Hilbert space of functions $S^1\times S^1\to \ca{H}_1\oplus\ca{H}_2$ with the inner product
	\[
	\scal{\ha{f}}{\ha{f}}_{\ha{\ca{V}}} = \frac{1}{4\pi^2} \int_{[0,2\pi]^2} \scal{\ha{f}(\theta,\phi)}{\Psi_Q^{-1}(\theta,\phi)\ha{f}(\theta,\phi)}_{\ca{H}_1\oplus\ca{H}_2} 
	d\theta d\phi
	\]
	On the lattice $\setZ^2$, it is given by the isometry $\ca{V} \to L^2(\Omega,\ca{G},\Prob)$ where $\ca{V}$ is the Hilbert space of functions $\setZ^2\to \ca{H}_1\oplus\ca{H}_2$ with the inner product
	\[
	\scal{f}{f}_{\ca{V}} = \sum_{\substack{(k,l)\in\setZ^2 \\ (k',l')\in\setZ^2}} \scal{f(k,l)}{ \gr{C}_{k-k',l-l'} f(k',l')} 
	\]
	where the matrices $\gr{C}_{k,l}$ are defined above in \eqref{eq:Fourier_coeff_Psi}
\end{theo}
In particular, one observes that the free energy density $f=\log\Lambda$ is given here by
\begin{equation}\label{eq:eigenval:expression}
	\Lambda = (2\pi)^{d_1+d_2} \exp\left( -\frac{1}{(2\pi)^2}\int_{[0,2\pi]^2}\log\det\Psi_Q(e^{i\theta_1},e^{i\theta_2})d\theta_1d\theta_2 \right)
\end{equation}
that we recover by an operadic renormalization approach in section~\ref{sec:eigenval}.

Such a construction is intimately related to the full plane geometry (top of figure~\ref{fig:admissiblepatterns}) or the torus one (with restriction of the Fourier modes to the roots of unity) since the Fourier transform is a global transformation not suited for the study of boundaries and local gluing operations.

Theorem~\ref{theo:georgii} is valid without any further hypothesis. However, in order to hide some feaible but lengthy computations due to degeneracies in specific models, we will work in most of the section under a further genericity assumption inherited of \cite{Bodiot}.
\begin{assumptionp}{(2Dsimple)}\label{assump:2Deasy}
	The face operator $Q$ satisfy:
	\begin{enumerate}[(i)]
		\item the zeroes of $\det \phi_{Q}^{\South\North}(w)$ have multiplicity one and, for each of them, $\dim \ker \phi_{Q}^{\South\North}(w) = 1$;
		
		\item the zeroes of $\det \phi_{Q}^{\West\East}(z)$ have multiplicity one and, for each of them, $\dim \ker \phi_{Q}^{\West\East}(z) = 1$;
		
		\item the zeroes $(z,w)$ of  $\det \Psi_Q(z,w)$ have multiplicity one, lie outside $S^1\times S^1$ and, for each of them,
		$\dim\ker \Psi_Q(z,w)=1$.
	\end{enumerate}
\end{assumptionp} 

\subsection{Transfer matrix: from cylinders to strips and half-planes}
\subsubsection{Strip elements in the Fourier basis} \label{sec:cylindertostrip}
The standard approach to many models of 2D statistical mechanics is the transfer matrix on cylinders: it corresponds to the computation of the surface power
\[
T_{\West\East,p}=\begin{tikzpicture}[guillpart,yscale=1.15,xscale=1.75]
	\fill[guillfill] (0,0) rectangle (4,1);
	\draw[guillsep] (4,0)--(0,0)--(0,1)--(4,1)--cycle ;
	\draw[guillsep]	(1,0)--(1,1) (2,0)--(2,1) (3,0)--(3,1);
	\node at (0.,0.5)  {$\bullet$};
	\node at (0.5,0.5) {$\ee_{1,Q}$};
	\node at (1.5,0.5) {$\ee_{1,Q}$};
	\node at (2.5,0.5) {$\dots$};
	\node at (3.5,0.5) {$\ee_{1,Q}$};
	\node at (4,0.5)  {$\bullet$};
\end{tikzpicture}
\]
with the identification of the two opposite boundary vertical edges and integration over the corresponding variable. As before, this can be lifted at the level of quadratic form and we define
\[
Q^{\patterntype{cyl}_{\West\East}}_{p}=\begin{tikzpicture}[guillpart,yscale=1.15,xscale=1.65]
	\fill[guillfill] (0,0) rectangle (4,1);
	\draw[guillsep] (4,0)--(0,0)--(0,1)--(4,1)--cycle ;
	\draw[guillsep]	(1,0)--(1,1) (2,0)--(2,1) (3,0)--(3,1);
	\node at (0.,0.5)  {$\bullet$};
	\node at (0.5,0.5) {$Q$};
	\node at (1.5,0.5) {$Q$};
	\node at (2.5,0.5) {$\dots$};
	\node at (3.5,0.5) {$Q$};
	\node at (4,0.5)  {$\bullet$};
\end{tikzpicture}_{\ca{Q}}
\]
as an operator on $\ca{H}_1^{2p}$. From the operadic point of view, it corresponds to the pointed pattern shape $\patterntype{cyl}_{\West\East}^*$ of \cite{Simon} (section 3.5.6.2) with colours $(p,1,h)$ (with a base point put on the vertical cut used to glue the opposite boundaries). Diagonalizing the operator $T_p$ is easy through discrete Fourier transform: modes are independent and are described by a vertical 1D dynamics (see \cite{Bodiot} for details) and the eigenvalue $\lambda_p$ is similar to \eqref{eq:eigenval:expression} with the integral replaced by a Riemann sum over the $p$-th roots of unity. 

After Fourier transform, it is easy to see that the quadratic form $Q^{\patterntype{cyl}_{\West\East}}_{p}$ act on $l^2(\mathbb{U}_p;\ca{H}_1^2)$ pointwise through:
\[
(Q^{\patterntype{cyl}_{\West\East}}_{p}u)(e^{2i\pi k/p})= \ca{S}^Q_{\West\East}(e^{2i\pi k/p}) u(e^{2i\pi k/p}) 
\]
with operators $\ca{S}^Q_{\West\East}: S^1 \to \End(\ca{H}_1^2)$ ($\ca{S}$ for "strip") given by
\begin{align*}
	\ca{S}^Q_{\West\East}(z) &= {\small Q_{[\South\North],[\South\North]} -(Q_{[\South\North],[\West]}+Q_{[\South\North],[\East]}z)\phi^{\West\East}_{Q}(z)^{-1}(Q_{[\West],[\South\North]}+Q_{[\East],[\South\North]}z^{-1})}
	\\
	&= \mathrm{Schur}_{\ca{H}_2} 
	\begin{pmatrix}
		Q_{\South\South} & Q_{\South\North} & Q_{\South\West}+z Q_{\South\East}  \\
		Q_{\North\South} & Q_{\North\North} & Q_{\North\West}+z Q_{\North\East}  \\
		Q_{\West\South}+z^{-1} Q_{\East\South} & 
		Q_{\West\North} + z^{-1} Q_{\East\North} & \phi^{\West\East}_{Q}(z)
	\end{pmatrix}
\end{align*}
where the Schur complement corresponds to integration w.r.t. the suitable Fourier mode of the r.v. on the vertical edges in the cylinder.

The passage from finite $p$ to the infinite strip can be done in various ways and is straightforward in the Fourier space: the space $l^2(\mathbb{U}_p;\ca{H}_1)$ is then replaced by $L^2(S^1; \ca{H}_1)$ by normalizing all the scalar products to obtain Riemann sums and the mode-wise multiplication operator $\ca{S}^Q_{\West\East}(e^{i\theta})$ defined above. This is easy at the level of quadratic forms but more subtle at the level of densities $\ee_{Q}$ since the spaces become infinite-dimensional when $p\to\infty$ and require either renormalization or to drop densities. After inverse Fourier transform, the space $L^2(S_1;\ca{H}_1)$ corresponds to the lattice space $l^2(\setZ;\ca{H}_1)$ which is precisely $\ca{W}_{\West\East}$ defined in \eqref{eq:halflinechoices} from gluing of opposite half-strips. 

The same classical approach by vertical transfer matrices in the vertical direction provides an asymptotic space $L^2(S^1;\ca{H}_2)$ on which acts pointwise the map $\ca{S}^Q_{\South\North}: S^1 \to \End(\ca{H}_2^2)$
\begin{align*}
	\ca{S}^Q_{\South\North}(w) &= Q_{[\West\East],[\West\East]}-(Q_{[\West\East],[\South]}+Q_{[\West\East],[\North]}w)\phi^{\South\North}_{Q}(w)^{-1}(Q_{[\South],[\West\East]}+Q_{[\North],[\West\East]}w^{-1})
	\\
	&= \mathrm{Schur}_{\ca{H}_1} 					\begin{pmatrix}
		\phi^{\South\North}_Q(w) & Q_{\South\West} + w^{-1}Q_{\North\West}& Q_{\South\East}+w^{-1}Q_{\North \East} \\
		Q_{\West\South} + w Q_{\West\North} & Q_{\West\West} & Q_{\West\East} \\
		Q_{\East\South} + w Q_{\East\North} & Q_{\East\West} & Q_{\East\East}
	\end{pmatrix}
\end{align*}
where the Schur complement corresponds to integration on the suitable Fourier mode on the internal horizontal edges of the vertical strip.

The two functions $\ca{S}^Q_{\West\East}$ and $\ca{S}^Q_{\South\North}$ are classical objects and are obtained from traditional transfert matrices approaches with Fourier transform. However, they hide "by construction" one of the dimension, which is integrated out, and hence break the dihedral symmetry. Their one-dimensional nature will be used below to extract useful but restricted information about the half-strip and corner fixed points which are purely two-dimensional. We will see afterwards how to complete the missing information.

\subsubsection{Basic operadic structure of the strip elements}

\begin{prop}[pointwise 1D structure of the strip elements]\label{prop:transfer1D:fourier}
	For any face operators $Q$ and $Q'$ acting on $\ca{H}_1^2\oplus\ca{H}_2^2$, for any $u\in S^1$, it holds
	\[
	\ca{S}^{\begin{tikzpicture}[guillpart,yscale=0.75,xscale=1]
			\fill[guillfill] (0,0) rectangle (1,2);
			\draw[guillsep] (0,0)--(1,0)--(1,2)--(0,2)--(0,0) (0,1)--(1,1);
			\node at (0.5,0.5) {$Q$};
			\node at (0.5,1.5) {$Q'$};
	\end{tikzpicture}}_{\West\East}(u) = \Schur_{1D}\left( \ca{S}_{\West\East}^Q(u) ,  \ca{S}_{\West\East}^{Q'}(u) \right)
	\]
	where the Schur complement on the left is taken w.r.t. $\ca{H}_2^2$ instead of $\ca{H}_2$ and the product $\Schur_{1D}$ on the right is the associative product inherited from the vertical lattice $\setZ$ as introduced in~\ref{prop:def:Schur1D} below (see \cite{Bodiot}).
	
	Moreover, for any $u\in S^1$, the self-adjoint positive definite quadratic form $\ca{S}_{\West\East}^Q(u)$ on $\ca{H}_1$ defines a $\ca{H}_1$-valued one-dimensional process as in \eqref{eq:def:1Dproc} and, under assumption~\ref{assump:2Deasy}, admits a Gibbs measure on $\setZ$ with left and right invariant boundary quadratic forms on $\ca{H}_1$ given by $G^*_L={G}^*_{\infty_\South}(u)$ and $G^*_R= G^*_{\infty_\North}(u)$ given by lemma~\ref{lemm:onedim:invariantfromW} with the substitution $Q=\ca{S}_{\West\East}^Q(u)$.
	
	\emph{Mutatis mutandis}, the quadratic forms $\ca{S}_{\South\North}^Q(u)$ provides a $\ca{H}_2$-valued one-dimensional process, which admits, under assumption~\ref{assump:2Deasy}, a Gibbs measure on $\setZ$ with left and right invariant boundary quadratic forms on $\ca{H}_2$ given by $G^*_L={G}^*_{\infty_{\West}}(u)$ and $G^*_R= G^*_{\infty_{\East}}(u)$.
\end{prop}
\begin{proof}
	The first associative property is a direct consequence of the associativity of Schur complements \eqref{eq:Schurcomp:assoc}. It is then easy to see that assumption~\ref{assump:2Deasy} implies, for each $u \in S^1$, assumption~\ref{assump:1D} for the corresponding operator $\ca{S}_{\West\East}^Q(u)$. A direct application of \cite{Bodiot} then provides the result.
\end{proof}

In the horizontal dimension we now have three elements $\ca{S}_{\West\East}^Q$, ${G}^*_{\infty_\South}$ and ${G}^*_{\infty_\North}$ that are constructed from $Q$ and provides functions from $S^1$ to $\End(\ca{H}_1^2)$, $\End(\ca{H}_1)$ and $\End(\ca{H}_1)$ respectively with suitable fixed point property under the vertical 1D product $\Schur_{1D}$. These elements are continuous and bounded in $u$: hence they act on $L^2(S^1;\ca{H}_1^2)$, $L^2(S^1;\ca{H}_1)$ and $L^2(S^1;\ca{H}_1)$ respectively. Using the Fourier transform $\ca{F}$ from $S^1$ to $\setZ$, they also act on 
$f\in\ca{W}_{\West\East}^2$ and  $g\in\ca{W}_{\West\East}$ as introduced in \eqref{eq:halflinechoices} respectively through:
\begin{subequations}\label{eq:def:stripandhalfplanefixed}
	\begin{align}
		\ov{Q}^{\infty_{\West\East}} f  &\coloneqq  \ca{F} \ca{S}_{\West\East}^Q \ca{F}^{-1} f
		\\
		\ov{Q}^{\infty_{\West\East},\infty_{\South}} g  &\coloneqq \ca{F} {G}^*_{\infty_{\South}}\ca{F}^{-1} g
		\\
		\ov{Q}^{\infty_{\West\East},\infty_{\North}} g  &\coloneqq \ca{F} {G}^*_{\infty_{\North}} \ca{F}^{-1} g
	\end{align}
\end{subequations}
In particular, from the definition of the Schur complements, the elements $\ov{Q}^{\infty_{\West\East},\infty_{\South}}$ and $\ov{Q}^{\infty_{\West\East},\infty_{\North}}$ are invariant under gluing with $\ov{Q}^{\infty_{\West\East}}$.

We can now state the first relation between these classical objects and the new ones obtained in the operadic approach in section~\ref{sec:fixedpoints} through fixed point equations. 

\begin{theo}\label{theo:gluingtodoublyinf}
	Under assumptions \ref{assump:2Deasy}
	The three elements $\ov{Q}^{\infty_{\West\East}}$, $\ov{Q}^{\infty_{\West\East},\infty_{\South}}$ and $\ov{Q}^{\infty_{\West\East},\infty_{\North}}$ defined in \eqref{eq:def:stripandhalfplanefixed} belong respectively to the spaces $\ca{Q}_{\infty_{\West\East},1}$, $\ca{Q}_{\infty_{\West\East},\infty_\South}$ and $\ca{Q}_{\infty_{\West\East},\infty_\South}$ introduced in \eqref{eq:strip:Qspaces} with spaces given in \eqref{eq:halflinechoices}.
	
	Moreover, the following three identities are valid under assumptions \ref{assump:2Deasy} and \ref{assump:dihedral}:
	\begin{subequations}
		\label{eq:gluing:hstostrip:ctohp}
		\begin{align}
			\label{eq:gluing:hstostrip:first}
			\ov{Q}^{\infty_{\West\East}} &= 
			\begin{tikzpicture}[guillpart,yscale=1.5,xscale=3]
				\fill[guillfill] (0,0) rectangle (2,1);
				\draw[guillsep] (0,0)--(2,0) (0,1)--(2,1) (1,0)--(1,1);
				\node at (0.5,0.5) {$\ov{Q}^{[\infty_\West,1]}$};
				\node at (1.5,0.5) {$\ov{Q}^{[\infty_\East,1]}$};
			\end{tikzpicture}_\ca{Q}
			\\
			\ov{Q}^{\infty_{\West\East},\infty_\South} &= 
			\begin{tikzpicture}[guillpart,yscale=1.5,xscale=3]
				\fill[guillfill] (0,0) rectangle (2,1);
				\draw[guillsep]  (0,1)--(2,1) (1,0)--(1,1);
				\node at (0.5,0.5) {$ \ov{Q}^{[\infty_\West,\infty_\South]}$};
				\node at (1.5,0.5) {$ \ov{Q}^{[\infty_\East,\infty_\South]}$};
			\end{tikzpicture}_\ca{Q}
			\\
			\ov{Q}^{\infty_{\West\East},\infty_\North} &= 
			\begin{tikzpicture}[guillpart,yscale=1.5,xscale=3]
				\fill[guillfill] (0,0) rectangle (2,1);
				\draw[guillsep] (0,0)--(2,0)  (1,0)--(1,1);
				\node at (0.5,0.5) {$ \ov{Q}^{[\infty_\West,\infty_\North]}$};
				\node at (1.5,0.5) {$ \ov{Q}^{[\infty_\East,\infty_\North]}$};
			\end{tikzpicture}_{\ca{Q}}	\end{align}
	\end{subequations} (base point placed on the cut) where the elements $\ov{Q}^{[\infty_a,1]}$ and $\ov{Q}^{[\infty_a,\infty_b]}$, $a\in\{\West,\East\}$ and $b\in\{\South,\North\}$ are the fixed points described in theorems~\ref{theo:westfixedpoint} and \ref{theo:corner:charactblocks}.
	Similar identities hold in the second direction with vertical gluings.
\end{theo}
\begin{proof}
	It is see from their definition or their characterization in lemma~\ref{lemm:onedim:invariantfromW} that the three functions $\ca{S}_{\West\East}^Q$, ${G}^*_{\infty_\South}$ and ${G}^*_{\infty_\North}$ are continuous function in $u\in S^1$ and are thus bounded. Moreover, using assumption~\ref{assump:2Deasy}, for any $u\in S^1$, the three operators $\ca{S}_{\West\East}^Q(u)$, ${G}^*_{\infty_\South}(u)$ and ${G}^*_{\infty_\North}(u)$ are invertible and hence, by continuity, are bounded away from $0$. Thus, they belong to the spaces $\ca{Q}_{\infty_{\West\East},1}$, $\ca{Q}_{\infty_{\West\East},\infty_\South}$ and $\ca{Q}_{\infty_{\West\East},\infty_\South}$.
	
	The gluing property requires a little bit more work that we delay to section~\ref{sec:proof:foldingtheorems} where all the operators will acquire interesting analytic interpretations. 
\end{proof}
We now provide three important remarks to illustrate how the various approaches meet in this theorem.
\begin{rema}\label{rema:fouriervsoperad}
	This theorem is expected since the infinite-volume Gibbs measure is unique in the present case. However, it is interesting to note that both sides of \eqref{eq:gluing:hstostrip:ctohp} have very different natures. The l.h.s. is obtained from the infinite-volume Gibbs measure constructed by analytical mean from the Gaussian Fourier space without using specifically the Markov property of the model but its other properties. On the other hand, the r.h.s. is obtained as solution of explicitly fixed point equations related to the Markov property encoded in the guillotine operadic structure; we thus expect it to be more easily subject to generalization for other models.
\end{rema}
\begin{rema}
	It is a standard mystery in the transfer matrix formalism that all computations can be performed along a choosen dimension but changing dimensions in the middle of a computation is in general impossible. Only the knowledge of the full infinite-volume Gibbs measure allows to identify results obtained by a choice or the other of the "time" dimension. Cutting strips into half-strips and half-planes into corners solves this question by introducing more fundamental objects: for example, the four corners allows to obtain either the horizontal or the vertical half-planes by gluing in the transverse direction. From this perspective, the fixed points of the previous operadic section carry more information.
\end{rema}
\begin{rema}
	The half-strip and corner fixed points are more fundamental objects than the strip and half-plane elements since the l.h.s. of \eqref{eq:gluing:hstostrip:ctohp} can be obtained by gluing from the objects in the r.h.s. However, the converse is \emph{not} true: we will see below that the knowledge of the l.h.s. provides only part of the elements in the r.h.s., namely the diagonal blocks, which have a one-dimensional nature. We will see below that the previous equations \eqref{eq:gluing:hstostrip:ctohp} are unable to provide the interactions between segments and half-lines along different dimensions.
\end{rema}
In order to prove theorems~\ref{theo:kolmogorovwithboundaryweights} (existence of explicit solutions to the fixed point equations) and \ref{theo:gluingtodoublyinf} (gluing of half-strips to strips), we now dive into the internal structure of the various elements $\ov{Q}^{[a,b]}$ and they can be related to each other through suitable analytical tools. In order to lighten as much as possible notations, we now always work under the assumption \ref{assump:dihedral}.

\subsubsection{Foldings}\label{sec:fold}

We have seen how \eqref{eq:gluing:hstostrip:ctohp} provides the strip and half-plane elements from the half-strip and corner ones; we now investigate how these formulae can be \emph{partially} reversed with the notation of foldings of strip elements to blocks of half-strip elements. 

Starting with the spaces introduced in \eqref{eq:canonicalhalflinespaces}, we have orthogonal decomposition of line spaces on the half-line spaces $\ca{W}_{LR}(\ca{H})=\ca{W}_{L}(\ca{H})\oplus\ca{W}_{R}(\ca{H})$ and we introduce the associated orthogonal projectors $P^{b} : \ca{W}_{LR}(\ca{H}) \to \ca{W}_b(\ca{H})$ for $b\in\{L,R\}$ as well as a canonical self-adjoint involution $J$ on $\ca{W}_{LR}(\ca{H})$ that maps a sequence $(f_k)_{k\in\setZ})$ to the sequence $(f_{-1-k})_{k\in\setZ})$ and hence maps $\ca{W}_{L}(\ca{H})$ to $\ca{W}_{R}(\ca{H})$ and vice versa.

Given any bounded operator $A$ on $\ca{W}_{LR}(\ca{H})$, we define the following four Toeplitz-like and Hankel-like operators, 
\begin{align*}
	\toep^{b}(A) &= P^\West A P^\West
	&
	\hank^{b}(A) &= P^\West A J P^\West
\end{align*}
If the operator $A$ on $l^2(\setZ;\ca{H})$ is defined as $\ca{F}\ha{a}\ca{F}$ where $\ha{a}$ is an operator which acts by pointwise multiplication on $L^2(S^1;\ca{H})$ with $(\ha{a}\ha{f})(u) = \ha{a}(u) \ha{f}(u)$ for $u\in S^1$ then we have the following sequences of elements in $\ca{B}(\ca{H})$:
\begin{align*}
	\scal{\indic{k}}{\toep^\West(\ha{a}) \indic{l}}&= \Fourier_{k-l}(\ha{a}) \indic{k< 0}\indic{l< 0}
	&
	\scal{\indic{k}}{\toep^\East(\ha{a}) \indic{l}} &=  \Fourier_{k-l}(\ha{a}) \indic{k\geq 0}\indic{l\geq 0}
	\\
	\scal{\indic{k}}{\hank^\West(\ha{a}) \indic{l}} &= \Fourier_{k+l+1}(\ha{a}) \indic{k< 0}\indic{l< 0}
	&
	\scal{\indic{k}}{\hank^\East(\ha{a}) \indic{l}} &= \Fourier_{k+l+1}(\ha{a}) \indic{\geq 0}\indic{l\geq0}
\end{align*}
using the same correspondence as in \eqref{eq:def:blockindicnotations}.

\begin{defi}[folding]\label{def:fold}
	For any bounded operator $A$ on $\ca{W}_{LR}(\ca{H})$, we define the two folded operators $\fold^{b}(A)$, $b\in\{L,R\}$, by
	\[
	\fold^{b}(A) = \toep^{b}(A) - \hank^{b}(A)
	\]
	acting respectively on $\ca{W}_L(\ca{H})$ and $\ca{W}_R(\ca{H})$.
\end{defi}
If $\ha{a}$ is a continuous function $S^1\to\End(\ca{H})$, we write $\fold^{b}(\ha{a})$ the folded operator associated to the Fourier transform of the multiplication operator associated to $\ha{a}$. The folded operators satisfy the following properties.
\begin{prop}
	Let $\ha{a}_1,\ha{a}_2 : S^1 \to \End(\ca{H})$ be two continuous and \emph{even} functions (i.e. $\ha{a}_i(e^{i\theta)})=\ha{a}_i(e^{-i\theta})$ for all $\theta\in[0,2\pi)$). 
	\begin{enumerate}[(i)]
		\item if, for all $u\in S^1$, $\ha{a}_i(u)$ is self-adjoint and positive definite, then $\fold^{b}(\ha{a}_i)$ is also self-adjoint and positive definite.
		\item folding is a multiplicative map:
		\begin{equation}\label{eq:fold:morphism}
			\fold^{b}(\ha{a}_1\ha{a}_2) = \fold^{b}(\ha{a}_1)\fold^b(\ha{a}_2)
		\end{equation}
	\end{enumerate}
\end{prop}
\begin{proof}
	For even operators $\ha{a}_i$, we have $J\ha{a}_i =\ha{a}_i J$, hence the self-adjoint property. The positive definite property is a consequence of the projector structure. The morphism property is obtained by first proving directly from the definitions of Toeplitz and Hankel operators that
	\begin{align*}
		\toep^b(\ha{a}_1\ha{a}_2)&=\toep^b(\ha{a}_1)\toep^\West(\ha{a}_2)+\hank^b(\ha{a}_1)\hank^\West(J\ha{a}_2J)
		\\
		\hank^b(\ha{a}_1\ha{a}_2)&=\toep^b(\ha{a}_1)\hank^\West(\ha{a}_2)+\hank^b(\ha{a}_1)\toep^\West(J\ha{a}_2J)
	\end{align*}
\end{proof}
We can also introduce additive foldings by $\fold^{b}_+(A) = \toep^b(A)+\hank^b(A)$ and show that they satisfy the same properties.

In order to have notations easier to interpret, we will write $\fold^{\South}$ (resp. $\fold^{\North}$, $\fold^{\West}$ and  $\fold^{\East}$) the folding operators for $\fold^{L}$ (resp. $\fold^{R}$, $\fold^{L}$ and $\fold^{R}$) on $\ca{W}_{\South}$ (resp. $\ca{W}_{\North}$, $\ca{W}_{\West}$ and $\ca{W}_{\East}$) and use the same convention for Toeplitz and Hankel parts.

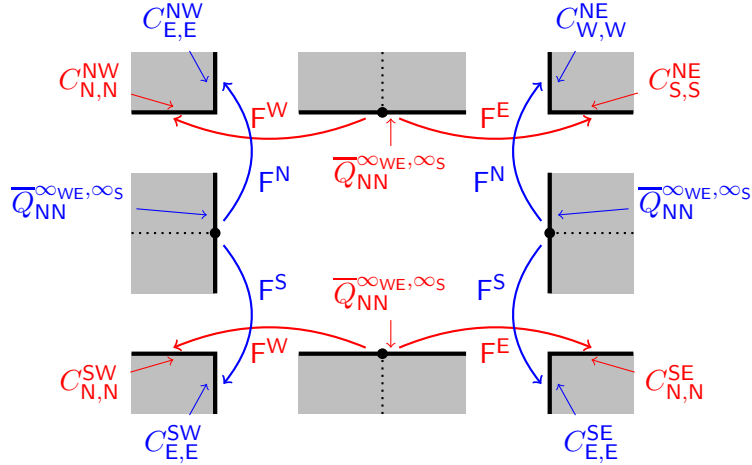
\begin{figure}
	\begin{center}
		\begin{tikzpicture}[yscale=0.8,xscale=1.1]
			\fill[guillfill] (0,0) rectangle (1,1);
			\draw[guillsep] (1,0)--(1,1)--(0,1);
			\node (CCSW_NN) at (-0.5,0.5) [red] {$C^{\South\West}_{\North,\North}$};
			\node (CCSW_EE) at (0.5,-0.5) [blue] {$C^{\South\West}_{\East,\East}$};
			\draw[->,red] (CCSW_NN) -- (0.5,0.9);
			\draw[->,blue] (CCSW_EE) -- (0.9,0.5);
			\fill[guillfill] (5,0) rectangle (6,1);
			\draw[guillsep] (5,0)--(5,1)--(6,1);
			\node (CCSE_NN) at (6.5,0.5) [red] {$C^{\South\East}_{\North,\North}$};
			\node (CCSE_WW) at (5.5,-0.5) [blue] {$C^{\South\East}_{\East,\East}$};
			\draw[->,red] (CCSE_NN) -- (5.5,0.9);
			\draw[->,blue] (CCSE_WW) -- (5.1,0.5);
			\fill[guillfill] (0,5) rectangle (1,6);
			\draw[guillsep] (0,5)--(1,5)--(1,6);
			\node (CCNW_SS) at (-0.5,5.5) [red] {$C^{\North\West}_{\North,\North}$};
			\node (CCNW_EE) at (0.5,6.5) [blue] {$C^{\North\West}_{\East,\East}$};
			\draw[->,red] (CCNW_SS) -- (0.5,5.1);
			\draw[->,blue] (CCNW_EE) -- (0.9,5.5);
			\fill[guillfill] (5,5) rectangle (6,6);
			\draw[guillsep] (5,6)--(5,5)--(6,5);
			\node (CCNE_SS) at (6.5,5.5) [red] {$C^{\North\East}_{\South,\South}$};
			\node (CCNE_WW) at (5.5,6.5) [blue] {$C^{\North\East}_{\West,\West}$};
			\draw[->,red] (CCNE_SS) -- (5.5,5.1);
			\draw[->,blue] (CCNE_WW) -- (5.1,5.5);
			\fill[guillfill] (2,0) rectangle (4,1);
			\draw[guillsep] (2,1)--(4,1);
			\coordinate (P1) at (3,0) ;
			\node (P2) at (3,1) [circle, fill, inner sep=0.5mm] {};
			\draw[thick, dotted] (P1)--(P2);
			\node (GS_NN) at (3.1,2) [red] {$\ov{Q}^{\infty_{\West\East},\infty_{\South}}_{\North\North}$};
			\draw[->,red] (GS_NN) -- (3.1,1.1);
			\fill[guillfill] (5,2) rectangle (6,4);
			\draw[guillsep] (5,2)--(5,4);
			\coordinate (P1) at (6,3) ;
			\node (P2) at (5,3) [circle, fill, inner sep=0.5mm] {};
			\draw[thick, dotted] (P1)--(P2);
			\node (GE_WW) at (6.75,3.5) [blue] {$\ov{Q}^{\infty_{\West\East},\infty_{\South}}_{\North\North}$};
			\draw[->,blue] (GE_WW) -- (5.1,3.3);
			\fill[guillfill] (0,2) rectangle (1,4);
			\draw[guillsep] (1,2)--(1,4);
			\coordinate (P1) at (0,3) ;
			\node (P2) at (1,3) [circle, fill, inner sep=0.5mm] {};
			\draw[thick, dotted] (P1)--(P2);
			\node (GW_EE) at (-0.75,3.5) [blue] {$\ov{Q}^{\infty_{\West\East},\infty_{\South}}_{\North\North}$};
			\draw[->,blue] (GW_EE) -- (0.9,3.3);
			\fill[guillfill] (2,5) rectangle (4,6);
			\draw[guillsep] (2,5)--(4,5);
			\coordinate (P1) at (3,6) ;
			\node (P2) at (3,5) [circle, fill, inner sep=0.5mm] {};
			\draw[thick, dotted] (P1)--(P2);
			\node (GN_SS) at (3.1,4) [red] {$\ov{Q}^{\infty_{\West\East},\infty_{\South}}_{\North\North}$};
			\draw[->,red] (GN_SS) -- (3.1,4.9);
			\draw[->,thick,red] (2.8,1.1) to [bend right=30] node [midway,below] {$\fold^\West$} (0.5,1.1);
			\draw[->,thick,red] (3.2,1.1) to [bend left=30] node [midway,below] {$\fold^\East$} (5.5,1.1);
			\draw[->,thick,red] (2.8,4.9) to [bend left=30] node [midway,above] {$\fold^\West$} (0.5,4.9);
			\draw[->,thick,red] (3.2,4.9) to [bend right=30] node [midway,above] {$\fold^\East$} (5.5,4.9);
			\draw[->,thick,blue] (1.1,2.8) to [bend left=30] node [pos=0.3,right] {$\fold^\South$} (1.1,0.5);
			\draw[->,thick,blue] (1.1,3.2) to [bend right=30] node [pos=0.3,right] {$\fold^\North$} (1.1,5.5);
			\draw[->,thick,blue] (4.9,2.8) to [bend right=30] node [pos=0.3,left] {$\fold^\South$} (4.9,0.5);
			\draw[->,thick,blue] (4.9,3.2) to [bend left=30] node [pos=0.3,left] {$\fold^\North$} (4.9,5.5);
		\end{tikzpicture}
	\end{center}
	\caption{\label{fig:fold:fromhalfplanetocorner}Diagonal blocks of corner fixed points from blocks of half-plane fixed points through foldings}
\end{figure}

We can now formulate the two main theorems that relate the operadic fixed point to the more classical approach by Fourier transform through folding. Their content, with their equivalent formulation on the other half-strips and corners, is illustrated in figure~\ref{fig:fold:fromhalfplanetocorner}.

\begin{theo}\label{theo:hs:fixedpoint:formulae}
	Under assumptions \ref{assump:dihedral} and \ref{assump:2Deasy}, for any $b\in\{\West,\East\}$, the diagonal blocks $B^{b}_{[\South\North],[\South\North]}$ of the half-strip fixed points $\ov{Q}^{[\infty_b,1]}$ of theorem~\ref{theo:westfixedpoint} are equal to
	\begin{align}	\label{eq:hs:blockfold}
		B^{\West}_{[\South\North],[\South\North]} &=\fold^{\West}( \ov{Q}^{\infty_{\West\East}} )
		&
		B^{\East}_{[\South\North],[\South\North]} &=\fold^{\East}( \ov{Q}^{\infty_{\West\East}} )
	\end{align}
\end{theo}

\begin{theo}[corner fixed point from foldings]\label{theo:corner:fixedpoint:formulae}
	Under assumptions~\ref{assump:dihedral} and \ref{assump:2Deasy}, the corner fixed point $\ov{Q}^{\infty_{\West},\infty_{\South}}$ of theorem~\ref{theo:corner:charactblocks} has blocks equal to
	\begin{subequations}
		\label{eq:corner:fixed:blocks}
		\begin{align}
			\label{eq:corner:fixed:blocks:fold}
			C^{\South\West}_{\North,\North}
			&=\fold^{\South}( \ov{Q}^{\infty_{\West},\infty_{\South\North}} )=
			\fold^\West(   G^*_{\infty_\South} )
			&
			C^{\South\West}_{\East,\East}
			&=\fold^{\West}( \ov{Q}^{\infty_{\West\East},\infty_{\South}} )=
			\fold^\South( G^*_{\infty_\West})
			\\
			\label{eq:corner:fixed:blocks:transverse}
			C^{\South\West}_{\East,\North}
			&= V^{\South\West}_{\searrow}
			&
			C^{\South\West}_{\North,\East}
			&= V^{\South\West}_{\nwarrow}
		\end{align}
	\end{subequations}
	where the operators $V^{\South\West}_{\searrow} : \ca{W}_{\West}\to\ca{W}_{\South}$ and $V^{\South\West}_{\nwarrow} : \ca{W}_{\South}\to\ca{W}_{\West}$ are defined by the elements
	\begin{subequations}
		\label{eq:def:cornerV}
		\begin{align}
			(V^{\South\West}_\searrow h)(-k-1) &= (U_Q^\West)^* (\id_{\ca{W}_{\West}}+ \fold^\West(\ha{W}_\South))\fold^\West(\ha{W}_\South)^k h \in \ca{H}_2
			\\
			(V^{\South\West}_\nwarrow g)(-k-1) &= (U_Q^\South)^* (\id_{\ca{W}_{\South}}+ \fold^\South(\ha{W}_\West))\fold^\South(\ha{W}_\West)^k g \in\ca{H}_1
		\end{align}
	\end{subequations}
	for all $k\geq 0$ and are adjoint to each others. All the other operators are described below in section~\ref{sec:cornerproof}.
\end{theo}
\begin{rema}
	These two theorems are partial reciprocal to theorem~\ref{theo:gluingtodoublyinf} but an important point is that they cannot be pushed further: the other blocks of the half-strip and corner fixed points are not accessible to a direct transfer matrix approach and can only be obtained from the fixed point equations of theorems~\ref{theo:westfixedpoint} and \ref{theo:corner:charactblocks}.
\end{rema}

\subsection{Proofs of the folding theorems}\label{sec:proof:foldingtheorems}
\subsubsection{Generalities}
The scheme of proof is simple: we need to show that the folded operators appearing in the theorems \ref{theo:hs:fixedpoint:formulae} and \ref{theo:corner:fixedpoint:formulae} satisfy the fixed point equations of theorems~\ref{theo:westfixedpoint} and \ref{theo:corner:charactblocks}.

The main difficulty of the proofs of these theorems is that the diagonal blocks in \eqref{eq:hs:blockfold} and \eqref{eq:corner:fixed:blocks} have to satisfy recursion that are not autonomous and involves the other blocks of $\ov{Q}^{[\infty_\West,1]}$ and $\ov{Q}^{[\infty_\West,\infty_\South]}$, which, as already said, can not be obtain from the transfer matrix. This difficulty is a consequence of the fact that, as illustrated in figure~\ref{fig:admissiblepatterns}, the approaches of sections~\ref{sec:fixedpoints} and \ref{sec:fourier} are reversed and meet only through these theorems. 

Therefore, the first step is to rewrite all the operators involved in terms of suitable underlying objects. In the present case, the idea is to use as most as possible the operators \eqref{eq:dimone:def:W} and their representations as Fourier coefficients \eqref{eq:1D:link:WwithFourier} inherited from \cite{Bodiot}.

The half-strip case is easy since the "small" block and the non-diagonal blocks have a very simple autonomous finite-dimensional structure. Proving theorem~\ref{theo:hs:fixedpoint:formulae} is then a simple rewriting of all the objects in terms of Fourier coefficients.

The corner case is much less easy since it does not involve the matrix $Q$ but rather the half-strip fixed points on the West and on the South. We must then generalize all the identities of \cite{Bodiot} used for theorem~\ref{theo:hs:fixedpoint:formulae} at a higher level for all $u\in S^1$. Moreover, all the blocks have to satisfy two independent equations (one fixed point and one recursion) as already discussed, which increases the number of remarkable identities to use.

\subsubsection{The half-strip case: proof of theorem~\ref{theo:hs:fixedpoint:formulae}.}
\begin{prop}\label{prop:hsfixed:rewritingsmallblocks}
	Under assumptions~\ref{assump:2Deasy} and \ref{assump:dihedral}, the East diagonal block of the fixed point is given by $\ov{Q}^{[\infty_\West,1]}$ 
	\[
	B^\West_{\East,\East} = T + A W^L_{Q,\West\East}
	\]
	where the operators $W^{L|R}_{Q,\West\East}$ are the shortcut notations of the operators $W^{L|R}_{Q_{[\West\East],[\West\East]}}$ defined in \eqref{eq:dimone:def:W} for the one-dimensional dynamics on $\ca{H}_2$ induced by $Q_{[\West\East],[\West\East]}$. The non-diagonal block is given by
	\[
	B^{\West}_{[\South\North],\East} = (B^{\West}_{\East,[\South\North]} )^*=\begin{pmatrix}
		U^\West_Q \\
		U^\West_Q 
	\end{pmatrix}
	\]
	where the operator $U_Q^\West : \ca{H}_2\to\ca{W}_\West$ is defined, for all $k\geq 0$, by
	\begin{align}\label{eq:hs:transversecoupling}
		(U^{\West}_Q f)(-k-1) 
		= U (\id_{\ca{H}_2} + W_{Q,\West\East}^L) (W_{Q,\West\East}^L)^k f
	\end{align}
\end{prop}
\begin{proof}
	This is a direct consequence of \cite{Bodiot}: equation~\eqref{eq:Westfixedpoint:1DEast} is the fixed point of a one-dimensional Gaussian Markov dynamics on $\setZ$ with coupling matrix $Q_{[\West\East],[\West\East]}$. Assumption~\ref{assump:2Deasy} implies assumption~\ref{assump:1D} and we obtain the result from lemma~\ref{lemm:onedim:invariantfromW}.
	
	The recursion of lemma~\ref{lemm:hs:recursion1} becomes under assumption~\ref{assump:dihedral}:
	$\gamma^{(a)}_{-1} = U(\id-K^{-1} A)$
	and $\gamma^{(a)}_{k-1} = \gamma_k^{(a)} (-K^{-1} A)$ for $k<-1$.
	A consequence of the representation the $W$ operators in terms of Fourier coefficients and the recursion between these coefficients provide $2T + AW^{L}_{Q,\West\East} + A(W^{L}_{Q,\West\East})^{-1}=0$, hence $K^{-1}A=-W^{L}_{Q,\West\East}$ and the result. The non-diagonal were announced to be geometric sequences and the operators $W^{L}_{Q,\West\East}$ provides the ratio.
\end{proof}

We first prove \eqref{eq:gluing:hstostrip:first} before theorem~\ref{theo:hs:fixedpoint:formulae} as a warm-up.

\begin{proof}[Proof of \eqref{eq:gluing:hstostrip:first}]
	In the block decomposition $\ca{W}_{\West\East}=\ca{W}_\West\oplus\ca{W}_\East$, both sides of \eqref{eq:gluing:hstostrip:first} are given by 
	\begin{align*}
		\mathrm{lhs}\eqref{eq:gluing:hstostrip:first}	&= \begin{pmatrix}
			P^\West\ov{Q}^{\infty_{\West\East}}P^\West
			&
			P^\West\ov{Q}^{\infty_{\West\East}} P^\East
			\\
			P^\East\ov{Q}^{\infty_{\West\East}}	P^\West
			&
			P^\East\ov{Q}^{\infty_{\West\East}}	P^\East
		\end{pmatrix}
		\\
		\mathrm{rhs}\eqref{eq:gluing:hstostrip:first}
		&= 	\begin{pmatrix}
			\fold^{\West}( \ov{Q}^{\infty_{\West\East}} ) -   U_Q^{\West}\ti{K}^{-1}(U_Q^{\West})^*
			&
			- U_Q^{\West} \ti{K}^{-1} (U_{Q}^{\East})^*
			\\
			-U_{Q}^{\East}\ti{K}^{-1}(U_Q^{\West})^*
			&
			\fold^{\East}( \ov{Q}^{\infty_{\West\East}} )  -U_{Q}^{\East}\ti{K}^{-1}(U_{Q}^{\East})^*
		\end{pmatrix}
	\end{align*}
	where the r.h.s. is obtained as the Schur complement w.r.t.~the space $\ca{H}_2$ on the vertical edge with $\ti{K}= 2T+AW^L_{Q,\West\East}+AW^R_{Q,\West\East}$ under the two assumptions. The East elements are obtained in a similar way as in proposition~\ref{prop:hsfixed:rewritingsmallblocks} with reversed indexed $k\to 1-k$ and right operators instead of left ones.
	
	Using $JP^\West J = P^{\East}$ and $JU_Q^\West = U^{\East}_Q$ under dihedral symmetry, the four block equalities above are all consequences of the single identity, that we now prove:
	\begin{align}\label{eq:gluingidtoshow}
		P^\West \ov{Q}^{\infty_{\West\East}}  P^\East &=- U^\West_Q \ti{K}^{-1} (U_Q^\East)^*
	\end{align}
	Using \eqref{eq:1D_somme_G_L_G_R}, we obtain $\ti{K}^{-1}=\ca{F}_0(\phi_Q^{\West\East})$ (Fourier coefficient). On the other hand, for any $k,l\geq 0$ and any $a,b\in\{\South,\North\}$, the definition of $\ov{Q}^{\infty_{\West\East}}$ rewritten under dihedral symmetry provide:
	\[
	\scal{ \indic{-l-1} }{ \ov{Q}^{\infty_{\West\East}}_{ab} \indic{k} } 
	= \Fourier_{-k-l-1}\left(   -U(1+\bullet)\phi_Q^{\West\East}(\bullet)^{-1} (1+\bullet^{-1})U \right) 
	\]
	Indeed, the block $Q_{[\South\North],[\South\North]}$ participates only the $0$-th Fourier coefficient. On the other hand,
	$\scal{ \indic{-l-1} }{	\mathrm{rhs}\eqref{eq:gluing:hstostrip:first}_{ab} \indic{k}}$ can be obtained by replacing the powers $(W^{L|R}_{Q,\West\East})^j$ by the Fourier coefficients of $\phi_W^{\West\East}$ using \eqref{eq:1D:link:WwithFourier}. The products of Fourier coefficients are simplified with \eqref{eq:1D:FourierPhiInv:remarkable} and give the expected identity \eqref{eq:gluingidtoshow}.
\end{proof}

Using the same type of computation, we can now complete the proof of theorem~\ref{theo:hs:fixedpoint:formulae}.

\begin{proof}[Proof of theorem~\ref{theo:hs:fixedpoint:formulae}]
	We first observe that, by definition, of the folded operator,
	\[
	\scal{\indic{k}}{\fold^{\West}(\ov{Q}^{\infty_{\West\East}})\indic{l}}
	=
	\Fourier_{k-l}(  \ca{S}^Q_{\West\East} ) - \Fourier_{k+l+1}(\ca{S}^Q_{\West\East})
	\]
	Using $f_n= \Fourier_n((\phi_Q^{\West\East})^{-1})$, we have,
	\[
	\Fourier_{n}(  \ca{S}^Q_{\West\East} ) = \delta_{n,0} Q_{[\South\North],[\South\North]} - U(2f_n+f_{n-1}+f_{n+1})U
	\]
	On the other hand, the coefficients of the recursion of lemma~\ref{lemm:hs:recursion2} can be rewritten in terms of the powers of the operators $W^{L}_{Q,\West\East}$ and replaced as before by the Fourier coefficients $f_n$. The recursion of lemma~\ref{lemm:hs:recursion2} is then equivalent to \eqref{eq:phiK:recursion:FourierCoeffs} and the theorem is proved.\end{proof}

A second interesting path of proof developed in \cite{BodiotThesis} relies on the proof of more algebraic intertwining properties relating the various operators. We present here only the formulation of these identities and refer to \cite{BodiotThesis} for the complete proof that relies again on the properties of the Fourier coefficients of $\phi_Q^{-1}$.
\begin{lemm}\label{lemm:FourierCoeffstoOperators}
	The following intertwining relations are valid for $a\in\{\West,\East\}$, using $f_0 = \Fourier_{0}((\phi_Q^{\West\East})^{-1})$:
	\begin{subequations}
		\begin{align}
			U_Q^W W_{Q,\West\East}^L &= \toep^\West(e^{i\bullet}) U_Q^W	\label{eq:WtoToep:throughU}
			\\
			- U_Q^a f_0 (U_Q^a)^* &= \hank^a( \ha{Q}^{\infty_{\West\East}} )
			\\
			- U_Q^\West f_0 (U_Q^\East)^* &= P^\West \ha{Q}^{\infty_{\West\East}} P^\East
			\\
			W_{Q,\West\East}^L  f_0 &=  f_0 (W_{Q,\West\East}^{L})^*
			\\
			W_{Q,\West\East}^L A^{-1} &= \left( (W_{Q,\West\East}^L)^2-\id\right) f_0
			\\
			\label{eq:shiftmapasToeplitz}
			D_1^L(f,u) &= \toep^\West(e^{-i\theta})f + u \indic{-1}
		\end{align}
	\end{subequations}
\end{lemm}

We now present a analytical consequence of theorem~\ref{theo:hs:fixedpoint:formulae} that was missing to verify that the solution $\ov{Q}^{[\infty_\West,1]}$ is admissible.

\begin{coro}\label{coro:hsfixed:belongstospace}
	The solution $\ov{Q}^{[\infty_\West,1]}$ of the fixed point equation of theorem~\ref{theo:westfixedpoint} belongs to $\ca{Q}_{\infty_\West,1}$.
\end{coro}
\begin{proof}
	We need to show that the operator is bounded and satisfies the infimum bound in \eqref{eq:hs:Qspaces}. From proposition~\ref{prop:hsfixed:rewritingsmallblocks} and theorem~\ref{theo:hs:fixedpoint:formulae}, all the blocks of $\ov{Q}^{[\infty_\West,1]}$ are bounded operators and hence this operator is bounded.
	
	For the infimum bound, we first write:
	\[
	\ov{Q}^{[\infty_a,1]} =
	\begin{pmatrix}
		I & L^* \\
		0 & I 
	\end{pmatrix}
	\begin{pmatrix}
		\mathrm{Schur}_{\ca{H}_2}(\ov{Q}^{[\infty_a,1]}) & 0 \\
		0 & B^\West_{\East,\East} 
	\end{pmatrix}
	\begin{pmatrix}
		I & 0 \\
		L & I 
	\end{pmatrix}	
	\]
	where $L$ is bounded operator that depends on the non-diagonal blocks $U_Q^\West$ and $B_{\East,\East}$. A quick computation using dihedral symmetry shows that 	\[
	\mathrm{Schur}_{\ca{H}_2}(\ov{Q}^{[\infty_a,1]}) = \toep^{\West}(\ov{Q}^{\infty_{\West\East}}) + \hank^{\West}(\ov{Q}^{\infty_{\West\East}})
	=  \fold_+^{\West}(\ov{Q}^{\infty_{\West\East}})
	\]
	Hence, $\mathrm{Schur}_{\ca{H}_2}(\ov{Q}^{[\infty_a,1]})$ is invertible and its inverse is bounded from the morphism property of folding and the invertibility and boundedness of $\ov{Q}^{\infty_{\West\East}}$. The same is true for $B^\West_{\East,\East}$. From the triangular structure above, it is then easy to show that $\ov{Q}^{[\infty_a,1]}$ is also invertible with a bounded inverse and the infinimum bound is thus satisfied.
\end{proof}

\subsubsection{The corner case: proof of theorem~\ref{theo:corner:fixedpoint:formulae}.}\label{sec:cornerproof}

\paragraph*{Preliminary results}

Theorem~\ref{theo:corner:charactblocks} provides six equations \emph{redundancy} that can be solved explicitly to obtain the corner elements. The redundancy adds constraints on the half-strip elements. Everytime a shift $D_1^L$ appears, we obtain recursions on the blocks of the corner quadratic forms, which can be solved explicitly. However, this is a bit more complicated than for half-strips since the initialization through elementary objects \eqref{eq:corner:fixed:EE:1D} and \eqref{eq:corner:fixed:NN:1D} is related to a infinite-dimensional operator $B^{\West}_{\South,\South}$ and $B^{\South}_{\West,\West}$ computed through theorem~\ref{theo:westfixedpoint}. 

Although this is perfectly feasible, we present here a second approach which provides a representation and an interpretation in terms of classical objects such as transfer matrices with suitable tools such as folding, which ensures nice proofs of existence (although extrinsic).

In order to formulate the folding representation, we introduce additional operators $W^{L|R}_\bullet$ for half-planes.
\begin{prop}\label{prop:halfplanetoW}
	The operators $G^*_{\infty_a}$ introduced in proposition~\ref{prop:transfer1D:fourier} admits, under assumptions~\ref{assump:2Deasy} and \ref{assump:dihedral} the following representations, for any $u\in S^1$
	\begin{align}
		G^*_{\infty_\West}(u) = \ca{S}_{\South\North}^Q(u)_{\East,\East} + \ca{S}_{\South\North}^Q(u)_{\East,\West} \ha{W}_\West(u)
		\\
		G^*_{\infty_\East}(u) = \ca{S}_{\South\North}^Q(u)_{\West,\West} +\ca{S}_{\South\North}^Q(u)_{\West,\East}\ha{W}_\East(u)
		\\
		G^*_{\infty_\South}(u) = \ca{S}_{\West\East}^Q(u)_{\North,\North} + \ca{S}_{\West\East}^Q(u)_{\North,\South}\ha{W}_\South(u)
		\\
		G^*_{\infty_\North}(u) = \ca{S}_{\West\East}^Q(u)_{\South,\South} + \ca{S}_{\West\East}^Q(u)_{\South,\North}\ha{W}_\North(u)
	\end{align}
	where the operators $\ha{W}_\South$ and $\ha{W}_\North$ (resp.~$\ha{W}_\West$ and $\ha{W}_\East$) act functions $S^1\to\End(\ca{H}_1)$ (resp. $\End(\ca{H}_2)$) such that $\ha{W}_\South(u)$ and $\ha{W}_\North(u)$ (resp.~$\ha{W}_\West(u)$ and $\ha{W}_\East(u)$) are the operators $W^L_K$ and $W^R_K$ defined in \eqref{eq:dimone:def:W} associated to $K=\ca{S}_{\West\East}^Q(u)$ (resp. $K=\ca{S}_{\South\North}^Q(u)$).
\end{prop}
\begin{proof}
	Under assumption~\ref{assump:2Deasy}, $\ca{S}_{\South\North}^Q(u)$ and $\ca{S}_{\West\East}^Q(u)$ are quadratic forms on $\ca{H}_2$ and $\ca{H}_1$ that satisfy \ref{assump:1D} and thus the results of \cite{Bodiot} can be applied. In particular, the operator $\ha{W}_a(u)$ inherit the representation in terms of Fourier coefficients of $\ca{S}_{\South\North}^Q(u)$ and $\ca{S}_{\West\East}^Q(u)$, which can themselves be written in terms of Fourier coefficients of $\Psi_Q^{-1}$ (see \cite{BodiotThesis} for precise formulae).	
\end{proof}

We first present a direct corollary of the fold representation of the fixed point of theorem~\ref{theo:corner:fixedpoint:formulae}.
\begin{coro}\label{coro:cornerfixed:belongstospace}
	The element $\ov{Q}^{[\infty_\West,\infty_\South]}$ as defined in theorem~\ref{theo:corner:fixedpoint:formulae} belongs to $\ca{Q}_{\infty_\West,\infty_\South}$.
\end{coro}
\begin{proof}
	The proof follows exactly the same steps as for corollary~\ref{coro:hsfixed:belongstospace} using alternative fold operators $\fold_+^a$ and the previous lemma. Similar bounds (with additional constants induced by the folding) are obtained in the same way and the computations are not reproduced here (see \cite{BodiotThesis} for details).
\end{proof}

Moreover, using the same approach with projectors as for the proof of \eqref{eq:gluing:hstostrip:first}, one first verify that the corner operators can be glued together to obtain the half-plane operators as stated in equations~\ref{eq:gluing:hstostrip:ctohp} and we not provide the proof again. The content of theorem~\ref{theo:corner:fixedpoint:formulae} shows that \eqref{eq:gluing:hstostrip:ctohp} can be partially reversed by folding to provide the diagonal blocks as for strips and half-strips.

\begin{proof}[Proof of theorem~\ref{theo:corner:fixedpoint:formulae}]
	The proof of theorem~\ref{theo:corner:fixedpoint:formulae} is split in the following steps:
	\begin{enumerate}
		\item the fixed point property~\eqref{eq:corner:fixed:EE:1D} and \eqref{eq:corner:fixed:NN:1D} of the blocks \eqref{eq:corner:fixed:blocks:fold}  proved in lemma~\ref{lemm:fixed:diagOK}
		\item the recursion \eqref{eq:corner:fixed:nondiagonal:NE} and \eqref{eq:corner:fixed:nondiagonal:EN} for $V^{\South\West}_{\nwarrow}$ and $V^{\South\West}_{\nwarrow}$ defined in \eqref{eq:def:cornerV} in lemma~\ref{lemm:fixed:VOK}
		\item the fact that $V^{\South\West}_{\nwarrow}$ and $V^{\South\West}_{\nwarrow}$ are adjoint to each other proved in lemma~\ref{lemm:fixed:VadjointOK}
		\item the transverse shift property \eqref{eq:corner:fixed:NN:shiftinv} and \eqref{eq:corner:fixed:EE:shiftinv} of the diagonal blocks proved in lemma~\ref{lemm:fixed:diagshiftOK}
	\end{enumerate}
	All the lemmata are proved below. There are organized in order to emphasize on the role of the two main tools with natural operadic interpretations: foldings and square associativity. 
	
	The first two points use only foldings as defined in section~\ref{sec:fold}, which are used to map results on transfer matrices and 1D dynamics to the diagonal blocks of the corners (they do not make orthogonal directions interact).
	
	The last two points are purely two-dimensional: the definitions used for the operators $V^{\South\West}_{\nwarrow}$ and $V^{\South\West}_{\nwarrow}$ have a one-dimensional nature but their adjunction property requires to mix two orthogonal directions. Moreover the diagonal blocks obtained from a 1D dynamics are required to satisfy a shift property in the second direction. Proving these properties requires, in a way or another, to use the purely two-dimensional part of the guillotine operad: the square associativity discussed below.
\end{proof}

\paragraph*{The folding part of the proof}

\begin{lemm}\label{lemm:fixed:diagOK}
	The diagonal blocks $C^{\South\West}_{a,a}$ with $a\in\{\North,\East\}$ given by folding in \eqref{eq:corner:fixed:blocks}  satisfy the two fixed point equations. \eqref{eq:corner:fixed:EE:1D} and \eqref{eq:corner:fixed:NN:1D}.
\end{lemm}
\begin{proof}
	This is direct consequence of the fixed point property of the half-plane elements with an additional fold. By construction, we have
	\begin{align*}
		G^*_{\infty_\West}(u) &=
		\Schur_{1D}(G^*_{\infty_\West}(u), \ca{S}^{Q}_{\South\North}(u))
	\end{align*}
	Applying the morphism property of the folding \eqref{eq:fold:morphism} to this equation, we immediately obtain \eqref{eq:corner:fixed:EE:1D}. The same computation in the transverse direction provides \eqref{eq:corner:fixed:NN:1D}.
\end{proof}

\begin{lemm}\label{lemm:fixed:VOK}
	The blocks \eqref{eq:corner:fixed:blocks} satisfy the non-diagonal equations \eqref{eq:corner:fixed:nondiagonal:NE} and \eqref{eq:corner:fixed:nondiagonal:EN}.
\end{lemm}
\begin{proof}
	Using the results of \cite{Bodiot} on $W^{L|R}$ operators in dimension one, we have
	\begin{align*}
		\ca{S}^Q_{\South\North}(u)_{\West\West}+ G^*_{\infty_\West}
		&= - \ca{S}^Q_{\South\North}(u)\ha{W}_\West(u)^{-1}
	\end{align*}
	The operator $L_2$ in \eqref{eq:corner:fixed:nondiagonal:NE} is obtained by folding this expression with $\fold^{\South}$ and using the morphism property of folding  \eqref{eq:fold:morphism} to obtain 
	$L_2^{-1} = -\fold^\South(\ha{W}_\West (\ca{S}_{\South\North}^{Q})^{-1}  )$.
	Using this, the r.h.s.~of \eqref{eq:corner:fixed:nondiagonal:NE} is given by \[
	D_1^L\begin{pmatrix}
		V^{\South\West}_{\nwarrow} \fold^{\South}(\ha{W}_\West)
		\\
		(U_Q^\South)^* (\id+\fold^{\South}(\ha{W}_\West))
	\end{pmatrix}
	=
	V^{\South\West}_{\nwarrow}
	\]
	from the definitions of the map $D_1^L$ and the geometric structure in the definition of $V^{\South\West}_{\nwarrow}$, hence obtaining the result. The second equation \eqref{eq:corner:fixed:nondiagonal:EN} is proved similarly by permuting the two directions.
\end{proof}

\begin{lemm}\label{lemm:fixed:diagshiftOK}
	The diagonal blocks $C^{\South\West}_{a,a}$ with $a\in\{\North,\East\}$ defined in  \eqref{eq:corner:fixed:blocks} also satisfy \eqref{eq:corner:fixed:NN:shiftinv} and \eqref{eq:corner:fixed:EE:shiftinv}.
\end{lemm}
\begin{proof}
	We prove only \eqref{eq:corner:fixed:NN:shiftinv} and the second equation is obtained by similar computations in the transverse directions. We first establish a lemma about the operators $\ha{W}_\West(u)$, which is a generalization of lemma~\ref{lemm:FourierCoeffstoOperators} by identifying a 1D process for each $u$.
	\begin{lemm}
		We have the following identities under the assumptions of dihedral symmetry and \ref{assump:2Deasy}, for all $u\in S^1$:
		\begin{subequations}
			\begin{align}
				\ha{W}_\West(u) \CPsiInv^{2,2}_{0,\bullet}(u) &=  \CPsiInv^{2,2}_{0,\bullet}(u) \ha{W}_\West(u)^*
				\\
				\ha{W}_\West(u) (\ca{S}_{\South\North}^Q(u)_{\West,\East})^{-1} &= \left( \ha{W}_\West(u)^2 -\id_{\ca{H}_1}\right) \CPsiInv^{2,2}_{0,\bullet}
			\end{align}
		\end{subequations}
	\end{lemm}
	We also need the following properties. All the proofs are consequence of the recursive properties of the Fourier coefficients $\CPsiInv_{k,l}$ of $\Psi_Q^{-1}$ sometimes with additional foldings that satisfy \eqref{eq:fold:morphism} (see \cite{BodiotThesis} for detailed computation if needed).
	\begin{lemm}\label{lemm:FourierCoeffstoOperators:cornercase}
		The following intertwining relations are valid:
		\begin{align*}
			V^{\South\West}_{\searrow} \fold^\West(\ha{W}_\South) &=  \toep^{\South}(e^{i\bullet}) V^{\South\West}_{\searrow}
			&
			V^{\South\West}_{\nwarrow} \fold^\South\left(
			\ha{W}_\West\right) &= \toep^\West(e^{i\bullet}) V^{\South\West}_{\nwarrow}
			\\
			-V^{\South\West}_{\nwarrow} \fold^{\South}(\CPsiInv^{2,2}_{0,\bullet}) V^{\South\West}_{\searrow} &= \hank^{\West}(\ha{Q}^{\infty_{\West\East},\infty_{\South}})
			&
			-V^{\South\East}_{\nearrow} \fold^{\South}(\CPsiInv^{2,2}_{0,\bullet}) V^{\South\East}_{\swarrow} &= \hank^{\East}(\ha{Q}^{\infty_{\West\East},\infty_{\South}})
			\\
			-V^{\South\West}_{\nwarrow} \fold^{\South}(\CPsiInv^{2,2}_{0,\bullet}) V^{\South\East}_{\swarrow} &= P^\West \ha{Q}^{\infty_{\West\East},\infty_{\South}} P^\East
		\end{align*}
	\end{lemm}
	
	We now consider the r.h.s.~$R$ of \eqref{eq:corner:fixed:NN:shiftinv} without the conjugation by $D_1^L$ block by block and simplify each expression. We have, using the previous expression for $L_2$ and the previous lemma, the following identities:
	\begin{align*}
		R_{1,1} =& \fold^{\West}( G^*_{\infty_\South} ) + V^{\South\West}_{\nwarrow} \fold^\South\left(
		\ha{W}_\West  (\ca{S}_{\South\North}^Q(u)_{\West,\East})^{-1}\right) V^{\South\West}_{\searrow}
		\\
		=& \fold^{\West}(G^*_{\infty_\South}) + V^{\South\West}_{\nwarrow} \fold^\South\left(
		(\ha{W}_\West^2-\id) \CPsiInv^{2,2}_{0,\bullet}\right) V^{\South\West}_{\searrow}
		\\
		=&\fold^{\West}(G^*_{\infty_\South}) - V^{\South\West}_{\nwarrow} \fold^\South\left(
		\CPsiInv^{2,2}_{0,\bullet}\right) V^{\South\West}_{\searrow} 
		\\
		&+ V^{\South\West}_{\nwarrow} \fold^\South\left(
		\ha{W}_\West\right)  \fold^\South\left(\CPsiInv^{2,2}_{0,\bullet}\right) \left(V^{\South\West}_{\nwarrow}\fold^\South\left(
		\ha{W}_\West\right) \right)^*
	\end{align*}
	We now use lemma~\ref{lemm:FourierCoeffstoOperators:cornercase} and its equivalent lemma in the transverse direction to simplify each term and we obtain
	\begin{align*}
		R_{1,1} &= \fold^{\West}(  G^*_{\infty_\South}) + \hank^\West( G^*_{\infty_\South}) + \toep^\West(e^{i\bullet}) V^{\South\West}_{\nwarrow}\fold^\South\left(\CPsiInv^{2,2}_{0,\bullet}\right) V^{\South\West}_{\searrow} \toep^\West(e^{-i\bullet})
		\\
		&= \toep^{\West}(  G^*_{\infty_\South})+ \toep^\West(e^{i\bullet}) \hank^\West(\ha{Q}^{\infty_{\West\East},\infty_{\South}}) \toep^\West(e^{-i\bullet})
		= \toep^{\West}(e^{i\bullet})  C^{\South\West}_{\North,\North}   \toep^{\West}(e^{-i\bullet})
	\end{align*}
	We now treat the other terms with the same rules. We have indeed
	\begin{align*}
		R_{1,2} &=  V^{\South\West}_{\nwarrow} \fold^{\South}\left(\ha{W}_\West
		(\ca{S}_{\South\North}^Q(u)_{\West,\East})^{-1}\right)  U_Q^\South
		= V^{\South\West}_{\nwarrow} \fold^{\South}\left((\ha{W}_\West^2-\id) \CPsiInv^{2,2}_{0,\bullet}\right)  U_Q^\South
		\\
		&= -V^{\South\West}_{\nwarrow} \fold^{\South}\left( \CPsiInv^{2,2}_{0,\bullet}\right)  U_Q^\South
		+ V^{\South\West}_{\nwarrow} \fold^{\South}(\ha{W}_\West) \fold^\South(\CPsiInv^{2,2}_{0,\bullet}) (\fold^{\South}(\ha{W}_\West))^*  U_Q^\South
		\\
		&= -V^{\South\West}_{\nwarrow} \fold^{\South}\left( \CPsiInv^{2,2}_{0,\bullet}\right) (\id+ \fold^{\South}(\ha{W}_\West))^* U_Q^\South
		+ V^{\South\West}_{\nwarrow} \fold^{\South}(\ha{W}_\West) \fold^\South(\CPsiInv^{2,2}_{0,\bullet}) (\id+\fold^{\South}(\ha{W}_\West))^*  U_Q^\South
	\end{align*}
	where we have introduced a telescopic term in the last equation. Using the definition of $V^{\South\West}_{\nwarrow}$ at $-1$, we obtain from lemma~\ref{lemm:FourierCoeffstoOperators:cornercase}
	$
	R_{1,2} =  \toep^\West(e^{i\bullet})C^{\South\West}_{\North,\North} \iota_{-1}$	
	where $\iota_{-1}: \ca{H}_1\to \ca{W}_{\West}$ is defined by $\iota_{-1}(u) = u\indic{-1}$. Indeed, it is easy to check that $\hank^\West(A)\iota_{-1} = \toep^\West(e^{i\bullet}) \toep^{\West}(A)\iota_{-1}$. Taking the adjoint produces a similar identity for $R_{2,1}$. We are now left with the last block $R_{2,2}$: 
	\begin{align*}
		R_{2,2} &= B^\South_{\North,\North} - (U_Q^\South)^* \fold^\South\left(\ha{W}_\West
		(\ca{S}_{\South\North}^Q(u)_{\West,\East})^{-1}\right)U_Q^\South
		\\
		&=B^\South_{\North,\North} + (U_Q^\South)^* \fold^\South\left( \CPsiInv^{2,2}_{0,\bullet} \right) U_Q^\South -
		(U_Q^\South)^* \fold^{\South}(\ha{W}_\West) \fold^\South\left( \CPsiInv^{2,2}_{0,\bullet} \right)\fold^{\South}(\ha{W}_\West)^* U_Q^\South  
		\\
		&=\iota_{-1}^* \fold^{\West}(G^*_{\infty_\South}) \iota_{-1}
	\end{align*}
	where we have introduced telescopic terms in order to use lemma~\ref{lemm:FourierCoeffstoOperators:cornercase} to make Hankel terms appear: the remaining terms precisely corresponds to \[\scal{\indic{-1}}{\toep^\West(G^*_{\infty_\South})\indic{-1}}= \Fourier_0(G^*_{\infty_\South})\] from the transverse representation lemma~\ref{lemm:Graal} below (which is provided by the square associativity). The r.h.s. of \eqref{eq:corner:fixed:NN:shiftinv} is now given, using \eqref{eq:shiftmapasToeplitz}, by 
	$D_1^L R (D_1^L)^*  = C^{\South\West}_{\North,\North}$ and we obtain the result.
\end{proof}

\paragraph*{The square associativity part of the proof}

Square associativity (see \cite{Simon} for its role in guillotine operads) asserts that a product of four elements in a subdivised rectangle can be performed sequentially starting in any dimension:
\[
\begin{tikzpicture}[guillpart,yscale=1.1,xscale=1.1]
	\fill[guillfill] (0,0) rectangle (2,2);
	\draw[guillsep] (0,2)--(2,2)--(2,0)--(0,0)--(0,2) (1,0)--(1,2) (0,1)--(2,1);
	\node at (0.5,0.5) {$Q_1$};
	\node at (0.5,1.5) {$Q_2$};
	\node at (1.5,0.5) {$Q_3$};
	\node at (1.5,1.5) {$Q_4$};
\end{tikzpicture}_\ca{Q}
= \begin{tikzpicture}[guillpart,yscale=.9,xscale=.9]
	\fill[guillfill] (0,0) rectangle (2,1);
	\draw[guillsep] (0,0)--(2,0)--(2,1)--(0,1)--(0,0) (1,0)--(1,1);
	\node at (0.5,0.5) {$1$};
	\node at (1.5,0.5) {$2$};
\end{tikzpicture}\circ
\left(
\begin{tikzpicture}[guillpart,yscale=1.1,xscale=1.1]
	\fill[guillfill] (0,0) rectangle (1,2);
	\draw[guillsep] (0,0)--(1,0)--(1,2)--(0,2)--(0,0) (0,1)--(1,1);
	\node at (0.5,0.5) {$Q_1$};
	\node at (0.5,1.5) {$Q_3$};
\end{tikzpicture}
,		
\begin{tikzpicture}[guillpart,yscale=1.1,xscale=1.1]
	\fill[guillfill] (0,0) rectangle (1,2);
	\draw[guillsep] (0,0)--(1,0)--(1,2)--(0,2)--(0,0) (0,1)--(1,1);
	\node at (0.5,0.5) {$Q_2$};
	\node at (0.5,1.5) {$Q_4$};
\end{tikzpicture}
\right)
=\begin{tikzpicture}[guillpart,yscale=.9,xscale=.9]
	\fill[guillfill] (0,0) rectangle (1,2);
	\draw[guillsep] (0,0)--(1,0)--(1,2)--(0,2)--(0,0) (0,1)--(1,1);
	\node at (0.5,0.5) {$1$};
	\node at (0.5,1.5) {$2$};
\end{tikzpicture}\circ\left(
\begin{tikzpicture}[guillpart,yscale=1.1,xscale=1.1]
	\fill[guillfill] (0,0) rectangle (2,1);
	\draw[guillsep] (0,0)--(2,0)--(2,1)--(0,1)--(0,0) (1,0)--(1,1);
	\node at (0.5,0.5) {$Q_1$};
	\node at (1.5,0.5) {$Q_2$};
\end{tikzpicture}
,		
\begin{tikzpicture}[guillpart,yscale=1.1,xscale=1.1]
	\fill[guillfill] (0,0) rectangle (2,1);
	\draw[guillsep] (0,0)--(2,0)--(2,1)--(0,1)--(0,0) (1,0)--(1,1);
	\node at (0.5,0.5) {$Q_3$};
	\node at (1.5,0.5) {$Q_4$};
\end{tikzpicture}
\right)
\]
The boundary elements are obtained through fixed points but their probabilistic interpretation as marginal weights coming from an infinite-volume Gibbs measures suggest interpretations such as:
\begin{align}\label{eq:westhalfplane:hypotheticallimit}
	\ca{S}^Q_{\South\North} &\overset{\mathbf{?}}{=}"\lim_{q\to\infty}" \begin{tikzpicture}[guillpart,yscale=1.,xscale=1.2]
		\fill[guillfill] (0,0) rectangle (1,3);
		\draw[guillsep] (0,0)--(1,0)--(1,3)--(0,3)--(0,0) (0,1)--(1,1) (0,2)--(1,2) ;
		\node at (0.5,0.5) {$Q$};
		\node at (0.5,1.5) {$\vdots $};
		\node at (0.5,2.5) {$Q$};
		\draw[<->] (-0.2,0.) -- node [left] {$q$} (-0.2,3.);
	\end{tikzpicture}
	&
	\ov{Q}^{[\infty_\West,1]} &\overset{\mathbf{?}}{=}
	"\lim_{p\to\infty}" \begin{tikzpicture}[guillpart,yscale=1.,xscale=1.2]
		\fill[guillfill] (0,0) rectangle (3,1);
		\draw[guillsep] (0,0)--(3,0)--(3,1)--(0,1)--(0,0) (1,0)--(1,1) (2,0)--(2,1);
		\node at (0.5,0.5) {$Q$};
		\node at (1.5,0.5) {$\ldots$};
		\node at (2.5,0.5) {$Q$};
		\draw[<-] (0,-0.2) -- node [below] {$p$} (3,-0.2);
	\end{tikzpicture}
\end{align}
These analytical limits are not proved in the present paper: we only use them to guess the identities to use. Under these interpretations, we expect:
\begin{align}\label{eq:swcorner:hypotheticlimit}
	\ov{Q}^{[\infty_\West,\infty_{\South\North}]}
	\overset{\mathbf{?}}{=}
	"\lim_{p,q\to\infty}" \begin{tikzpicture}[guillpart,yscale=0.9,xscale=1.]
		\fill[guillfill] (0,0) rectangle (3,3);
		\draw[guillsep] (0,0)--(3,0)--(3,3)--(0,3)--(0,0) (0,1)--(3,1) (0,2)--(3,2) (1,0)--(1,3) (2,0)--(2,3);
		\node at (0.5,0.5) {$Q$};
		\node at (0.5,1.5) {$\vdots $};
		\node at (0.5,2.5) {$Q$};
		\node at (1.5,0.5) {$\ldots$};
		\node at (1.5,1.5) {$\ddots $};
		\node at (1.5,2.5) {$\ldots$};
		\node at (2.5,0.5) {$Q$};
		\node at (2.5,1.5) {$\vdots $};
		\node at (2.5,2.5) {$Q$};
		\draw[<->] (3.2,0.) -- node [right] {$q$} (3.2,3.);
		\draw[<-] (0,-0.2) -- node [below] {$p$} (3,-0.2);
	\end{tikzpicture}
\end{align}
The right big guillotine partition can be organized in two ways using square associativity provided the conjectural identity:
\begin{align*} \lim_{p\to\infty} 
	\begin{tikzpicture}[guillpart,yscale=1.5,xscale=1.5]
		\fill[guillfill] (0,0) rectangle (3,1);
		\draw[guillsep] (0,0)--(0,1) (1,0)--(1,1) (2,0)--(2,1) (3,0)--(3,1);
		\node at (0.5,0.5) {$\ca{S}^Q_{\South\North}$};
		\node at (1.5,0.5) {$\ldots$};
		\node at (2.5,0.5) {$\ca{S}^Q_{\South\North}$};
		\draw[<-] (0,-0.2) -- node [below] {$p$} (3,-0.2);
	\end{tikzpicture} 
	\overset{\mathbf{?}}{=} \lim_{p\to\infty} \begin{tikzpicture}[guillpart,yscale=1.1,xscale=1.5]
		\fill[guillfill] (0,0) rectangle (2,3);
		\draw[guillsep] (0,0)--(2,0)--(2,3)--(0,3) (0,1)--(2,1) (0,2)--(2,2);
		\node at (1,0.5) {$\ov{Q}^{[\infty_\West,1]}$};
		\node at (1,1.5) {$\vdots$};
		\node at (1,2.5) {$\ov{Q}^{[\infty_\West,1]}$};	
		\draw[<->] (2.2,0.) -- node [right] {$q$} (2.2,3.);
	\end{tikzpicture}
\end{align*}
On the left part, this corresponds to an horizontal 1D dynamics in $\ca{W}_{\South\North}$ and, in the Fourier space, $\ov{Q}^{[\infty_\West,\infty_{\South\North}]}$ is given by the 1D fixed point $G^*_{\infty_\West}$. The right part correspond to a vertical 1D dynamics on $\ca{W}_\West$ which can be studied by Fourier transform and integrating out mode by mode. This is formalized in the following key lemma, used at the end of the previous proof.

\begin{lemm}[transverse representation of half-plane fixed points]\label{lemm:Graal}
	The half-plane element $	G^*_{\infty_\West}(u)$ is equal, for all $u\in S^1$, to
	\begin{equation}\label{eq:graal:otherrep}
		\begin{split}
			G^*_{\infty_\West}(u)
			&= 	
			\ov{Q}^{[\infty_\West,1]}_{\East\East} 
			-
			\left(
			\ov{Q}^{[\infty_\West,1]}_{\East\South} 
			+
			\ov{Q}^{[\infty_\West,1]}_{\East\North} u
			\right)
			\phi_{\patterntype{hs}_\West}(u)^{-1}
			\left(
			\ov{Q}^{[\infty_\West,1]}_{\South\East} 
			+
			\ov{Q}^{[\infty_\West,1]}_{\North\East} u^{-1}
			\right)
			\\
			&= \mathrm{Schur}_{\ca{W}_\West}\begin{pmatrix}
				\ov{Q}^{[\infty_\West,1]}_{\East\East} & \ov{Q}^{[\infty_\West,1]}_{\East\South} 
				+
				\ov{Q}^{[\infty_\West,1]}_{\East\North} u
				\\
				\ov{Q}^{[\infty_\West,1]}_{\South\East} 
				+
				\ov{Q}^{[\infty_\West,1]}_{\North\East} u^{-1}& 	\phi_{\patterntype{hs}_\West}(u)^{-1}
			\end{pmatrix}
		\end{split}
	\end{equation}
	with the operator $\phi_{\patterntype{hs}_\West}(u)$ defined by
	\[
	\phi_{\patterntype{hs}_\West}(u)^{-1} = \ov{Q}^{[\infty_\West,1]}_{\North\North} +  \ov{Q}^{[\infty_\West,1]}_{\South\South} + \ov{Q}^{[\infty_\West,1]}_{\South\North} u + \ov{Q}^{[\infty_\West,1]}_{\North\South}u^{-1}
	\]
	obtained from the half-strip fixed point element $\ov{Q}^{[\infty_\West,1]}$.
\end{lemm}
\begin{proof}
	A very long proof may consist in writing all the operators in terms of powers of $W^{L|R}_K$ operators and then use Fourier coefficients representation of these operators but it totally hides the nature of the nature. We first consider
	\begin{align*}
		\begin{tikzpicture}[guillpart,yscale=0.5,xscale=2.]
			\fill[guillfill] (0,0) rectangle (3,3);
			\draw[guillsep] (2,0)--(2,3) (3,0)--(3,3);
			\node at (1,1.5) {r.h.s.\eqref{eq:graal:otherrep}};
			\node at (2.5,1.5) {$\ov{Q}^{\infty_{\South\North}}$};
		\end{tikzpicture}_{\ca{Q}}(u)
		& = 
		\Schur^{1D}\left(
		\mathrm{Schur}_{\ca{W}_\West}\Upsilon^v_{\ov{Q}^{[\infty_\West,1]}}(u),
		\mathrm{Schur}_{\ca{H}_1}\Psi^v_{Q}(u)
		\right)
	\end{align*}
	where $\Upsilon_{\ov{Q}^{[\infty_\West,1]}}(u)$ is the matrix in \eqref{eq:graal:otherrep} and $\Schur^{1D}$ is the one-dimensional Schur product \eqref{eq:Schur1D:def}. Using the associativity of the Schur complements, we then obtain, 
	\begin{align*}
		\begin{tikzpicture}[guillpart,yscale=0.5,xscale=2.]
			\fill[guillfill] (0,0) rectangle (3,3);
			\draw[guillsep] (2,0)--(2,3) (3,0)--(3,3);
			\node at (1,1.5) {r.h.s.\eqref{eq:graal:otherrep}};
			\node at (2.5,1.5) {$\ov{Q}^{\infty_{\South\North}}$};
		\end{tikzpicture}_{\ca{Q}}(u)
		&=	\mathrm{Schur}_{\ca{W}_\West\oplus\ca{H}_1} \circ\Upsilon^v_{Q'}(u)
	\end{align*}
	where $Q'$ is the following gluing of a face $Q$ on the West half-strip \emph{before the shift $D^1$}:
	\begin{align*}
		Q' &= \mathrm{Schur}_{\ca{H}_2^\mathrm{cut}}
		(j_{
			\begin{tikzpicture}[guillpart,yscale=0.4,xscale=0.4]
				\fill[guillfill] (0,0) rectangle (2,1);
				\draw[guillsep] (0,0)--(2,0)--(2,1)--(0,1) (1,0)--(1,1);
			\end{tikzpicture}
		}\left(\ov{Q}^{[\infty_\West,1]},Q\right)) 	&= D_1^L
		\begin{tikzpicture}[guillpart,yscale=1.5,xscale=1.5]
			\fill[guillfill] (0,0) rectangle (3,1);
			\draw[guillsep] (0,0)--(3,0)--(3,1)--(0,1) (2,0)--(2,1);
			\node at (1,0.5) {$\ov{Q}^{[\infty_\West,1]}$};
			\node at (2.5,0.5) {$Q$};
		\end{tikzpicture}_Q^{D^L}
		(D_1^L)^{-1} 
		\\
		&= D_1^L \begin{tikzpicture}[guillpart,yscale=1.5,xscale=1.5]
			\fill[guillfill] (0,0) rectangle (2,1);
			\draw[guillsep] (0,0)--(2,0)--(2,1)--(0,1) ;
			\node at (1,0.5) {$\ov{Q}^{[\infty_\West,1]}$};
		\end{tikzpicture}_Q
		(D_1^L)^{-1}
	\end{align*}
	A direct computation shows that the conjugation by $D_1^L$ commutes with $\Upsilon^v_\bullet(u)$ and intertwines $\mathrm{Schur}_{\ca{W}_{\West}\oplus\ca{H}_1}$ with $\mathrm{Schur}_{\ca{W}_{\West}}$, so that we have
	\[
	\begin{tikzpicture}[guillpart,yscale=0.4,xscale=2.]
		\fill[guillfill] (0,0) rectangle (3,3);
		\draw[guillsep] (2,0)--(2,3) (3,0)--(3,3);
		\node at (1,1.5) {r.h.s.\eqref{eq:graal:otherrep}};
		\node at (2.5,1.5) {$\ov{Q}^{\infty_{\South\North}}$};
	\end{tikzpicture}_{\ca{Q}}(u) 
	= \mathrm{r.h.s.~\eqref{eq:graal:otherrep}}
	\]
	For each $u\in S^1$, the unicity of the fixed point under assumption~\ref{assump:2Deasy} provides that the r.h.s.~of \eqref{eq:graal:otherrep} is then equal to $\ov{Q}^{\infty_{W},\infty_{\South\North}}(u)$.
\end{proof}

Square associativity has a second incarnation in the following adjunction lemma.

\begin{lemm}\label{lemm:fixed:VadjointOK}
	The two operators $V^{\South\West}_{\nwarrow}$ and $V^{\South\West}_{\searrow}$ defined in \eqref{eq:def:cornerV} are bounded and adjoint to each other.
\end{lemm}

\begin{proof}
	All the operators $W^{L|R}_\bullet(u)$ have an operator norm strictly smaller than one and folding is a morphism so that all the geometric sequences are convergent in the operator norm and the two operators are bounded.
	
	In order to find the correct commutative diagram, we first need to understand the structure of the formulae for $V^{\South\West}_{\searrow}$ and $V^{\South\West}_{\nwarrow}$ in terms of infinite gluing of half-strips in one direction or another:
	\begin{equation}\label{eq:adjunction:squareassoc}
		\begin{tikzpicture}[guillpart,yscale=1.15,xscale=1.5]
			\fill[guillfill] (0,0) rectangle (4,4);
			\draw[guillsep] (4,0)--(4,4)--(0,4) (0,1)--(4,1) (0,2)--(4,2) (0,3)--(4,3);
			\node at (2,0.5) {$\vdots$};
			\node at (2,1.5) {$\ov{Q}^{[\infty_\West,1]}$};
			\node at (2,2.5) {$\ov{Q}^{[\infty_\West,1]}$};
			\node at (2,3.5) {$\ov{Q}^{[\infty_\West,1]}$};
			\draw[ultra thick,red,->] (3,4)--(3,3);
			\draw[ultra thick,red,->] (3,3)--(3,2);
			\draw[ultra thick,red,->] (3,2)--(3,1);
			\draw[ultra thick,violet,->,dashed] (3,1)--(4,0.5);
			\draw[ultra thick,violet,->,dashed] (3,1)--(4,1.5);
		\end{tikzpicture}
		\quad\text{ and }\quad
		\begin{tikzpicture}[guillpart,yscale=0.8,xscale=2.5]
			\fill[guillfill] (0,0) rectangle (4,4);
			\draw[guillsep] (4,0)--(4,4)--(0,4) (1,0)--(1,4) (2,0)--(2,4) (3,0)--(3,4);
			\node at (0.5,2) {$\dots$};
			\node at (1.5,2) {$\ov{Q}^{[1,\infty_\South]}$};
			\node at (2.5,2) {$\ov{Q}^{[1,\infty_\South]}$};
			\node at (3.5,2) {$\ov{Q}^{[1,\infty_\South]}$};
			\draw[ultra thick,red,->] (4,3)--(3,3);
			\draw[ultra thick,red,->] (3,3)--(2,3);
			\draw[ultra thick,red,->] (2,3)--(1,3);
			\draw[ultra thick,violet,->,dashed] (1,3)--(0.5,4);
			\draw[ultra thick,violet,->,dashed] (1,3)--(1.5,4);
		\end{tikzpicture}
	\end{equation}
	where the red full arrows correspond respectively to the successive powers of $\fold^{\West}(\ha{W}_{\South})$ and $\fold^{\South}(\ha{W}_\West)$ and the violet dashed arrows to the operators $(U_Q^\West)^*$ and $(U_Q^{\South})^*$ respectively. The horizontal and vertical half-strips can be interpreted as in \eqref{eq:westhalfplane:hypotheticallimit} and we thus expect both operators to be adjoint.
	
	Many rigorous proofs along the same line can be written. For example, one can write, for any $k,l>0$,
	\begin{align*}
		\scal{v_\South\indic{-k-1}}{ V^{\South\West}_{\searrow} v_\West \indic{-l-1}}  = \sum_{m\geq 0}\int_{S^1} f_{k,l,m}(\theta,v_\South,v_\West)\frac{d\theta}{2\pi}
	\end{align*}
	with
	\begin{align*}
		&f_{k,l,m}(\theta,v_\South,v_\West)	=
		\\
		& \scal{ U(\id+W^L_{Q,\West\East})(W^L_{Q,\West\East})^m v_\South}{
			(\id+\ha{W}_\South(\theta)) \ha{W}_\South(\theta)^k v_\West (e^{i(m-l)\theta}-e^{i(l+m+1)\theta})}
	\end{align*}	
	and a similar expression for $\scal{ V^{\South\West}_{\nwarrow},v_\South\indic{-k-1}}{ v_\West \indic{-l-1}}$.
	Using expressions of all the $W$ operators in terms of Fourier coefficients \eqref{eq:1D:link:WwithFourier} of $\phi$ functions and Schur complements, a lengthy computations provides the result.
	
	Another proof consist in studying the full quadratic form on the lattice of \eqref{eq:swcorner:hypotheticlimit} and minimizing it with the boundary conditions on North and East: the quadratic form on the minimizing harmonic function is equal to the corner quadratic form. By dividing the domain into horizontal or vertical strips as in \eqref{eq:adjunction:squareassoc}, the coupling between both sides can be shown to equal to $V^{\South\West}_{\nearrow}$ or $(V^{\South\West}_{\nearrow})^*$. The result can also be obtained in various way through representations of quadratic forms by Green functions of random walks cut along the various axis of \eqref{eq:adjunction:squareassoc} or even by a careful treatment of the infinite size limits in \eqref{eq:adjunction:squareassoc} through \cite{Bodiot}. In all the cases, the computations is related to the geometry of square associativity.
\end{proof}

\section{Complementary results}\label{sec:complementary}
\subsection{Quick reminder on Schur complements}\label{sec:complementary:Schur_Complements}
Most algebraic operations in the present paper rely on the Schur complement of lemma~\ref{lemm:Schur}. We now summarize quickly most useful properties of the Schur complement, see \cite{zhang2006schur} for more details.

Given an operator $M$ with the following block structure on the decomposition $\ca{H}=\oplus_{1\leq j\leq r}\ca{H}_j$
\[
M = \begin{pmatrix}
	M_{11}	&	M_{12} & \ldots & M_{1r}	\\
	M_{21}  &  M_{22} & \ldots & M_{2r}	\\
	\vdots  & \vdots & \ddots & \vdots \\
	M_{r1}  & M_{r2} & \ldots & M_{rr}
\end{pmatrix},
\]
the $k$-th Schur complement $\mathrm{Schur}_{\ca{H}_k}(M)$ is defined whenever $M_{kk}$ is invertible and is an operator acting on  $\ca{H}'=\oplus_{j\neq k} \ca{H}_j$ with $r-1$ blocks indexed by $\{1,\ldots,r\}\setminus\{k\}$ given by
\begin{equation}\label{eq:Schur_compl}
	\left(\mathrm{Schur}_{\ca{H}_k}(M)\right)_{ij} =  M_{ij}-M_{ik}M_{kk}^{-1}M_{kj}
\end{equation}
Without loss of generality, the blocks can be permuted so that the removed block is the last one. In this case, a unique LDU decomposition provides
\begin{equation}\label{eq:Schur:LDU}
	\begin{pmatrix}
		M_{\bullet} &  M_{\bullet k} \\
		M_{k\bullet} & M_{kk} 
	\end{pmatrix}
	=
	\begin{pmatrix}
		I &  M_{\bullet k}M_{kk}^{-1}
		\\
		0 & I
	\end{pmatrix}
	\begin{pmatrix}
		\mathrm{Schur}_{\ca{H}_k}(M) & 0 \\
		0 & M_{kk}
	\end{pmatrix}
	\begin{pmatrix}
		I & 0  \\
		M_{kk}^{-1}M_{k\bullet} & I
	\end{pmatrix}
\end{equation}
from which most results can be are derived easily. The inverse of $M$ can be written easily and one also has \begin{equation}\label{eq:Schur:det}
	\det(M) = \det(M_{kk})\det\mathrm{Schur}_{\ca{H}_k}(M)
\end{equation}
From the LDU decomposition, one also obtains that, $M$ is positive definite if and only if $M_{kk}$ and $\mathrm{Schur}_{\ca{H}_k}(M)$ are both positive definite. Moreover, we also have the following associativity property used all through the paper:
\begin{equation}\label{eq:Schurcomp:assoc}
	\mathrm{Schur}_{\ca{H}_l}\left(	\mathrm{Schur}_{\ca{H}_k} (M) \right)
	=
	\mathrm{Schur}_{\ca{H}_k}\left(	\mathrm{Schur}_{\ca{H}_l} (M) \right)
	=
	\mathrm{Schur}_{\ca{H}_l\oplus\ca{H}_k} (M)
\end{equation}
as soon as the Schur complements are well defined. In Gaussian integration, it corresponds to the successive marginalizations of a joint law and Fubini's theorem.

\subsection{Reminder on the one-dimensional Gaussian Markov processes}\label{sec:dimone:reminder}

We consider the following one-dimensional graph $\setZ$ with vertices $V_n=\{0,1,\ldots,n\}$ and edges $E_n=\{(k-1,k); 0< k \leq n\}$. Given self-adjoint positive definite  operators $(Q_{n-1,n})_{N_1 \leq n\leq N_2}$ on a Hilbert space $\ca{H}_2$ and $G_{N_1}$ and $G_{N_2}$ on $\ca{H}$, we can consider Gaussian $\ca{H}$-valued random variables $(X_{n})_{N_1\leq n\leq N_2}$ on the vertices with a  joint law has a density w.r.t. the Lebesgue measure on $\ca{H}^{n+1}$ given by:
\begin{equation}\label{eq:def:1Dproc}
	\exp\left( -\frac{1}{2} \sum_{N_1\leq k\leq N_2} \begin{pmatrix}
		x_{k-1} \\ x_k
	\end{pmatrix}^* Q_{k-1,k} \begin{pmatrix}
		x_{k-1} \\ x_k
	\end{pmatrix} -\frac{1}{2} \sum_{i\in\{N_1,N_2\}}x_{N_i}^* G_{N_i} x_{N_i} 
	\right)
\end{equation}
A straightforward consequence of Gaussian integration (see lemma~\ref{lemm:Schur}) is that the marginal law of $(x_i)_{i\in I}$ with $I\subset [N_1,N_2]$ has a density given by Schur complements of the previous global quadratic form on $\ca{H}^{N_2-N_1+1}$. As a one-dimensional Markov process, there is an associative product corresponding the marginalization w.r.t. one variable. 

\begin{prop}\label{prop:def:Schur1D}
	Let $A$ and $B$ be self-adjoint positive definite operators on $\ca{H}^2$ respectively with block structures
	\begin{align*}
		A &= \begin{pmatrix}
			A_{11} & A_{12} \\
			A_{21} & A_{22} 
		\end{pmatrix}
		&
		B &= \begin{pmatrix}
			B_{11} & B_{12} \\
			B_{21} & B_{22} 
		\end{pmatrix}
	\end{align*}
	then the operator $j_{1D}(A,B)$ on $\ca{H}^3$ defined by
	\begin{equation}
		\label{eq:def:joinoperator}
		j_{1D}(A,B) = \begin{pmatrix}
			A_{11} & A_{12} & 0 \\
			A_{21} & A_{22}+B_{11} & B_{12} \\
			0 & B_{21} & B_{22}
		\end{pmatrix} 
	\end{equation}
	is again self-adjoint positive definite and the Schur product defined by the Schur complement with respect to the second copy of $\ca{H}$
	\begin{equation}\label{eq:Schur1D:def}
		\Schur_{1D}(A,B) = \mathrm{Schur}_{\ca{H}_2}( j_{1D}(A,B) )
		= \begin{pmatrix}
			A_{11}-A_{12}T^{-1}A_{21} & -A_{12}T^{-1}B_{12} \\
			-B_{21}T^{-1}A_{21} & B_{22}-B_{21}T^{-1}B_{12}
		\end{pmatrix},
	\end{equation}
	with $T=A_{22}+B_{11}$,
	is a well-defined self-adjoint positive definite operator on $\ca{H}^2$ and is associative, i.e. 
	\[
	\Schur_{1D}( \Schur_{1D}(A_1,A_2), A_3) = \Schur_{1D}( A_1,\Schur_{1D}(A_2,A_3))
	\]
\end{prop}
This proposition is a direct consequence of the previous properties. In order to deal with boundaries, we also introduce the two following left and right action.
\begin{prop}\label{prop:Schur1D:actions}
	Let $Q$ and $G$ be self-adjoint positive definite quadratic forms respectively on $\ca{H}^2$ and $\ca{H}$. We define $\Schur_{1D}^{(L)}(G,Q)$ and $\Schur_{1D}^{(R)}(Q,G)$ by the Schur complements
	\begin{align*}
		\Schur_{1D}^{(L)}(G,Q)	&= \Big(Q+(G,0) / (Q+(G,0))_{11}\Big)
		\\
		\Schur_{1D}^{(R)}(Q,G)	&=\Big(Q+(0,G) / (Q+(0,G)_{22})\Big)
	\end{align*}
	which are again self-adjoint positive definite quadratic forms on $\ca{H}$. Moreover, we have the following action property:
	\begin{subequations}
		\label{eq:Schur:leftrightmodules}
		\begin{align}
			\Schur_{1D}^{(L)}(\Schur_{1D}^{(L)}(G,Q_1),Q_2) &= \Schur_{1D}^{(L)}(G,\Schur_{1D}(Q_1,Q_2))
			\\
			\Schur_{1D}^{(R)}(Q_1,\Schur_{1D}^{(R)}(Q_2,G)) &= \Schur_{1D}^{(R)}(\Schur_{1D}(Q_1,Q_2),G)
		\end{align}
	\end{subequations}
\end{prop}
Then, for any $N_1\leq M_1\leq M_2\leq N_2$, the marginal law of the Gaussian process on $\{M_1,\ldots,M_2\}$ is given by the subset of matrices $Q_{i-1,i}$, $M_1<i\leq M_2$ and boundary elements \begin{align*}
	G'_{M_1} &= \Schur_{1D}^{(L)}(G_{N_1},\Schur_{1D}(Q_{N_1,N_1+1},\ldots Q_{M_1-1,M_1})) \\
	G'_{M_2} &= \Schur_{1D}^{(R)}(\Schur_{1D}(Q_{M_2,M_2+1},\ldots Q_{N_2-1,N_2}),G_{N_2}) 
\end{align*}

When all operators $Q_{k,k+1}$ are equal to some operator $Q$, Perron-Frobenius eigenvectors $\ee_{G_L}$ and $\ee_{G_R}$, needed to define a infinite-volume Gibbs measure on $\setZ$ by Kolmogorov's extension theorem are then determined by
\begin{align}\label{eq:dimone:Schureigen}
	G^*_L &= \Schur_{1D}^{(L)}(G^*_L,Q) & G^*_R&= \Schur_{1D}^{(R)}(Q,G^*_R)
\end{align}
and are studied in detail in \cite{Bodiot}. 

From $Q$ alone, we may also define a Gibbs measure on $\setZ$ without using the Markov property but only the Fourier transform, as for the two-dimensional process. As a generalization of \cite{georgii} to the vector case, we may introduce the function $\phi_K: \setC^*\to \End(\ca{H})$, 
$
\phi_Q(z) = Q_{LL}+Q_{RR} + Q_{LR}z+Q_{RL}z^{-1}
$
(the function $J$ in chapter~13 of \cite{georgii}), the set of zeroes
\[
\ca{S}^{1D}_K = \{ z\in \setC^*; \det\phi_K(z)=0 \}
\]
Equations in $G^*_L$ and $G^*_R$ are non-linear due to the non-linearty of Schur complements: these difficulties, combined with the need of a $u$-dependent description in proposition \ref{prop:transfer1D:fourier} in the two-dimensional case, require a  parametrization of these operators in terms of more fundamental objects, which can be generalized to the two-dimensional situation. Most of the results in \cite{Bodiot} about these parametrizations are stated under the following assumptions.
\begin{assumptionp}{(simple1D)}
	\label{assump:1D}
	The edge operator $K$ is such that the multiplicity of each non-zero roots $z\in\ca{S}^{1D}_K$ for the polynomial $\det( z\phi_K(z) )$ is equal to $1$ and $\dim\ker \phi_K(z)=1$.
\end{assumptionp}
If they are not satisfied, then formulae have to be modified in order to incorporate higher multiplicity and degeneracy in the poles if $\phi_K(z)^{-1}$: computations are still feasible in this non-generic case but this is not the purpose of the present paper. Under assumption \ref{assump:1D}, we introduce the unique operators $W^{L}_K$ and $W^{R}_K$ on $\ca{H}$ defined by
\begin{subequations}
	\label{eq:dimone:def:W}
	\begin{align}
		W^L_K v &= \begin{cases}
			(1/z) v & \text{for all $|z|>1$ and $v\in\ker\phi_Q(z)$} \\
			0 & \text{for $v\in \ker K_{LR}$}
		\end{cases}
		\\
		W^R_K v &= \begin{cases}
			z v & \text{for all $|z|<1$ and $v\in\ker\phi_Q(z)$} \\
			0 & \text{for $v\in \ker K_{RL}$}
		\end{cases}
	\end{align}
\end{subequations}

\begin{lemm}[from \cite{Bodiot}]\label{lemm:onedim:invariantfromW}
	Under assumption~\ref{assump:1D}, the unique operators $G^*_L$ and $G^*_R$ solutions of \eqref{eq:dimone:Schureigen} are given by
	\begin{align*}
		G^*_L &= Q_{RR}+Q_{RL} W^{L}_Q
		&
		G^*_R &= Q_{LL}+Q_{LR} W^{R}_Q
	\end{align*}
\end{lemm}
We also check that the sum of the two operators are related to the zero coefficient of the Fourier transform
\begin{equation}\label{eq:1D_somme_G_L_G_R}
	G^*_L+G^*_R = \left({\Fourier}_0(\phi_Q(\cdot)^{-1}) \right)^{-1} 
\end{equation}
and the relations between the operator $W^{L|R}_Q$ and the Fourier transform are deeper. Using $f_n=\Fourier_n(\phi_Q^{-1})$, it is indeed shown in \cite{Bodiot} that, for any $k\geq 0$,
\begin{align}
	\label{eq:1D:link:WwithFourier}
	(W_Q^L)^k &=  f_k f_0^{-1} 
	&
	(W_Q^R)^k &=  f_{-k} f_0^{-1}
\end{align}
As shown in \cite{Bodiot}, it is shown that, for any $k,l\in\setZ$ with $kl\geq 0$ (same sign), one has
\begin{align}
	\label{eq:1D:FourierPhiInv:remarkable}
	f_k f_0^{-1} f_l = f_{k+l}
\end{align}
We also have the following recursive relation:
\begin{equation}\label{eq:phiK:recursion:FourierCoeffs}
	(Q_{LL} + Q_{RR}) f_n + Q_{LR} f_{n-1} + Q_{RL}f_{n+1}
	=\id \delta_{n,0}
\end{equation}
and an extensive use of these identities is made in the proof of theorems~\ref{theo:hs:fixedpoint:formulae} and \ref{theo:corner:fixedpoint:formulae}.

\subsection{Operadic structure of quadratic forms}\label{sec:moreonguill}

\subsubsection{Strip structure: pairings and shifts.}

The fixed point equations of section~\ref{sec:fixedpointeqs} require only the detailed structure of half-strips and corners. However, in order to take into account associativity diagrams of the type
\begin{align*}
	\begin{tikzpicture}[guillpart,yscale=1.1,xscale=1.3,baseline=0.25cm]
		\fill[guillfill] (0,0) rectangle (4,1);
		\draw[guillsep] (0,0)--(4,0) (0,1)--(4,1) (1.5,0)--(1.5,1) (2.75,0)--(2.75,1) ;
		\node at (0.75,0.5) {$1$};
		\node at (2.125,0.5) {$2$};
		\node at (3.375,0.5) {$3$};
		\node (P1) at (3.5,0) [circle, fill, inner sep=0.5mm] {};
		\node (P2) at (3.5,1) [circle, fill, inner sep=0.5mm] {};
		\draw[thick, dotted] (P1)--(P2);
		\draw[->] (2.75,1.2)-- node [midway, above] {$s$} (3.5,1.2);
	\end{tikzpicture}_\ca{Q}
	&= 
	\begin{tikzpicture}[guillpart,yscale=1.1,xscale=1.3,baseline=0.25cm]
		\fill[guillfill] (0,0) rectangle (4,1);
		\draw[guillsep] (0,0)--(4,0) (0,1)--(4,1) (2.75,0)--(2.75,1) ;
		\node at (1.375,0.5) {$1'$};
		\node at (3.375,0.5) {$2'$};
		\node (P1) at (3.5,0) [circle, fill, inner sep=0.5mm] {};
		\node (P2) at (3.5,1) [circle, fill, inner sep=0.5mm] {};
		\draw[thick, dotted] (P1)--(P2);
		\draw[->] (2.75,1.2)-- node [midway, above] {$s$} (3.5,1.2);
	\end{tikzpicture}_\ca{Q}\left(
	\begin{tikzpicture}[guillpart,yscale=1.1,xscale=1.3,baseline=0.25cm]
		\fill[guillfill] (0,0) rectangle (3,1);
		\draw[guillsep] (0,0)--(3,0)--(3,1)--(0,1) (2,0)--(2,1);
		\node at (1,0.5) {$1$};
		\node at (2.5,0.5) {$2$};
	\end{tikzpicture}_{\ca{Q}}
	,
	\id_3
	\right)
	\\
	&= 
	\begin{tikzpicture}[guillpart,yscale=1.1,xscale=1.3,baseline=0.25cm]
		\fill[guillfill] (0,0) rectangle (4,1);
		\draw[guillsep] (0,0)--(4,0) (0,1)--(4,1) (1.5,0)--(1.5,1)  ;
		\node at (0.75,0.5) {$1'$};
		\node at (2.25,0.5) {$2'$};
		\node (P1) at (3.5,0) [circle, fill, inner sep=0.5mm] {};
		\node (P2) at (3.5,1) [circle, fill, inner sep=0.5mm] {};
		\draw[thick, dotted] (P1)--(P2);
		\draw[->] (1.5,1.2)-- node [midway, above] {$p+s$} (3.5,1.2);
	\end{tikzpicture}_\ca{Q}\left(
	\id_1, 
	\begin{tikzpicture}[guillpart,yscale=1.1,xscale=1.3,baseline=0.25cm]
		\fill[guillfill] (0,0) rectangle (2.5,1);
		\draw[guillsep] (2.5,0)--(0,0)--(0,1)--(2.5,1) (1,0)--(1,1); 
		\node at (0.5,0.5) {$2$};
		\node at (2,0.5) {$3$};
	\end{tikzpicture}_{\ca{Q}}
	\right)
\end{align*}
designed in \cite{Simon}, we require the following minimal definition.
\begin{defi}[shift-with-pairing property]\label{def:shiftpairingprop}
	Let $\ca{H}$ be a Hilbert space. Let $\ca{W}_L$ and $\ca{W}_R$ be Hilbert spaces with, respectively, the left and the right $\ca{H}$-shift property. Let $\ca{W}_{LR}$ be a Hilbert space with an action $(\tau_p)$ of $\setZ$-translations. The triplet $(\ca{W}_L,\ca{W}_R,\ca{W}_{LR})$ satisfy the shift-with-pairing property if there exists an isometry $h:\ca{W}_L\oplus\ca{W}_R \to\ca{W}_{LR}$ such that
	\[
	\tau_s \circ h\left( D_p^L(w_L,h) , w_R \right) = \tau_{s+p} \circ h\left( w_L , D_{p}^R(h,w_R) \right)
	\]
\end{defi}
This definition ensures the necessary identifications in guillotine partitions as in  $\rho_0$ but with a infinite line in at least one dimension and the corresponding pointings. 

\subsubsection{Proof of the guillotine structure theorem}\label{sec:proofguillot}

\begin{proof}[Proof of theorem~\ref{theo:guillstruct}]
	From \cite{Simon}, it is enough to show the following steps:
	\begin{itemize}
		\item for any elementary guillotine partition of the type
		\begin{equation}\label{eq:typeguill}
			\begin{tikzpicture}[guillpart,yscale=1.5,xscale=1.]
				\fill[guillfill] (0,0) rectangle (3,2);
				\draw[guillsep,dashed] (0,0)--(3,0)--(3,2)--(0,2)--(0,0);
				\draw[guillsep] (0,1)--(3,1);
				\node at (1.5,0.5) {$Q_1$};
				\node at (1.5,1.5) {$Q_2$};
			\end{tikzpicture}
			\qquad\text{or}\qquad
			\begin{tikzpicture}[guillpart,yscale=1,xscale=1.5,rotate=-90]
				\fill[guillfill] (0,0) rectangle (3,2);
				\draw[guillsep,dashed] (0,0)--(3,0)--(3,2)--(0,2)--(0,0);
				\draw[guillsep] (0,1)--(3,1);
				\node at (1.5,0.5) {$Q_1$};
				\node at (1.5,1.5) {$Q_2$};
			\end{tikzpicture}
		\end{equation}
		where each of the four boundary sides is full or absent depending on the type of boundary or rectangles, the product is well defined between suitable spaces $\ca{Q}_{p_1,q_1}$ and $\ca{Q}_{p_2,q_2}$ with compatible sizes and takes values in the suitable $\ca{Q}_{p',q'}$. It is enough to show that the Schur complements are well-defined under the boundedness assumptions on the quadratic forms and satisfy themselves the same type of boundedness assumptions. This is done in lemma~\ref{lemm:Schurbounds}
		
		\item the three fundamental associavities ---horizontal, vertical and square--- are satisfied: this is direct consequence of \eqref{eq:Schurcomp:assoc} since the three-part or four-part guillotine partitions of associativies corresponds to the same Schur complements taken in different orders. 
	\end{itemize}
	
	We now proove the key analytic lemma for the first step. For any bounded operator $K$ on a Hilbert space $\ca{H}$, we define
	\[
	m_K =\inf_{x\neq 0} \frac{\scal{x}{Kx}}{\scal{x}{x}}
	\]
	
	\begin{lemm}\label{lemm:Schurbounds}
		Let $A$ and $B$ be two operators respectively on $\ca{H}_1\oplus\ca{H}_2$ and $\ca{H}_2\oplus\ca{H}_3$ where the $\ca{H}_i$ are three Hilbert spaces such that $A$ and $B$ are bounded and satisfy $m_A>0$ and $m_B>0$. Then the operator $j(A,B)$ on the orthogonal sum $\ca{H}_1\oplus\ca{H}_2\oplus\ca{H}_3$ defined as \eqref{eq:def:joinoperator} is bounded and satisfy $m_{j(A,B)}> \min(m_A,m_B)$.
		Moreover, the Schur complement $C=\mathrm{Schur}_{\ca{H}_2}(j(A,B))$ is well-defined, bounded and satisfy again $m_C>0$.
	\end{lemm}
	\begin{proof}
		The bounds on $j(A,B)$ are easy from the definition of $j$ by writing $x=(x_1,x_2,x_3)$. By restriction, the block $J_2=A_{22}+B_{11}$ is also bounded and satisfy $m_{J_2}\geq m_J>0$ and is thus invertible: the Schur complement is well defined and we have:
		\[
		j(A,B) =  \begin{pmatrix} 
			I & L^* \\ 0 & I
		\end{pmatrix}\begin{pmatrix}
			C & 0 \\
			0 & J_2
		\end{pmatrix} \begin{pmatrix} 
			I &0 \\ L & I
		\end{pmatrix}
		\] with bounded operators $L$, $C$ and $J_2$. From the triangular structure, one obtains $\scal{x}{Cx}\geq \min(m_A,m_B)||(x,-Lx)||^2$ and $m_C>0$.
	\end{proof}
	For any guillotine partition of type \eqref{eq:typeguill}, we then consider $\ca{H}_1$ as the orthogonal sum of the spaces in the external boundaries of $Q_1$, $\ca{H}_2$ the space associated to the cut in the partition and $\ca{H}_3$, the orthogonal sum of the spaces in the external boundaries of $Q_2$ and then use this lemma to obtain the guillotine structure of all the the spaces $\ca{Q}_{a,b}$ introduced here.
\end{proof}

\subsection{Eigenvalues from renormalization}\label{sec:eigenval}

All the computations of Gaussian densities have been lifted, using proposition~\ref{prop:fromweightstoquadraticforms}, at the level of the parameter space $\ca{Q}_{p,q}$ using Schur complements in place of (measure-theoretical) Gaussian integration.

There is no reference measure on the spaces $\ca{W}_{a}$ and $\ca{W}_{ab}$ due to the infinite dimension and products cannot be defined simply by integration of densities on spaces of densities. However, at the level of quadratic forms, no problem occurs. In order to come back to probability laws at the end in order to have practical computations, we still have to extend our computations on quadratic forms in \emph{all} the spaces $\ca{Q}_{\bullet,\bullet}$ (boundaries included) to \emph{all} the spaces $\setR\times \ca{Q}_{\bullet,\bullet}$ where the first element encodes a suitable renormalization of "densities" using generalization of the cocycles $\gamma_{ab}$ in \eqref{eq:horizSchur} and \eqref{eq:vertSchur}.

\paragraph*{From cylinders to strips}  The interested reader may acquire an intuition of these cocycles for strips, half-strips and corners by computing by hand the cocycles as finite products in the context of the cylinders discussed in section~\ref{sec:cylindertostrip} and sending $p\to\infty$ and considering $p$-th roots of the cycles: all the finite products then becomes exponential of Riemann sums that converge to suitable eigenvalues such as \eqref{eq:eigenval:expression}. Detailed computation are presented in the thesis \cite{BodiotThesis}.

\paragraph*{A renormalization approach}

We will now see how to relate the eigenvalue $\Lambda$ of \eqref{eq:eigenval:expression} to the computation of the term
\[
\begin{tikzpicture}[guillpart,yscale=1.15,xscale=1.3]
	\fill[guillfill] (0,0) rectangle (3,1);
	\draw[guillsep] (0,0)--(3,0)--(3,1)--(0,1) (2,0)--(2,1);
	\node at (2.5,0.5) {$Q$};
	\node at (1,0.5) {$\ov{Q}^{[\infty_\West,1]}$};
\end{tikzpicture}
\]
with Schur complements that correspond to Gaussian integration.
The expression~\eqref{eq:eigenval:expression} of $\Lambda$ involves the determinant of the function $\Psi_Q$: using \eqref{eq:Schur:det}, this can be rewritten as
\[
\det \Psi_Q(u_1,u_2) = \det \phi_Q^{\West\East}(u_1) \det \mathrm{Schur}_{\ca{H}_2}(\Psi_Q(u_1,u_2))
\]
and thus the eigenvalue factorizes into $\Lambda = \Lambda^{1D}_{\West\East} \Lambda'$ where $\Lambda^{1D}_{\West\East}$ is the one-dimensional eigenvalue associated to $Q_{[\West\East],[\West\East]}$ and $\Lambda'$ is given by
\begin{equation}
	\label{eq:transverseeigenval}
	\Lambda' = (2\pi)^{d_1}\exp\left(-\frac{1}{4\pi^2} \int_{S_1\times S_1}
	\log \det \ca{S}(e^{i\theta_1},e^{i\theta_2}) d\theta_1 d\theta_2
	\right)
\end{equation}
with $\ca{S}(u)$ given by
\[
\ca{S}(u_1,u_2)  = \phi_Q^{\South\North}(u_2) -\Psi_Q(u_1,u_2)_{1,2}\phi_Q^{\West\East}(u_1)^{-1} \Psi_Q(u_1,u_2)_{2,1}
\]
The one-dimensional part $\Lambda^{1D}_{\West\East}$ is easily understood: it corresponds to the eigenvalue of the 1D process on vertical edges on the strips with zero boundary conditions on the horizontal edges and can be deduced from the 1D results of \cite{Bodiot}. In particular, it corresponds to the Gaussian integration over the vertical cut in the previous guillotine partition, which produces a normalization factor
\[
\Lambda^{1D}_{\West\East}=(2\pi)^{d_2} \det\left( \ov{Q}^{[\infty_\West,1]}_{\East\East} + Q_{\West\West}\right)^{-1} =\gamma_{\West\East}(\ov{Q}^{[\infty_\West,1]},Q)
\]
for the cocycle described in \eqref{eq:horizSchur}. 

The second part $\Lambda'$ is more involved and is related to the morphism property of the half-strip fixed point $\ov{Q}^{[\infty_\West,1]}$ under the shift maps $D_1^L$. To do so, we need to understand the renormalized mass carried by the half-strip. In the transverse direction, this half-strip element can be glued with a corner element along a half-line: if traditional Gaussian integration would hold, this would produce a factor $(2\pi)^{d_1\times \infty}\det( C^{\South\West}_{\North,\North}+\ov{Q}^{[\infty_\West,1]}_{\South,\South})$ but this expression is ill-defined since $\ca{W}_{\West}$ is infinite-dimensional. Renormalization is required as for the strips above. To do so, we way truncate the spaces $\ca{W}_{\West}$ at order $n$, i.e. consider the sequence of smaller spaces $l^2(\{-n,-n+1,\ldots,-1\};\ca{H}_1)\simeq \ca{H}_1^n$ and consider the corresponding truncation of
\[
\ca{R}_n(\ov{Q}^{[\infty_\West,1]}_{[\South\North],[\South\North]}) = \toep^{\West}_n(\ha{Q}^{\infty_{\West\East}})-\hank^{\West}_n(\ha{Q}^{\infty_{\West\East}})
\]
These truncated blocks define a sequence of vertical dynamics on the truncated half-line space $\ca{H}_1^n$ with its own $1D$ eigenvalue as studied in \cite{Bodiot} with the eigenvalue
\[
\Lambda^{1D}_{\ca{R}_n(\ov{Q}^{[\infty_\West,1]}_{[\South\North],[\South\North]})} = (2\pi)^{n d_1} \exp\left(
-\frac{1}{2\pi} \int_0^{2\pi} \log \det \phi^{1D}_{\ca{R}_n(\ov{Q}^{[\infty_\West,1]}_{[\South\North],[\South\North]})}(e^{i\theta}) d\theta
\right)
\]
Now, in the parameter space $\setR\times \ca{Q}_{\infty_\West,1}$, we the consider the renormalized element $(1/\Lambda^{1D}_{\ca{R}_n(\ov{Q}^{[\infty_\West,1]}_{[\South\North],[\South\North]})}, \ca{R}_n(\ov{Q}^{[\infty_\West,1]}))$
However, we have seen that elements $\ov{Q}^{[\infty_\West,1]}$ are fixed points up to shift morphisms $D_1^L$ but this shift exactly produce a shift of $1$ in the truncation order. Thus, a careful description of the renormalization of the masses along the sequence of subspaces $l^2(\{-n,-n+1,\ldots,-1\};\ca{H}_1)$ produces an additional factor
\begin{equation}\label{eq:lambdaasrenormalizedratio}
	\Lambda' = \lim_{n\to\infty} \frac{\Lambda^{1D}_{\ca{R}_{n+1}(\ov{Q}^{[\infty_\West,1]}_{[\South\North],[\South\North]})}}{\Lambda^{1D}_{\ca{R}_n(\ov{Q}^{[\infty_\West,1]}_{[\South\North],[\South\North]})}}
\end{equation}

We still have to show how this limit is related to the expected expression \eqref{eq:transverseeigenval}. To do so, we treat separately the Toeplitz and the Hankel part, in the following way,
\[
\phi^{1D}_{\ca{R}_n(\ov{Q}^{[\infty_\West,1]}_{[\South\North],[\South\North]})}(e^{i\theta}) 
= \phi^{1D}_{\toep^{\West}_n(\ha{Q}^{\infty_{\West\East}})}(e^{i\theta})
\left(
\id_{\ca{H}_1} - \phi^{1D}_{\toep^{\West}_n(\ha{Q}^{\infty_{\West\East}})}(e^{i\theta})^{-1}\phi^{1D}_{\hank^{\West}_n(\ha{Q}^{\infty_{\West\East}})}(e^{i\theta})
\right)
\]
which produces a factorization of the 1D eigenvalue into two terms
\[
\begin{split}
	\Lambda^{1D}_{\ca{R}_n(\ov{Q}^{[\infty_\West,1]}_{[\South\North],[\South\North]})} = &
	\Lambda^{1D}_{\toep^{\West}_n(\ha{Q}^{\infty_{\West\East}})}
	\\
	& \times e^{-\int_0^{2\pi} \log \det\left(
		\id_{\ca{H}_1} - \phi^{1D}_{\toep^{\West}_n(\ha{Q}^{\infty_{\West\East}})}(e^{i\theta})^{-1}\phi^{1D}_{\hank^{\West}_n(\ha{Q}^{\infty_{\West\East}})}(e^{i\theta})
		\right) \frac{d\theta}{2\pi}}
\end{split}
\]
The determinantal term in the exponential is a Fredholm determinant since the Hankel part can be seen as a Hilbert-Schmidt perturbation of the Toeplitz part (see the Fourier coefficients): we thus expect the exponential to converge to a finite contribution $\Lambda''$, which becomes irrelevant in the asymptotic ratio \eqref{eq:lambdaasrenormalizedratio}. From the definition of the operator $\ov{Q}^{\infty_{\West\East}}$, we now identify directly
\[
\phi^{1D}_{\toep^{\West}_n(\ha{Q}^{\infty_{\West\East}})}(e^{i\theta}) = \toep^\West_{n} \left( \mathrm{Schur}_{\ca{H}_2}\Psi_Q(\bullet,e^{i\theta}) \right)
\]
where the Toeplitz is applied to the $\bullet$ part of the function at fixed $e^{i\theta}$. Using now Szeg\"o's limit formula \cite{szego1915grenzwertsatz, bottcher2013analysis}, we now obtain 
\[
\lim_{n\to\infty} \tfrac{\det \toep^\West_{n+1} \left( \mathrm{Schur}_{\ca{H}_2}\Psi_Q(\bullet,e^{i\theta}) \right)}{\det \toep^\West_{n} \left( \mathrm{Schur}_{\ca{H}_2}\Psi_Q(\bullet,e^{i\theta}) \right)}
= e^{
	\int_0^{2\pi}
	\log\det\left(
	\mathrm{Schur}_{\ca{H}_2}\Psi_Q(e^{i\phi},e^{i\theta})
	\right) \frac{d\phi}{2\pi}
}
\]
so that the ratio \eqref{eq:lambdaasrenormalizedratio} converges to the expression \eqref{eq:transverseeigenval} of $\Lambda'$.


\end{document}